\documentclass[11pt]{article}
\usepackage{amsfonts,amsmath}
\usepackage{latexsym}
\usepackage{amsthm}
\usepackage{amscd}
\usepackage{commath}
\usepackage{epsfig}
\usepackage{amssymb}
\usepackage{caption}
\usepackage{enumitem}
\usepackage{mathtools}
\usepackage{booktabs}
\usepackage[table]{xcolor}
\usepackage[toc,page]{appendix}
\usepackage[nottoc]{tocbibind}
\usepackage{array}
\newcolumntype{P}[1]{>{\centering\arraybackslash}p{#1}}

\usepackage[english]{babel}

\usepackage[letterpaper,top=2cm,bottom=2cm,left=2.5cm,right=2.5cm,marginparwidth=1.5cm]{geometry}

\usepackage[colorlinks=true, allcolors=blue,bookmarksnumbered]{hyperref}
\hypersetup{
	citecolor = red!80!black
}
\usepackage{color,soul}
\setstcolor{red}
\setul{}{.5ex}
\usepackage{float,subcaption}
\usepackage{graphicx}
\usepackage{stackengine}
\usepackage{tikz}
\usetikzlibrary{intersections}
\usepackage[export]{adjustbox}
\usepackage{array}

\usepackage[capitalise]{cleveref}
\usepackage{microtype}
\usepackage[percent]{overpic}

\usepackage{lineno}

\newtheorem{defn}{Definition}[section]
\newtheorem{thm}[defn]{Theorem}
\newtheorem{theorem}[defn]{Theorem}
\newtheorem{prop}[defn]{Proposition}
\newtheorem{lem}[defn]{Lemma}
\newtheorem{cor}[defn]{Corollary}
\theoremstyle{remark}
\newtheorem{remark}[defn]{Remark}
\newtheorem{question}[defn]{Question}
\newtheorem{conj}[defn]{Conjecture}

\numberwithin{equation}{section}
\numberwithin{figure}{section}

\newcommand{\bb}{\begin{equation}}
\newcommand{\ee}{\end{equation}}

\newcommand{\x}{x_i ^{(k)}}
\newcommand{\xj}{x_j ^{(k)}}

\newcommand{\half}{\frac{1}{2}}

\DeclareMathOperator{\arcsinh}{arcsinh}
 
\newcommand{\Hyp}[1]{{\mathbb{H}_{#1} }}

\newcommand{\origin}{\mathbf{o}}

\newcolumntype{?}{!{\vrule width 1.5pt }}
\newlength\savedwidth

\newcommand{\tightoverset}[2]{%
  \mathop{#2}\limits^{\vbox to -.5ex{\kern-0.75ex\hbox{$#1$}\vss}}}
  
\newcommand{\harp}[1]{\vec{#1}}

\newenvironment{proofof}[1]{{\medbreak\noindent \em Proof of #1.\enspace }}{\hfill\qed\medbreak}
\newcommand\sut{\,;\;}
\newcommand{\dfn}[1]{\textbf{\textit{#1}}}
\newcommand\mob{\textnormal{\textsf{M\"ob}}}  
\newcommand\HH{\mathbb{H}} 
\newcommand\uball{\mathbb{B}}  
\newcommand\uhs{\mathbb{U}}  
\newcommand\Unif{\mathrm{Unif}}  
\newcommand\corona{\widetilde{ \partial \HH_{d} }}

\newcommand\Leb{\mathrm{Leb}}
\newcommand\sph{\partial \HH_d}
\newcommand\Sph{\mathbb{S}}  
\newcommand{\R}{\mathbb{R}}
\newcommand{\PP}{\mathbb{P}}
\newcommand{\EE}{\mathbb{E}}
\newcommand{\TT}{\mathbb{T}}
\newcommand{\DD}{\mathbb{D}}
\newcommand{\ud}{\mathrm{d}}
\newcommand{\bud}{\mathbf{d}}  
\newcommand{\ue}{\mathrm{e}}
\newcommand{\ui}{\mathrm{i}}
\newcommand\borel{\mathcal{B}}
\newcommand\cells{\mathcal{V}}
\newcommand\cellsl{\mathcal{V}^{(\lambda)}}
\newcommand\dcells{\mathcal{D}}
\newcommand\ocell{\mathcal{C}}
\newcommand\tcell{\mathcal{T}}
\newcommand\scell{\mathcal{T}}
\newcommand\ST{\,;\;}
\newcommand\rf{\mathrm{rf}}  
\newcommand\Ste{\mathrm{Ste}}  
\newcommand\Exp{\mathrm{Exp}}  
\newcommand\Geom{\mathrm{Geom}}  
\newcommand\Pois{\mathrm{Pois}}  
\newcommand\Kh{\widehat K}  
\newcommand\bfz{\mathbf{0}}
\newcommand\RR{\mathbf{R}}  
\newcommand\XX{\mathbf{X}}  
\newcommand\XXl{\mathbf{X}^{(\lambda)}}  
\newcommand\YY{\mathbf{Y}}  
\newcommand\GKT{\mathcal{{L}}}  
\newcommand\II[1]{\mathbf{1}_{#1}}  
\newcommand\DelVol{\mathsf{DV}}  
\newcommand\IDV{\mathsf{IDV}}  
\newcommand\JJ{\mathsf{j}}  
\newcommand\cc{\mathsf{c}}  
\newcommand\fNS{f_{\mathrm{NS}}}  
\newcommand\fSV{f_{\mathrm{SV}}}  
\newcommand\fVN{f_{\mathrm{VN}}}  
\newcommand\Vor{\mathrm{Vor}}  
\newcommand\flr[1]{\lfloor #1 \rfloor}
\newcommand\Vol{\mathrm{Vol}}  
\newcommand\Intn{\mathrm{I}}  
\newcommand\Cen{\mathrm{Cen}}  
\newcommand{\EEtyp}{\mathbb{E}^{\mathrm{typ}}}  

\def\rlabel #1 #2{\begin{equation} \label{#1} #2 \end{equation}}

\def\rproof{\begin{proof}}

\def\Qed{\end{proof}}

\def\eqaln#1{\begin{align*} #1 \end{align*}}

\def\rcases#1{\begin{cases} #1 \end{cases}}

\makeatletter 
\tikzset{ 
reuse path/.code={\pgfsyssoftpath@setcurrentpath{#1}} 
} 
\tikzset{even odd clip/.code={\pgfseteorule}, 
protect/.code={ 
\clip[overlay,even odd clip,reuse path=#1] 
(current bounding box.south west) rectangle (current bounding box.north east)
; 
}} 
\makeatother 
\usetikzlibrary{3d,perspective,quotes,angles}

\title{\textsc{Ideal Poisson--Voronoi tessellations\\ on hyperbolic spaces}}
\author{
Matteo \textsc{D'Achille}\thanks{Universit\'e Paris-Saclay.\hfill  \href{mailto:nicolas.curien@gmail.com}{\texttt{matteo.dachille@universite-paris-saclay.fr}}},
Nicolas \textsc{Curien}\thanks{Universit\'e Paris-Saclay.\hfill  \href{mailto:nicolas.curien@gmail.com}{\texttt{nicolas.curien@gmail.com}}},
Nathana\"el \textsc{Enriquez}\thanks{Universit\'e Paris-Saclay and Ecole Normale Sup\'erieure.\hfill  \href{mailto:nathanael.enriquez@universite-paris-saclay.fr}{\texttt{nathanael.enriquez@universite-paris-saclay.fr}}},\\
Russell \textsc{Lyons}\thanks{Indiana University.\hfill  \href{mailto:rdlyons@indiana.edu}{\texttt{rdlyons@iu.edu}}},
and
Meltem \textsc{{\"U}nel}\thanks{Universit\'e Paris-Saclay.\hfill  \href{mailto:unel@lipn.univ-paris13.fr}{\texttt{unel@lipn.univ-paris13.fr}}} }
\date{19 March 2025}

\begin{document}
\maketitle

\begin{abstract} We study the limit in low intensity of Poisson--Voronoi tessellations in hyperbolic spaces $ \mathbb{H}_{d}$ for $d \geq 2$. In contrast to the Euclidean setting, a limiting nontrivial ideal tessellation $ \mathcal{V}_{d}$ appears as the intensity tends to $0$. The tessellation $  \mathcal{V}_{d}$ is  a natural, isometry-invariant decomposition of $ \mathbb{H}_{d}$ into countably many unbounded polytopes, each with a unique end.  We study its basic properties, in particular, the geometric features of its cells.\end{abstract}

\tableofcontents

\bigskip
\begin{quote}
[The Cheshire cat] vanished quite slowly, beginning with the end of the tail, and ending with the grin, which remained some time after the rest of it had gone.

``Well! I've often seen a cat without a grin,'' thought Alice; ``but a grin without a cat! It's the most curious thing I ever saw in my life!''

\smallskip
--- \textit{Alice's Adventures in Wonderland}, Lewis Carroll
\end{quote}

\section{Introduction}
Voronoi diagrams go back to Descartes. Their uses span the sciences, social sciences, and engineering.
Poisson--Voronoi tessellations are ubiquitous objects in stochastic geometry~{\cite[Chapter 4]{Moller1994}}. They have been used to model real-world networks (see~\cite[Section 19.3.2]{baccelli2010stochastic} for an introduction and \cite{bb6} for a concrete application) and have also been studied for their purely theoretical properties and their intrinsic beauty. 
The recent works \cite{bhupatiraju,BudzinskiCurienPetri} independently study properties of Poisson--Voronoi tessellations in hyperbolic space in the limit when the intensity of their nuclei tends to 0. Both these works also mention without proof the existence of a limiting tessellation. Here, we prove existence of this limit and study its fundamental properties. The existence of a nontrivial limit is surprising when one's intuition is grounded in Euclidean space.

\paragraph{Ideal Poisson--Voronoi tessellations.} Let $(E,d_{E},\mathbf{o},\mu)$ be an abstract, locally compact metric space equipped with an origin point, $\origin$, and an infinite Radon measure, $\mu$, such that the spheres centered at $\mathbf{o}$ have measure $0$. For $\lambda >0$, we consider a Poisson cloud of points $  \mathbf{X}^{(\lambda)}=(X^{{(\lambda)}}_1, X^{{(\lambda)}}_2, \ldots )$ with intensity $\lambda \cdot \mu$ (the points being ranked by their increasing distances to the origin of $E$). \ This point process enables us to define the \dfn{Voronoi diagram} $$ \mathrm{Vor}(\mathbf{X}^{(\lambda)})\coloneqq(C_1,C_2,\ldots)$$ relative to $\mathbf{X}^{(\lambda)}$, which is a tiling of $E$ where the tile (or cell) $C_{i}$ is made of the points of $E$ that are closer (in the weak sense) to $X_{i}^{(\lambda)}$ than to any other $X_{j}^{(\lambda)}$. When the underlying space $E$ is $\mathbb{R}^d$ equipped with Lebesgue measure (hence having polynomial volume growth), the diagrams $\mathrm{Vor}(\mathbf{X}^{(\lambda)})$ degenerate towards the trivial tiling $(E)$ when the intensity $\lambda$ tends to $0$.  However, if the underlying space has exponential growth, for example on $d$-regular metric trees with $d \geq 3$ equipped with the length measure, or on $d$-dimensional hyperbolic spaces $d \geq 2$ equipped with their volume measure, then $ \mathrm{Vor}(\mathbf{X}^{(\lambda)})$ may converge in distribution as $\lambda \to 0$ to a nontrivial random tiling, which we name the ideal Poisson--Voronoi tessellation. We refer to Section~\ref{sec:conv}, and in particular~\cref{lem.convvor}  for details of the convergence and to \cref{sec.basicfacts} for background on hyperbolic spaces.

\paragraph{Ideal Poisson--Voronoi tessellations on hyperbolic spaces $\mathbb{H}_d$.} In particular, our general convergence result, \cref{lem.convvor}, applies when $(E,d_{E},\origin, \mu)$ is the hyperbolic space $ \mathbb{H}_d$ equipped with its volume measure and yields:
\begin{theorem}[\textsc{Convergence of low-intensity tessellations}]\label{thm.convhyp} Let $ \mathbf{X}_{d}^{{(\lambda)}}$ be a Poisson point process with intensity $\lambda>0$ on the $d$-dimensional hyperbolic space $ \mathbb{H}_{d}$ with $d \geq 2$. Then we have the convergence in law 
$$ \mathrm{Vor}\big(\mathbf{X}_{d}^{{(\lambda)}}\big) \Rightarrow{} \mathcal{V}_{d}.$$
The limiting tessellation $\mathcal{V}_{d}$ is called the \dfn{ideal Poisson--Voronoi tessellation} of\/ $ \mathbb{H}_{d}$; see \cref{fig:voronoiH}.
\end{theorem}

\begin{figure}[!h]
 \begin{center}
 \includegraphics[height=5cm]{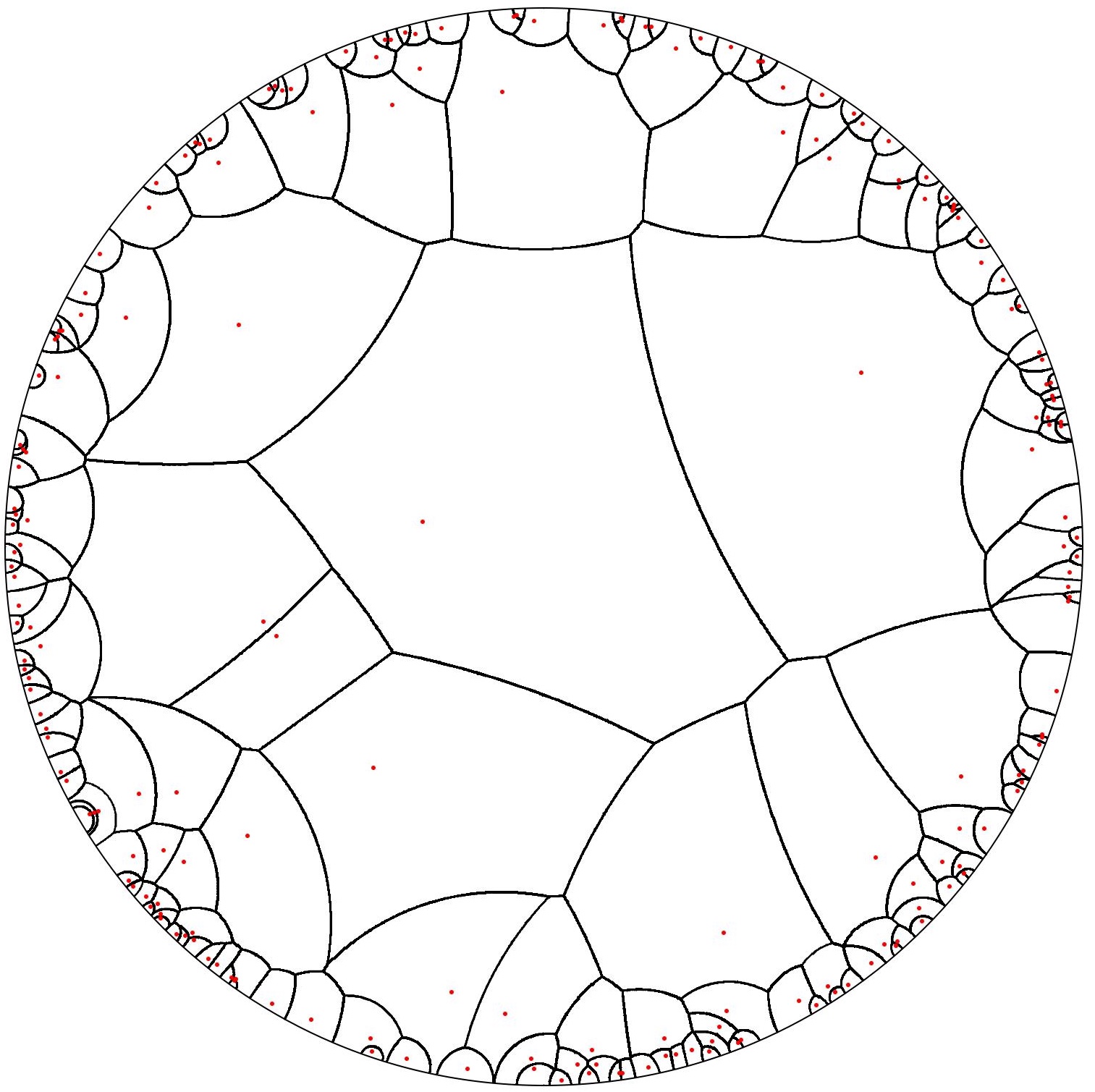}
  \includegraphics[height=5cm]{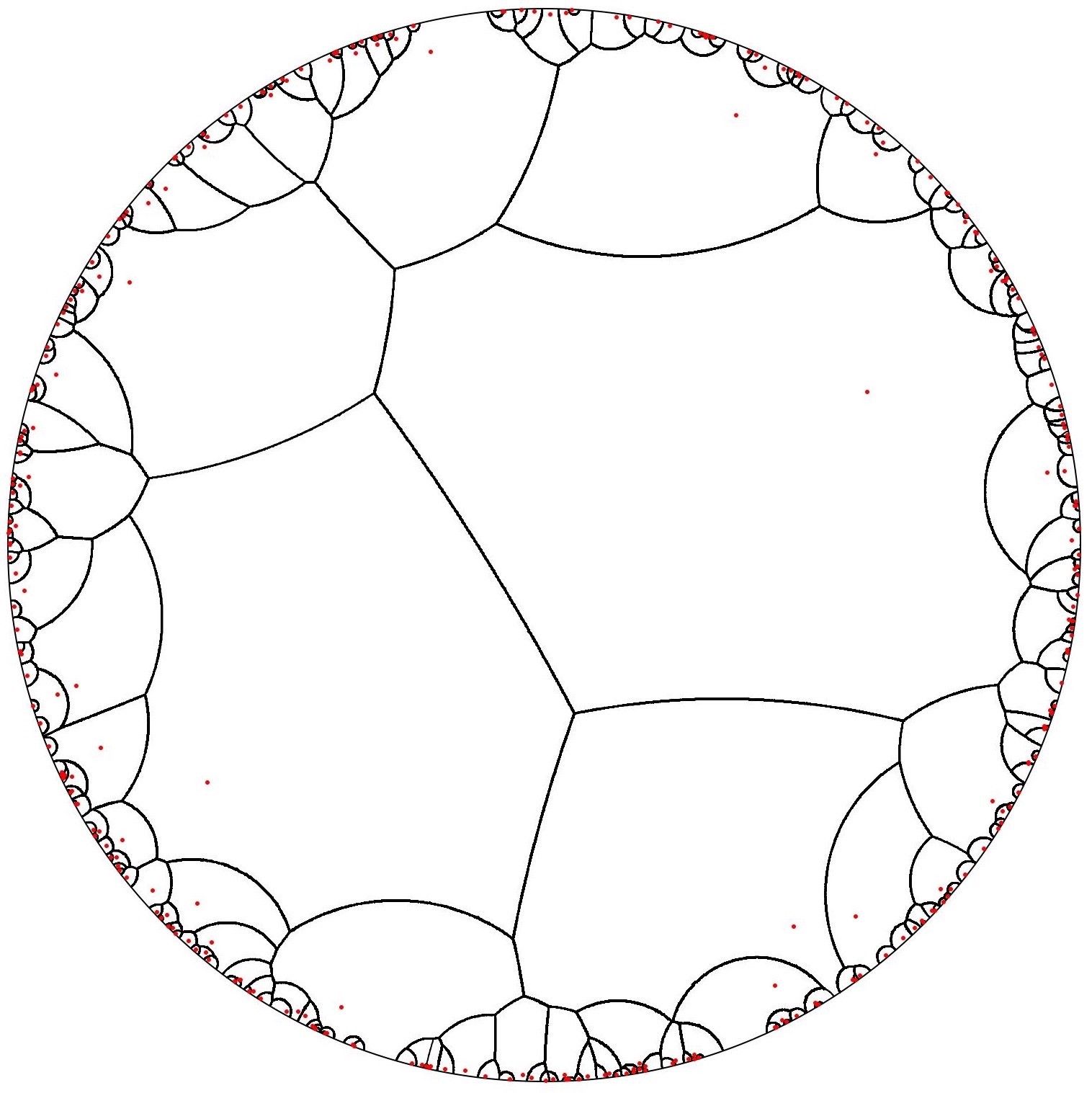}
   \includegraphics[height=5cm]{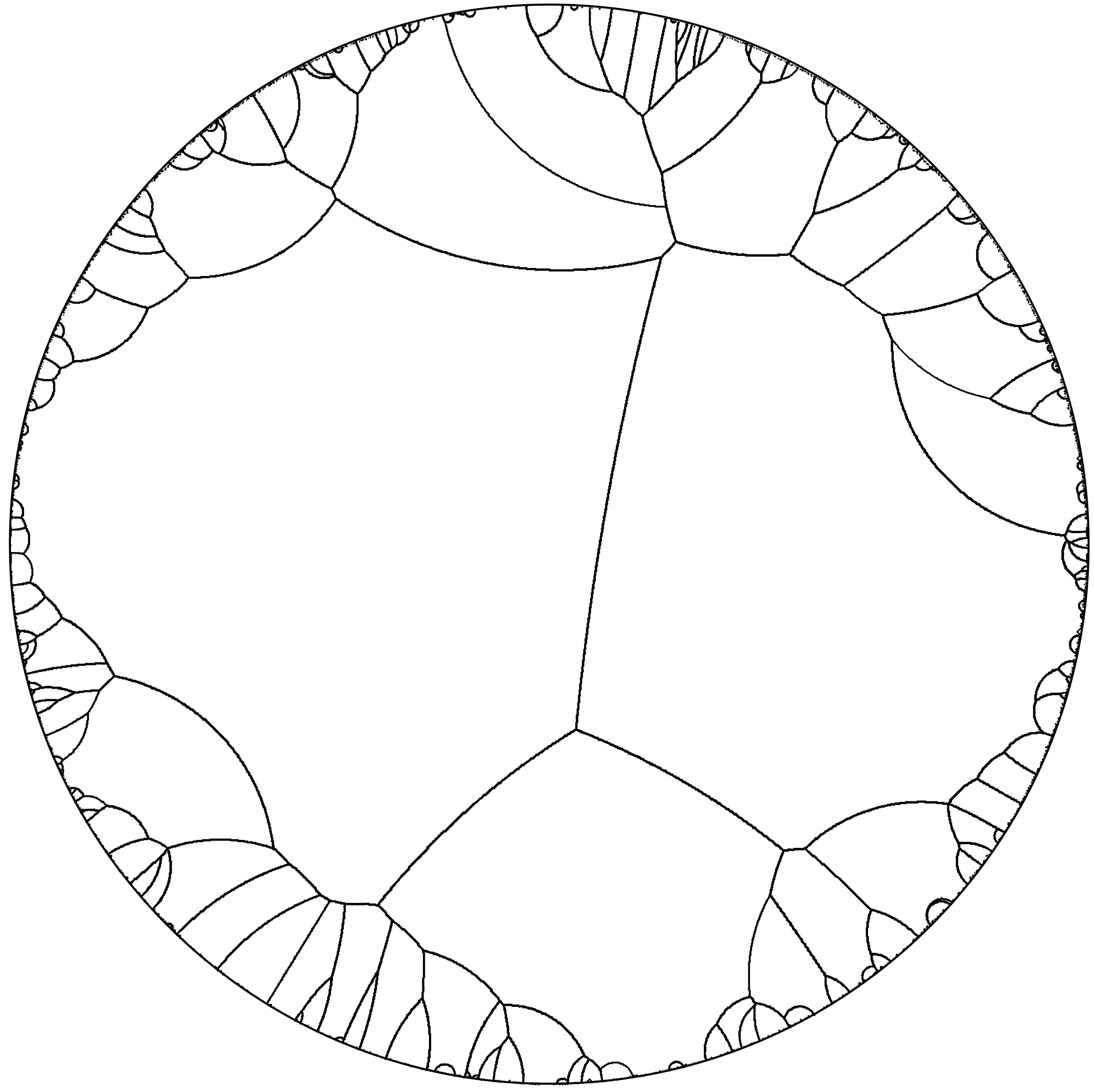}
 \caption{ From  \cite{BudzinskiCurienPetri}. Left to right: Poisson--Voronoi tessellations of the hyperbolic plane (in the unit disk model) with decreasing intensity. Their limit (on the right) is $\mathcal{V}_2$, the  \emph{ideal Poisson--Voronoi tessellation} of the hyperbolic plane.}\label{fig:voronoiH}
 \end{center}
 \end{figure}

By \dfn{tessellation} in $\R^d$ or $\HH_d$, we mean a locally finite collection of nonempty, closed sets (called ``tiles" or ``cells") that are (geodesically) convex, have nonempty and pairwise disjoint interiors, and that cover the space. The cells are then (possibly unbounded) polytopes, i.e., each cell is the convex hull of a locally finite set of points (its vertices)---equivalently, each cell is the intersection of a locally finite set of closed half-spaces (see, e.g., \cite[Lemma 10.1.1]{SchneiderWeil} for tessellations in $\R^d$; the same follows for tessellations in $\HH_d$ by using the Beltrami--Klein model, where every complete, totally geodesic submanifold of dimension $k$ is the intersection of the Euclidean unit ball $\uball_d$ with an affine Euclidean plane of dimension $k$ \cite[Section 6.3]{ratcliffe}). The faces of a cell are the intersections of the cells with one or more of its supporting (geodesic) hyperplanes. In our case, the intersection of two cells will always be a face of each cell, or empty; such a tessellation is called \dfn{face-to-face}. Furthermore, in our case, a $k$-face will always be the intersection of $d-k+1$ cells, but no fewer and no more; such a tessellation is called \dfn{normal}.

The limiting tessellation $\mathcal{V}_{d}$ can be constructed from a Poisson process of points on the boundary of the $d$-dimensional hyperbolic space, $\partial \mathbb{H}_d $, cross the real numbers, the second coordinates of which we call \dfn{delays}. Equivalently, this product is the space of horospheres in $\HH_d$.  In coordinates, we provide the following explicit descriptions of $\cells_d$ as a multiplicatively weighted Voronoi diagram, where $\Leb$ denotes Lebesgue measure on the appropriate spaces, $\Sph_{d-1} \coloneqq \partial \uball_d$, and $| \,\cdot\, |$ denotes the Euclidean norm:

\begin{thm}[\textsc{Coordinate description of $ \mathcal{V}_d$}]\label{t.weighted} For the unit ball model of\/ $\HH_d$, let $N$ be a Poisson point process with intensity $\Unif \otimes \Leb$ on $\Sph_{d-1} \times \R_+$, while for the upper half-space model of\/ $\HH_d$, let $N$ be a Poisson point process with intensity $\Leb \otimes \Leb$ on $\R^{d-1} \times \R_+$. In both models, the ideal Poisson--Voronoi tessellation $\cells_d$ consists of the cells $\bigl\{C(\theta, r) \ST (\theta, r) \in N\bigr\}$ defined by
\rlabel e.coordCells
{C(\theta, r) 
\coloneqq
\bigl\{z \ST r^{1/2(d-1)} |z - \theta| \le s^{1/2(d-1)} |z - \psi| \mbox{ for all } (\psi, s) \in N\bigr\}.}
\end{thm}

This enables us to study the stochastic properties of $ \mathcal{V}_{d}$ directly. In particular, we prove that $\mathcal{V}_{d}$ is a natural cell decomposition of $ \Hyp{d}$ in the following sense:
\begin{thm}[\textsc{Tiles have one end}] \label{thm:decomposition} Almost surely,
\begin{itemize}
\item $\cells_d$ is a tessellation of\/ $\HH_d$;
\item each tile of\/  $ \mathcal{V}_{d}$ is unbounded with an infinite number of bounded faces and with a unique limit point on the ideal boundary of\/ $\HH_d$, which we refer to as its \dfn{end};  
\item $\cells_d$ is face-to-face and normal; and
\item the law of\/ $ \mathcal{V}_{d}$ is invariant under every isometry of\/ $ \mathbb{H}_d$.
\end{itemize}
In particular,  in dimension $2$, the union $ \partial \mathcal{V}_{2}$ of the boundaries of the tiles  is a random embedding in $ \mathbb{H}_{2}$ of the 3-regular tree with geodesic edges whose law is invariant under isometries of the hyperbolic plane.
\end{thm}

The last bullet point of Theorem~\ref{thm:decomposition} follows from the M{\"o}bius invariance of the hyperbolic measure, hence of the law of the Poisson cloud $ \mathbf{X}_{d}^{{(\lambda)}}$ for each $\lambda >0$. The third point is also classical in stochastic geometry and is in particular a well-known fact for standard Poisson--Voronoi tessellations in Euclidean space; see \hbox{\cite[Theorems 10.2.1 and 10.2.3]{SchneiderWeil}}. 

\begin{figure}[!hbtp]
\centering
\begin{minipage}{.45\textwidth}
\begin{overpic}[width=\textwidth]{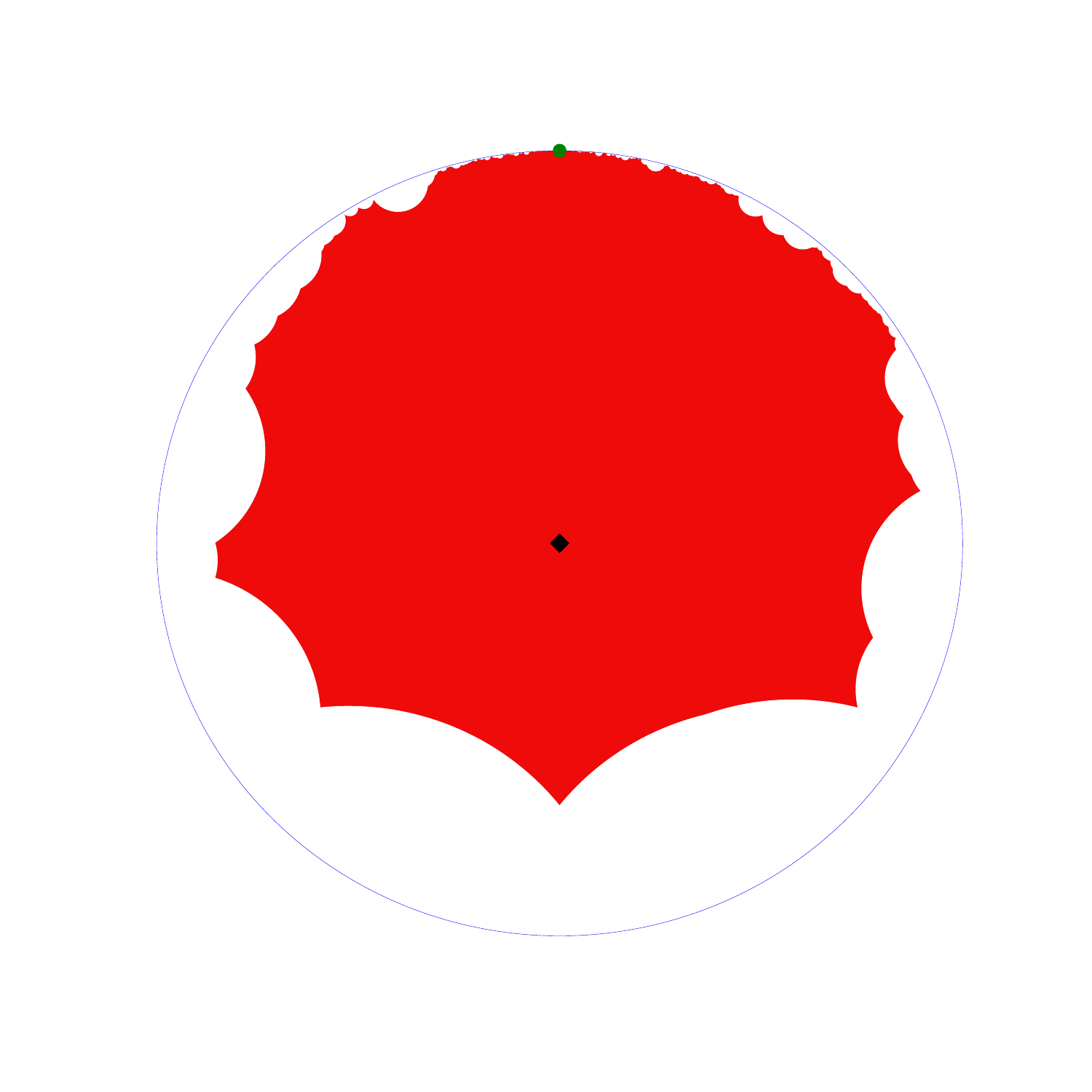}
   \put (45,45) {$\origin$}
\end{overpic}   
     \end{minipage}
     \hspace{12pt}
     \begin{minipage}{.47\textwidth}
\begin{overpic}[width=\textwidth]{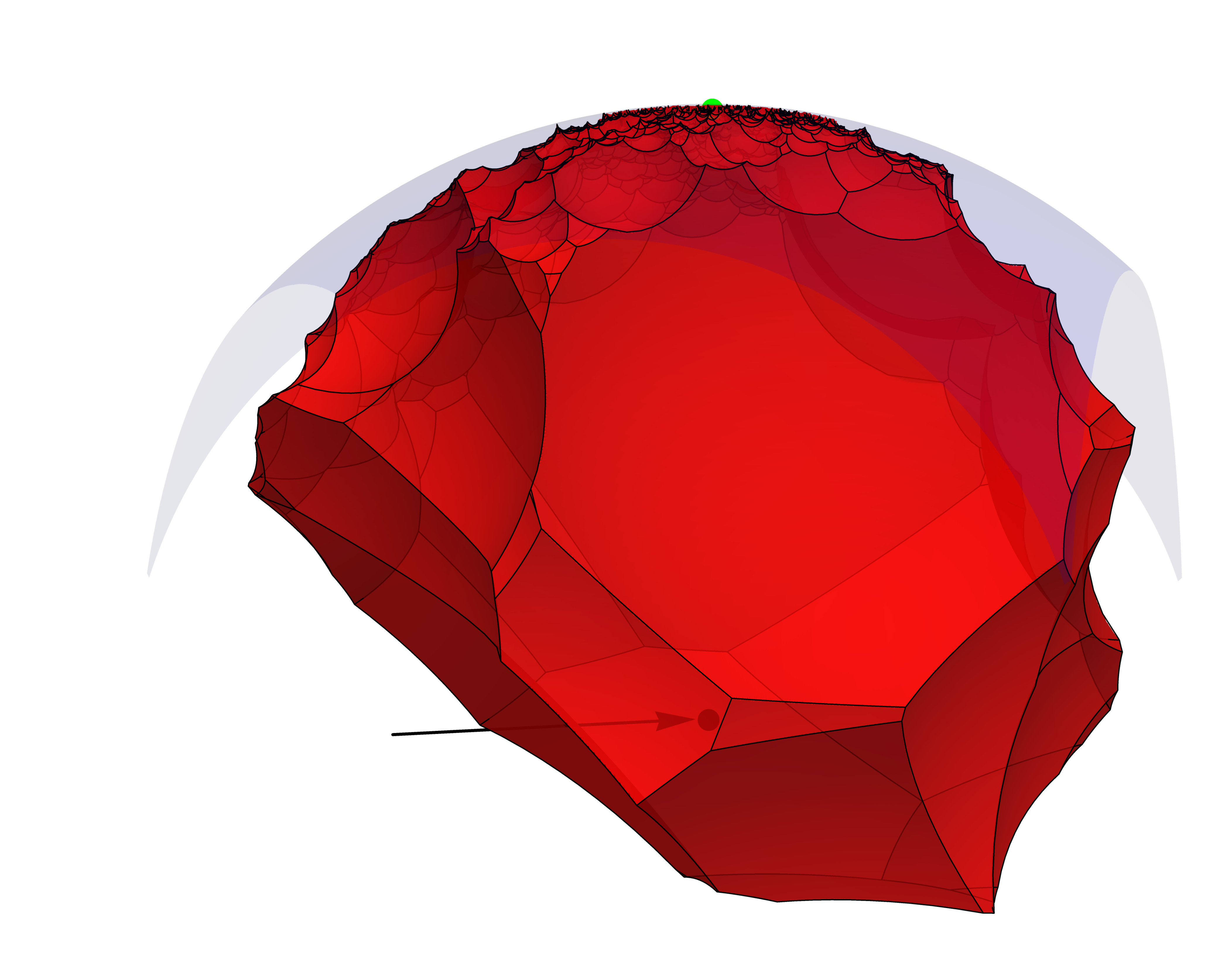}
   \put (26,18) {$\origin$}
\end{overpic}   
     \end{minipage}
    \caption{Simulations of the cell of $ \mathcal{V}_d$ containing the origin in the Poincar\'e ball model of $\mathbb{H}_d$ for $d=2$ (left) and $d=3$ (right), obtained by sampling from \eqref{eq.intm} conditionally on the median value $\frac{(d-1)}{c_{d}}\log{2}$ of $\mathcal{E}_{d}$. }
    \label{fig.cell_B}
\end{figure}

We denote by $\mathcal{C}_d$ the cell of $\mathcal{V}_{d}$ containing the origin of $ \mathbb{H}_d$, which, by the M{\"o}bius invariance of $\mathcal{V}_{d}$, has the same law as the cell of any other fixed point. We refer to $\ocell_d$ as the \dfn{zero cell}; see \Cref{fig.cell_B}. Through the zero cell, we investigate the fine properties of the tiles of $\mathcal{V}_{d}$.
As we stated in Theorem \ref{thm:decomposition}, the cell $ \mathcal{C}_d$ almost surely has a unique end $\in \partial \mathbb{H}_{d}$, and once we view $ \mathcal{C}_d$ in the upper half-space model of $ \mathbb{H}_{d}$ with its unique end sent to $\infty$ and the origin sent to $(0, 0, \ldots, 0, 1)$, its law can be described in a surprisingly simple and appealing way using a \dfn{deposition model}:  Let $ \mathcal{E}_{d}$ be a random variable with law ${\rm{Exp}}\big(\frac{c_{d}}{d-1}\big)$, where 
\rlabel e.defc_d
{c_d\coloneqq2^{2-d}\frac{\pi^{\frac{d}{2}}}{\Gamma\big(\frac{d}{2}\big)}
}
is $2^{1-d}$ times the Euclidean volume of $\Sph_{d-1}$. Conditionally on $ \mathcal{E}_{d}$, let  $\Pi_d$ be a Poisson cloud of hemispheres in $ \mathbb{R}^{d-1} \times \mathbb{R}_+$ with centers $x \in \mathbb{R}^{d-1}$ and radii $\rho >0$ having intensity 
\begin{equation}\label{eq.intm}
 2 \cdot \mathcal{E}_{d} \cdot \ud x \ \rho^{1-2d}\mathrm{d}\rho \, \mathbf{1}_{ \rho \leq \sqrt{1+|x|^{2}}}\; .
 \end{equation}

\begin{thm}[\textsc{Description of $ \mathcal{C}_d$}]\label{thm.superposition} The law of $ \mathcal{C}_d$ is given by the complement of all open hemispheres whose centers and radii are given by $ \Pi_d$; see \Cref{fig.cell_A}.
\end{thm}

\begin{figure}[!h]
 \begin{center}
\includegraphics[height=0.3\linewidth]{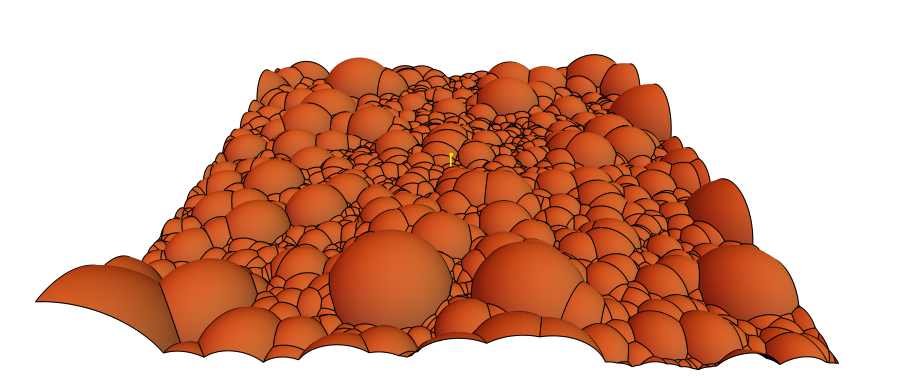}

\includegraphics[width=.7\linewidth]{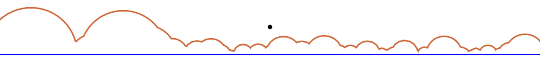}
 \caption{Representation of a finite portion of $\partial\mathcal{C}_3$ (top) and $\partial\mathcal{C}_{2}$ (bottom) in the upper half-space models in dimensions $3$ and $2$ (conditional as in \Cref{fig.cell_B}). In both cases, the portions go from $-10$ to $10$ in the coordinate directions on the ideal boundary. The origin is shown on top as a yellow dot with a vertical line below it, while the origin is shown as a black dot on bottom. Note that below the boundary of the zero cell lie all the other cells. A portion of such in dimension 3 is shown in \Cref{f.foam-cake}.}
 
 \label{fig.cell_A}
 \end{center}
 \end{figure}

\begin{figure}[!h]
 \begin{center}
\includegraphics[width=.77\linewidth]{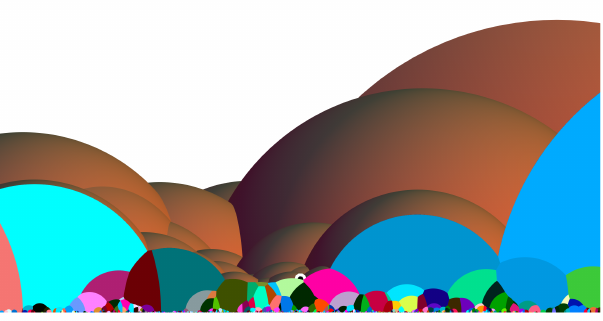}
 \caption{A vertical slice through the origin showing a portion of the boundary of the zero cell as well as those cells that intersect the vertical plane of the slice. The origin is shown as a black dot inside a white dot.}
 \label{f.foam-cake}
 \end{center}
 \end{figure}

In particular, if one forgets about the condition $\mathbf{1}_{ \rho \leq \sqrt{1+|x|^{2}}}$ in the intensity \eqref{eq.intm}, which excludes only finitely many hemispheres a.s., one sees that $ \mathcal{C}_d$ has almost the same law as a random dilation of the complement of balls whose centers and radii have intensity $\rho^{1-2d}\mathrm{d}\rho \,\mathrm{d}x$. This complement is the epigraph of a random continuous field over $ \mathbb{R}^{d-1}$ whose marginal law is made explicit in Proposition \ref{prop.hscpl}. This field and these balls have a law that is invariant under all Euclidean isometries of $\R^{d-1}$ and yields after orthogonal projection a random tessellation of $ \mathbb{R}^{d-1}$ that is a special case of the Laguerre tessellations studied in \cite{gusakova2022delaunay}.

The above Poissonnian construction of the zero cell $ \mathcal{C}_d$ is the main tool to study various distributional quantities such as: 
\begin{itemize}
\item the hole probability (Proposition \ref{prop.holeprob}) for $ \mathcal{C}_d$,
\item the stationary distribution of the height  and azimuth field describing the deposition model (Proposition~\ref{prop.hscpl}),
\item the vertex intensity of $\mathcal{C}_d$ in the Euclidean stationary model (Proposition \ref{prop.vertint}).
\end{itemize}
By passing to the limit $\lambda \to 0$ in several computations made for nonzero-intensity Poisson--Voronoi tessellations of $ \mathbb{H}_{d}$ \cite{GodKabTha}, we also compute explicitly the $k$-dimensional face intensities for $0 \leq k \leq d-1$ of the tiling $ \mathcal{V}_{d}$ in Theorem \ref{t.FaceInt}.  At certain points in this work, we study the dual complex to $\cells_d$, the ideal Poisson--Delaunay tessellation.

The paper is structured as follows: In Section \ref{sec:conv}, we set up an abstract framework for convergence of Voronoi tessellations whose points converge to the ideal boundary of a space.  Section~\ref{sec.idpvths} establishes the main properties of the ideal Poisson--Voronoi tessellations in $d$-dimensional hyperbolic spaces, along with the first three theorems stated in the introduction. At the end of \cref{sec.idpvths}, we develop the basic properties of the ideal Poisson--Delaunay tessellation. We then turn to precise computations of various distributions related to $ \mathcal{V}_d$ and to $\mathcal{C}_{d}$. Using  \cite{GodKabTha}, we compute the $k$-face intensities in Section \ref{s:intensities}, while we focus on the cell of the origin, $ \mathcal{C}_{d}$, in Section \ref{s.zerocell}. Section~\ref{sec.trees} treats the case for regular trees and recalls some results obtained by \cite{bhupatiraju}. Finally, Section \ref{s.open} provides some future directions. \medskip

\noindent \textbf{Acknowledgements}. N.C.~thanks Thomas Budzinski and Bram Petri for enlightening discussions around and during the conception of \cite{BudzinskiCurienPetri}. The work of M.A.~and N.C.~was supported by ANR RanTanPlan and ERC Consolidator Grant SuperGRandMA (Grant No.~101087572). N.E.\ was partially supported by the CNRS grant RT 3477, Geométrie Stochastique. R.L.\ was partially supported by NSF grants DMS-1612363 and DMS-1954086 and the Indiana University Institute for Advanced Study. M.\"U.\ was supported by grants from the Fondation Mathématique Jacques Hadamard (FMJH). All authors thank the reviewer for thoughtful remarks.

\section{Abstract convergence results}\label{sec:conv}
In this section, we give an abstract point of view on the convergence of Voronoi diagrams based on the concept of the Gromov boundary of a metric space $(E,d)$. The main idea is that building a Voronoi diagram requires us to compare only \emph{differences of distances}, rather than actual distances. In particular, the forthcoming Theorem \ref{lem.convvor} is purely deterministic and could be applied to many different spaces. 

\subsection{Gromov boundary and ideal Voronoi diagrams}\label{s.horo}
Let $(E,d)$ be a locally and boundedly compact metric space. The set $C(E)$ of real-valued continuous functions on $E$ is endowed with the topology of uniform convergence on every compact subset of $E$. We define an equivalence relation on $C(E)$ by declaring two functions equivalent if they differ by an additive constant; the associated quotient space endowed with the quotient topology is denoted by $C(E)/\mathbb{R}$. Following \cite{Gro81}, one can embed the original space $E$  in $C(E)/\mathbb{R}$ using the injection
$$ i : \begin{array}{ccccc} 
 E &\longrightarrow&{C}(E)& \longrightarrow& {C}(E)/\mathbb{R}\\ x&\longmapsto & d_{x}\coloneqq d(x,\cdot) & \longmapsto & \overline{d_{x}}\, .
\end{array}
$$
The \dfn{Gromov compactification} of $E$ is then the closure of $i(E)$ in $C(E)/\mathbb{R}$. The \dfn{Gromov boundary}%
\ or \dfn{horoboundary} $\partial E$ of $E$ is composed of the points in the closure of $i(E)$ in $C(E)/ \mathbb{R}$ that are not already in $i(E)$. The points in $\partial E$ are called \dfn{horofunctions}; see~\cite{Gro81}. They are obtained as limits of shifted distance functions $d(x_n,\cdot)-d(x_n,\origin)$ for some sequence of points $x_n \to \infty$. 
Let us denote a point on the Gromov boundary by $\theta$ and fix the associated representative function $ \mathrm{d}_{\theta}(\cdot)$ so that $ \mathrm{d}_{\theta}(\origin)=0$ for all $\theta \in \partial E$. In particular, if $\theta, \theta' \in \partial E$ and $x \in E$, one can make sense of the ``difference of distances''
 \begin{equation} \mathrm{d}_{\theta}(x) - \mathrm{d}_{\theta'}(x) \in \mathbb{R}. \label{eq:difference}  \end{equation}  We will then enhance the Gromov boundary by the addition of a second coordinate, real numbers which we call the \dfn{delays}, yielding the \dfn{extended ideal boundary}, $ \widehat{\partial E} \coloneqq \partial E \times \mathbb{R}$. Extending the preceding display, if $x \in E$ and $(\theta,\delta), (\theta',\delta') \in \widehat{\partial E}$ are two extended ideal points, we say that $x$ is \dfn{closer} to $(\theta,\delta)$ than to $(\theta',\delta')$ if 
 $$ \mathrm{d}_{\theta}(x) - \mathrm{d}_{\theta'}(x) \leq \delta'- \delta.$$ Letting
 \[
\ud\bigl(x, (\theta, \delta)\bigr)
\coloneqq
\ud_\theta(x) + \delta,
 \]
which we call the \dfn{linear separation} of $x$ from $(\theta, \delta)$ and which is real valued, the condition becomes $ \mathrm{d}\bigl(x, (\theta, \delta)\bigr) \leq \mathrm{d}\bigl(x, (\theta',\delta')\bigr)$. 

The level sets of a horofunction $\ud_\theta$ are called \dfn{horospheres} at $\theta$. A horosphere is thus a limit of spheres in the sense that $\ud_\theta(x) = \delta$ iff there are spheres containing $x$ centered at points $x_n \to \theta$ with radii $d(x_n, \origin) + \delta_n$ and $\delta_n \to \delta$. The horosphere $\ud_\theta^{-1}\bigl[\{\delta\}\bigr]$ is also the zero set of $\ud\bigl(\cdot, (\theta, -\delta)\bigr)$. If $\phi$ is an isometry of $(E, d)$, then $\phi$ acts naturally on $\partial E$ mapping one horosphere to another in the following way, as shown from the definition: for all $x \in E$,
\rlabel e.horotransform
{\ud_{\phi(\theta)}\bigl(\phi(x)\bigr) 
=
\ud_\theta(x) - \ud_\theta\bigl(\phi^{-1}(\origin)\bigr).}
Hence, if we write 
\rlabel e.horoaction
{\phi(\theta, \delta) \coloneqq \bigl(\phi(\theta), \delta + \ud_\theta(\phi^{-1}(\origin))\bigr),}
then we obtain an action of $\phi$ on the space of horospheres, $ \widehat{\partial E}$, that preserves linear separation:
\rlabel e.preserve-sep
{\forall x \in E \enspace \forall (\theta, \delta) \in \widehat{\partial E} \quad
\ud\bigl(\phi(x), \phi(\theta, \delta)\bigr)
=
\ud\bigl(x, (\theta, \delta)\bigr).}

\paragraph{(Ideal) Voronoi diagrams.} Let us first recall the basic definition of Voronoi diagrams. Fix a sequence of distinct points $ (x_i \ST i \geq 1) \in E$, called the \dfn{nuclei}. We will always suppose  that the nuclei $x_i$ are ranked by their increasing distances to the origin of $E$ and that $ d( \origin, x_{i}) \to \infty$. The  \dfn{Voronoi diagram} $ \mathrm{Vor}(x_{i} \ST i \geq 1) = ( C_{i} \ST i \geq 1)$ is defined via its \dfn{tiles} $C_{i}$ associated to $ x_i $ via
$$
C_{i} \coloneqq \left\lbrace w \in E \ST \forall j \geq 1\ d(w,x_i)\leq d(w,x_j) \right\rbrace  .
$$

We also define the ``open" version of the tiles denoted by $C_i^\circ$ when the weak inequality is replaced by a strict one. Notice that  the $C_i^\circ$ are  disjoint; since we imposed that $d( \origin, x_{j}) \to \infty$ as $j \to \infty$, it follows that $C_i^\circ$ are indeed open sets.

\medskip

 The framework above enables us to define diagrams using nuclei \emph{that are not points of $E$} but extended ideal points of $ \widehat{\partial E}$. More precisely, if $ \boldsymbol{\theta} \coloneqq (\theta_{i} \ST i \geq 1)$ is a sequence of boundary points on $\partial E$ and $ \boldsymbol{\delta}\coloneqq (\delta_{i} \in \mathbb{R} \ST i \geq 1)$ is an increasing sequence of reals  tending to $\infty$, we define the (ideal) Voronoi diagram associated to $ ( \boldsymbol{\theta},  \boldsymbol{\delta})$ as follows: 

 \begin{defn}[\textsc{Ideal Voronoi diagrams}]\label{d.IVT} The Voronoi diagram $ \mathrm{Vor}\bigl( (\theta_{i}, \delta_{i}) \ST {i \geq 1}\bigr)$ is given by its tiles $(C_{i} \ST i \geq 1)$, where 
 $$ C_{i} \coloneqq \bigl\{ x \in {E} \ST \forall j \ne i\enspace \mathrm{d}\bigl(x,(\theta_{i}, \delta_{i})\bigr) \leq \mathrm{d}\bigl(x, (\theta_{j}, \delta_{j})\bigr)  \bigr\},$$ and similarly with $C_i^\circ$ defined via a strict inequality.
 \end{defn}
 We say that the diagram is \dfn{nondegenerate} if $\overline{C_i^\circ} = C_i$ for every $i \geq 1$.  Notice that since we impose that $\delta_{i} \to \infty$, every compact $K \subset E$ intersects only a finite number of tiles of the diagram, and the tiles $C_i^\circ$ are again open. Also, the diagram is invariant under a shift of all delays $(\delta_{i} \ST i \geq 0)$ because it depends only on the differences of linear separations.
 \subsection{Convergence of diagrams}
 As we will see, the concept of Gromov boundary and ideal diagrams is well suited for studying \emph{convergence} of standard Voronoi diagrams.
We will use the \dfn{Fell topology} on the collection of closed subsets of $E$, which is generated by the sets $\{F \sut F \cap C = \varnothing\}$ for compact $C \subset E$ and the sets $\{F \sut F \cap G \ne \varnothing\}$ for open $G \subseteq E$. This topology makes the collection of closed subsets compact, while if the empty set is omitted, then the collection becomes locally compact (see \cite[Section 12.2]{SchneiderWeil}).
\begin{defn}[\textsc{Convergence of diagrams}] We will say that the sequence of diagrams $\bigl((C_i^{(k)} \ST i \geq 1)\bigr)_{k \ge 1}$ converges to $(C_i \ST i \geq 1)$ as $ k \to\infty$ if for each $i \geq 1$, the closed subsets $C_i^{(k)}$ converge to the closed subset $C_i$ in the Fell topology.\end{defn}

The following lemma, roughly speaking, entails the continuity of the Voronoi-{diagram} mapping with respect to the convergence of nuclei towards the extended Gromov boundary, $ \widehat{\partial E}$.
\begin{thm}\label{lem.convvor} Suppose that for each $k$, we have a sequence of points $(x_i^{ (k)} \ST i \geq 1)$ ranked by increasing distances to $\origin$ on a space $(E,d)$ that together satisfy the following three conditions:
\begin{enumerate}[label=\textup{(\roman*)}]
\item \textsc{(Convergence to the boundary)} For all $i\geq 1$, we have convergence in the Gromov sense   $$x_i^{ ( k)}  \xrightarrow[    k \to \infty]{} \theta_i \in  \partial E.$$
\item \textsc{(Convergence of proto-delays)} For all $i\geq 1$, we have  $$d(x_i^{ ( k)},  \origin) - d(x_1^{ ( k)}, \origin)  \xrightarrow[k \to \infty]{} \delta_{i}.$$ Furthermore, $\delta_{i} \to \infty$ as $i \to \infty$.
\item \textsc{(Nondegeneracy)} $\mathrm{Vor}\bigl( (\theta_{i}, \delta_{i}) \ST {i \geq 1}\bigr)$ is nondegenerate. 
\end{enumerate}
Then the Voronoi diagrams $ \mathrm{Vor}(  x_{i} ^{(k)} \ST \allowbreak i \geq 1)$ converge as $ k \to \infty$ to the ideal Voronoi diagram $ \mathrm{Vor}\bigl( (\theta_{i}, \delta_{i}) \ST {i \geq 1}\bigr)$.
\end{thm}
\begin{remark}[Degenerate cases] \label{remark:chiant}To see that the third condition in the lemma is needed, suppose that $E = [0, \infty)$, so that $\partial E$ consists of a single point, $\theta$. Then  $x_{i}^{(k)} \to \theta$ as soon as $  d(x_{i}^{(k)}, \origin) \to \infty$ as $k \to \infty$. If $d(x_1^{(k)}, \origin) < d(x_2^{(k)}, \origin)$ for each $k$ and the limiting delays satisfy $0=\delta_{1} = \delta_{2} <  \delta_{3} \leq  \cdots$, then the tiles $C_{i}$ of the limiting degenerate ideal diagram
$\mathrm{Vor}\bigl( (\theta, \delta_{i}) \ST i \geq 1\bigr)$ are $ C_{1} = C_{2} = E$  and $C_{i} = \varnothing$ for $i \geq 3$, yet $C_2^{(k)} \to \varnothing$ as $k \to\infty$. \end{remark}

\begin{proof} Denote by $C_{i}^{(k)}$ (resp., $C_i^{\circ,(k)}$) the closed (resp., open) tiles of $\mathrm{Vor}(  x_{i} ^{(k)} \ST i \geq 1)$ and by $C_{i}$ (resp., $C_i^\circ$) those of $\mathrm{Vor}\bigl( (\theta_{i}, \delta_{i}) \ST {i \geq 1}\bigr)$. Since the space of closed subsets of $E$ endowed with the Fell topology is itself compact, we can suppose up to passing to a subsequence that we have 
$$  C_i^{(k)} \xrightarrow[k\to\infty]{} \mathfrak{C}_{i}$$ in the Fell topology for some closed subset $ \mathfrak{C}_{i}$, and our goal is now to show that $ \mathfrak{C}_i = C_i$.
We define the proto-horofunctions
$$
 \mathrm{d}_{\x} (z) \coloneqq d(\x, z) - d (\x,\origin) ,
$$
as well as the proto-delays $ \delta_i^{(k)} \coloneqq d(x_i^{ ( k)},  \origin) - d(x_1^{ ( k)}, \origin)$. A given point $z \in E$ thus belongs to $C_{i}^{(k)}$ iff for all $j \ne i$, we have the inequality 
\bb\label{diffij}
d(\x,z)- d(\xj,z) =  \mathrm{d}_{\x} (z) -  \mathrm{d}_{\xj} (z) + \delta_i^{(k)} - \delta_j^{(k)} \leq  0 ,
\ee and similarly for $C_i^{\circ, (k)}$ with a strict inequality. 
By our first two assumptions, the function of $z$ in the last display converges uniformly on compact sets towards  
$$  \mathrm{d}_{\theta_{i}}(\cdot) - \mathrm{d}_{\theta_{j}}(\cdot) +\delta_{i}- \delta_{j}.$$
 It follows immediately that we have $ \mathfrak{C}_i = \lim_{k \to \infty} C_i^{(k)} \subseteq C_i$: indeed, if $z \in \lim_{k \to \infty} C_{i}^{{(k)}}$, then there are $z_{i}^{(k)} \in C_{i}^{(k)}$ satisfying \eqref{diffij} with $z_{i}^{(k)} \to z$, so that passing \eqref{diffij} to the limit, we deduce that $z \in C_{i}$.  For the other inclusion, we show that $C_{i}^{\circ} \subset \mathfrak{C}_{i}$ and use the non-degeneracy $ \overline{C_i^\circ} = C_i$ to conclude. If $z\in C_i^\circ$, then for every fixed $j \ne i$, \eqref{diffij} holds with strict inequality for all large $k$, whence $z \in C_i^{(k)}$ for all large $k$. This gives $z \in \mathfrak C_i$, as desired.
\end{proof}

\section{Ideal Poisson--Voronoi tessellations on hyperbolic spaces}\label{sec.idpvths}
As mentioned in the introduction, the limit in low intensity of Poisson--Voronoi tessellations in $ \mathbb{R}^{d}$ is trivial. This comes from the fact that, although the Gromov boundary of $ \mathbb{R}^{d}$ is nontrivial (it is homeomorphic to $ \mathbb{S}_{d-1}$), the polynomial growth of $ \mathbb{R}^{d}$ imposes that the difference of distances to $\origin$ (the delays) of the first two closest points in a PPP with intensity $\lambda$ tends to $\infty$ in probability as $\lambda \to 0$. Indeed, it is easy to convince oneself that superpolynomial growth is required to get tight delays as $\lambda \to 0$. A natural choice of such a space is the $d$-dimensional hyperbolic space, $ \mathbb{H}_{d}$.

\subsection{Background on hyperbolic spaces}
\label{sec.basicfacts}
For the hyperbolic space $ \mathbb{H}_d$ with dimension $d \geq 2$, we will use the model of the open unit ball $ \mathbb{B}_d \coloneqq \{ x \in \mathbb{R}^{d}\ST |x|< 1\}$ equipped with the metric $2|dx|/\bigl(1 - |x|^2\bigr)$, which leads to the  $$ \mbox{distance} \quad (x,y) \mapsto 2 \arcsinh{\frac{|x-y|}{\sqrt{(1-|x|^2)(1-|y|^2)}}} \quad \mbox{and measure} \quad \Bigl(\frac{2}{1-|x|^{2}} \Bigr)^{d}\mathbf{1}_{\mathbb{B}_d } \cdot \mathrm{Leb}, $$ and the model of the upper half-space $ \mathbb{U}_{d}\coloneqq \mathbb{R}^{d-1} \times \mathbb{R}_{>0}$ equipped with the metric $|dx|/x_d$, where $x = (x_1, x_2, \dots, x_d)$, which leads to the $$\mbox{distance} \quad (x,y) \mapsto 2 \arcsinh{\frac{|x-y|}{2 \sqrt{x_{d} \,  y_{d}}}}\quad \mbox{and measure} \quad \frac{1}{x_{d}^{d}}\mathbf{1}_{x_{d}>0}\cdot \mathrm{Leb}, $$ 
where $\Leb$ denotes the usual Lebesgue measure; see~{\cite[pp.\ 9, 11, 13, 14]{Stoll}}. 
We will write respectively $\mathrm{d}_{ \mathbb{H}_d}$ and $ \mathrm{Vol}_{ \mathbb{H}_d}$ for the hyperbolic distance and measure. The origin of $\mathbb{H}_d$ will be denoted by $\origin$, meaning the center of the ball or $(0, 0, \ldots, 1)$ in $\mathbb{U}_d$.

The isometries of $\mathbb{H}_d$ form the group $\mob_d$ of  M\"obius transformations of $\HH_{d}$.
An isometric mapping from $\mathbb{B}_d$ onto $\mathbb{U}_d$ is given by the generalized Cayley transform, which is the diffeomorphism $\kappa\colon \mathbb{B}_d \rightarrow \mathbb{U}_d$ defined by
\begin{align*}
\kappa(z) &= x \coloneqq \frac{1}{z_1 ^2 + \dots + z_{d-1} ^2 + (z_d-1)^2}\big( 2z_1, \dots, 2z_{d-1}, 1-|z|^2 \big) , \\
\kappa^{-1} (x) &= z \coloneqq \frac{1}{x_1 ^2 + \dots + x_{d-1} ^2 + (x_d+1)^2}\big( 2x_1, \dots, 2x_{d-1}, |x|^2 -1 \big) 
\end{align*}
for $z \in \mathbb{B}_d$ and $x \in \mathbb{U}_d$ \cite[p.~66]{leejm}. If $\rf$ denotes reflection in the plane $x_d = 0$, then the maps $z \mapsto \kappa\bigl(\rf (z)\bigr)$ and $x \mapsto \rf\bigl(\kappa^{-1}(x)\bigr)$ are both restrictions of inversion in the sphere of radius $\sqrt2$ centered at $(0, 0, \dots, -1) \in \R^d$; see \cite[Exercise 2.4.6(a)]{Stoll}.
The maps $\kappa$ and $\kappa^{-1}$ extend continuously to the boundaries where $|z| = 1$ and  $x_d = 0$ or $x = \infty$, yielding
the stereographic projection from the north pole onto the plane containing the equator, $\mathrm{Ste}\colon \mathbb{S}_{d-1} \to  \mathbb{R}^{d-1}$, and its inverse.
The image of the uniform measure on $ \Sph_{d}$ by $\mathrm{Ste}$ to $\mathbb{R}^{d-1}$ is given by:
\begin{lem}\label{lem.stereo} The image of the uniform measure on $ \mathbb{S}_{d-1}$ by the stereographic projection $\mathrm{Ste}\colon \mathbb{S}_{d-1} \to  \mathbb{R}^{d-1}$ is the stereographic law given by 
$$\frac1{c_d\bigl(1+|x|^2\bigr)^{d-1}} \,\ud x
=
\frac{1}{{c_d}\bigl(1+\sum_{i=1}^{d-1} x_i^2 \bigr)^{d-1}} \mathrm{d}x_1 \cdots\,  \mathrm{d}x_{d-1},$$   where $c_d$ is as in \eqref{e.defc_d}.
\end{lem}
\begin{proof} This can be proved by a tedious but elementary calculation or by the explicit expression of the metric tensor on the sphere $ \mathbb{S}_{d-1}$ in stereographic coordinates in \cite[Equation 3.7]{leejm}. 
A third proof uses the hyperbolic Poisson kernels, since $\Ste$ is merely the extension to the boundaries of the isometry $\kappa$; for such a proof, see \cite[Exercise 5.7.15]{Stoll}.
\end{proof}

\subsection{The ideal tessellation $ \mathcal{V}_d$} \label{s.IPVT-stronger}

Below we prove a slightly stronger version of Theorem~\ref{thm.convhyp}. 
 
 \begin{theorem}\label{cor.thetaD} The Poisson--Voronoi tessellation of $( X_{i}^{(\lambda)} \ST i \geq 1)$ converges in distribution as $\lambda \downarrow 0$ towards the nondegenerate ideal Voronoi diagram $\mathrm{Vor}( \boldsymbol{\Theta},  \mathbf{D})$ where $\boldsymbol{\Theta} = (\Theta_{1}, \ldots )$ are i.i.d.\ uniform angles over $\mathbb{S}_{d-1} = \partial \mathbb{B}_{d}$ and $  \mathbf{D} = ( D_{i} \ST i \geq 1)$ is such that $ (\frac{c_{d}}{d-1}\mathrm{e}^{(d-1)D_{i}})_{i\geq 1}$ is a homogeneous Poisson point process (PPP) on $ \mathbb{R}_{+}$ of unit intensity. The process $\mathbf D$ is independent of $\boldsymbol \Theta$.
 \end{theorem}

We call $\mathrm{Vor}( \boldsymbol{\Theta},  \mathbf{D})$ the \dfn{ideal Poisson--Voronoi tessellation} of $ \mathbb{H}_d$ and denote it by $ \mathcal{V}_{d}${; we still need to prove that it is a tessellation, which we will do in \Cref{s.mobius} when we prove \cref{thm:decomposition}}.

 \begin{proof} 
 
 Let $\mathbf{X}^{(\lambda)}=(X_i^{(\lambda)} \, \ST \, i\geq 1 )$ be a PPP of intensity\footnote{The normalization of the intensity as $\lambda^{d-1}$ is here to ensure that the closest point to $\origin$ in $ \mathbf{X}^{(\lambda)}$ is roughly at distance $\log(1/\lambda)$ as $\lambda \to 0$ for every $d \geq 2$.} $\lambda^{d-1} \cdot \mathrm{Vol}_{ \mathbb{H}_d}$ on the hyperbolic space $\Hyp{d}$, where as usual the points are ranked by increasing hyperbolic distance to the origin, $\origin$. In particular, almost surely $\mathrm{d}_\Hyp{d}(\origin,X_i^{(\lambda)})$ is strictly increasing in $i \geq 1$. Our goal is to prove convergence of the corresponding Voronoi tessellations towards a limiting ideal diagram as $\lambda \to 0$. To apply Theorem \ref{lem.convvor}, one needs to check convergence of points towards the ideal boundary and convergence of proto-delays. Both turn out to be very easy:

\paragraph{Convergence of proto-delays.} By the mapping theorem for Poisson processes \cite[Theorem 24.16]{klenke2013probability}, it is simple to see that as $\lambda \to 0$, we have 
$$ \frac{\mathrm{d}_\Hyp{d}(\origin,X_1^{(\lambda)})}{|\!\log \lambda|} \xrightarrow[\lambda \downarrow 0]{(\mathbb{P})} 1,$$ and so it is natural to introduce the shifted distances, i.e., the \dfn{{proto-}delays}, as
\begin{equation}\label{def.delays}D_i^{\left(\lambda\right)} \coloneqq \mathrm{d}_\Hyp{d}(\origin,X_i^{(\lambda)})- \log(1/\lambda), \quad  i\geq 1. \end{equation}

The mapping theorem for Poisson processes readily entails that, 
as $\lambda \downarrow 0$, the set of increasing proto-delays $( D_i^{(\lambda)})_{ i\geq 1}$ converges in law to the increasing points  $(D_{i})_{i\geq 1}$ of a Poisson point process on $\mathbb{R}$ with intensity measure $$c_d \,  \mathrm{e}^{(d-1)s} \,\mathrm{d}s.$$
Equivalently, the set $(  \frac{c_{d}}{d-1}\mathrm{e}^{(d-1)D_i})_{ i\geq 1 }$ is a rate-1 homogeneous Poisson process on $ \mathbb{R}_+$.
These statements can be proved via the following simple calculation. Recall that the volume growth function for hyperbolic space is given by $$f_{d}({u}) \coloneqq \mathrm{Vol}_{\Hyp{d}}\big(B_\Hyp{d}(\origin,{u})\big)=\Omega_d \int_0^{{u}} \left(\sinh{\rho}\right)^{d-1} \;  \mathrm{d}\rho, $$ where $\Omega_d\coloneqq2\frac{\pi^{d/2}}{\Gamma(\frac{d}{2})}$ is the Euclidean volume of $\mathbb{S}_{d-1}$. A straightforward calculation shows that for all $x,y \in \mathbb{R}$, we have
$$
\lim_{\lambda \downarrow 0} \lambda^{d-1} \left ( f_{d}\Bigl(x+\log{\frac{1}{\lambda}}\Bigr)-f_{d}\Bigl(y+\log{\frac{1}{\lambda}}\Bigr) \right) = \frac{2^{1-d}}{d-1}\Omega_d \bigl( \mathrm{e}^{(d-1)x}- \mathrm{e}^{(d-1)y} \bigr),$$ whence both claims follow from the mapping theorem for Poisson processes.

\paragraph{Convergence to ideal points.} Let us return to the Poisson point process of nuclei $(X_i^{(\lambda)}\ST i\geq 1  )$ and adopt the ball model $ \mathbb{B}_{d}$ for hyperbolic space. By rotational symmetry, it is clear that conditionally on the distance process $\bigl(\mathrm{d}_\Hyp{d}(\origin,X_i^{(\lambda)}) \ST i \geq 1\bigr)$, or equivalently conditionally on the proto-delays, the angles $(\Theta_{i}^{(\lambda)} \ST i \geq 1)$ of the points $X_{i}^{(\lambda)}$ in the ball model $ \mathbb{B}_{d}$ are i.i.d.\ uniform over $ \mathbb{S}_{d-1}$. It is well known that the Gromov boundary of $ \mathbb{B}_{d}$ is $\mathbb{S}_{d-1}$, or equivalently that a divergent sequence of points $x_{i} \in \mathbb{B}_{d}$ with angles $ \theta_{i}$ converges to $ \theta \in \partial \mathbb{B}_{d} = \mathbb{S}_{d-1}$ if and only if $ x_{i} \to \infty$ and $\theta_{i} \to \theta$~(see
 \cite[Example 8.11, page~265]{BridsonHaefliger}). 

\medskip

 We can now conclude the proof of the theorem. By the Skorokhod embedding theorem, we can couple, on the same probability space, all the Poisson point processes $( X_{i}^{(\lambda)} \ST i \geq 1)$ for $\lambda >0$ in such a way that 
for every $i \geq 1$, the proto-delays and the angles converge almost surely: $D_{i}^{(\lambda)} \to D_{i}$ and $\Theta_{i}^{(\lambda)} \to \Theta_{i}$. In particular, $ D_{i} \to \infty$ almost surely and the convergence of angles implies the convergence towards the Gromov boundary. Using the fact that all $\Theta_i$ and $D_i$ have continuous distributions, it is an easy matter to check that the limiting ideal diagrams are a.s.\ nondegenerate. We can then apply~\cref{lem.convvor} to get the desired convergence. \end{proof}

\subsection{Computing the separations} \label{sssec:compsep}
Having identified the limiting tessellation, we now aim to look more closely into its properties. The final piece needed to prove Theorems \ref{t.weighted} and \ref{thm:decomposition} is a concrete way of computing distances from extended ideal points. For this, we will see that it is more practical to first perform a change of variable and consider the exponentials of the delays.\medskip

 Recall the definition of an extended ideal point $(\theta, \delta) \in \widehat{\partial \mathbb{H}_d} = \partial \mathbb{H}_d\times \mathbb{R}$ where $\theta$ is a boundary point and $\delta \in \mathbb{R}$ is a delay. It will be convenient in the rest of the manuscript to consider the following equivalent description after taking the exponentials of the delays: Introduce $ \widetilde{ \partial \mathbb{H}_d } \coloneqq \partial\mathbb{H}_d\times \mathbb{R}_{+}$, which we call the \dfn{corona}, and consider the image of the extended ideal points 
\[
(\theta, \delta) \in \widehat{\partial \mathbb{H}_d} \mapsto \bigl( \theta, \frac{c_{d}}{d-1}\mathrm{e}^{(d-1) \delta} \bigr) \in  \widetilde{ \partial \mathbb{H}_d },
\]
where $c_d$ is as in \eqref{e.defc_d}. We call the first coordinate of a point in the corona its \dfn{angle} and the second coordinate its \dfn{radius}. We will also use the letters $r, r_i, R, R_i$ for radii and the letter ``$\delta$'' or ``$D$" for delays. 
It follows from the first part of the proof of \cref{cor.thetaD} that $(\theta _i, \frac{c_d}{d-1} \ue^{(d-1)D_i})_{i \geq 1}$ is a Poisson point process in the corona with intensity
\begin{equation}  \label{eq:defmud}
\mu_{d} \coloneqq  \mathrm{Unif} \otimes \mathrm{Leb};
\end{equation}
the points of this PPP will be called \dfn{ideal nuclei}. In these new coordinates where $r = \frac{c_d}{d-1} \ue^{(d-1)\delta}$, we introduce the \dfn{(exponential) separation} between $z \in \mathbb{H}_d$ and a point $(\theta,r)$ in the corona in terms of the linear separation between $z$ and an ideal point $(\theta,\delta)$ as follows:
\begin{equation}\label{eq.defsep}
\mathbf{d}\bigl(z, (\theta, r)\bigr)
\coloneqq
\frac{c_d}{d-1} \exp\bigl\{(d-1) \ud \bigl(z, (\theta, \delta)\bigr)\bigr\}.
\end{equation}
This monotone{, strictly} increasing transformation preserves the inequality in \cref{d.IVT} and hence the associated ideal Voronoi tessellation.
 See \Cref{f.corona} for an example of the Poisson point process on the corona and its associated Voronoi and Delaunay tessellations.
In the latter tessellation, one associates to every Voronoi vertex $v$ the ideal simplex formed from the $d+1$ angles of the ideal nuclei whose separation from $v$ is smallest{. The resulting collection of ideal simplices forms the Delaunay tessellation}; complete details are given in \cref{s.delaunay}.

\begin{figure}[htp] 
\includegraphics[width=\textwidth]{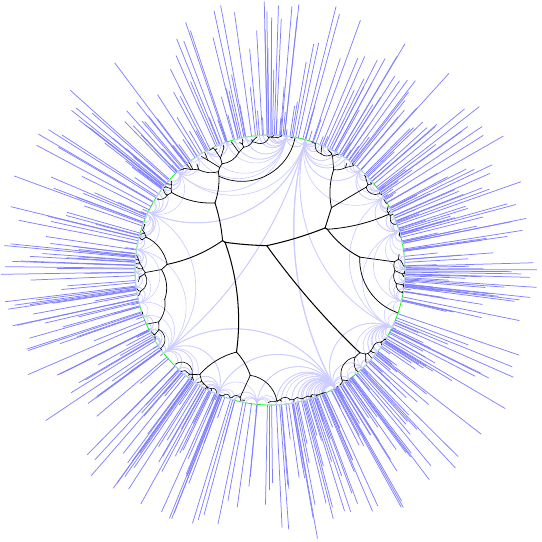}
\caption{Portions of the ideal Voronoi $ \mathcal{V}_d$ (in black) and Delaunay tessellations (in light blue, where ideal nuclei are joined if the corresponding tiles are adjacent in $ \mathcal{V}_d$), with the corona showing the first 500 ideal nuclei. The radii of the nuclei are scaled linearly to $[1.02, 2.02]$ for graphical reasons. Each point $(\theta, r)$ in the corona is joined by a line segment to $\theta$ in the ideal boundary.
}
		\label{f.corona}
\end{figure}

\subsubsection{Separations to the ideal nuclei via the Poisson kernel} Remember that the ideal boundary of $\mathbb{B}_d$ is identified to the $(d-1)$-dimensional sphere, $ \mathbb{S}_{d-1}$. In this model, we recall the expression of the (hyperbolic)  \dfn{Poisson kernel} that gives the density of the harmonic measure on $\mathbb{S}_{d-1}$ seen from a point $z \in \mathbb{B}_d$: For $z \in  \mathbb{B}_d$ and $\theta\in \mathbb{S}_{d-1}$, write
\[
K(z, \theta)
\coloneqq
\Bigl(\frac{1 - |z|^2}{|z - \theta|^2}\Bigr)^{d-1};
\]
see \cite[Definition 5.1.1]{Stoll}. For the upper half-space model, given $z \in \uhs_d$ and $\theta \in \R^{d-1}$ , the kernel is
\[
K(z, \theta)
\coloneqq
\frac1{c_d}\Bigl(\frac{z_d}{|z - \theta|^2}\Bigr)^{d-1}; 
\]
see \cite[(5.6.1)]{Stoll}. However, a slightly modified kernel is more useful to us in $\uhs_d$:
\[
\Kh(z, \theta)
\coloneqq
\rcases
{\displaystyle \Bigl(\frac{z_d\bigl(1+|\theta|^2\bigr)}{|z - \theta|^2}\Bigr)^{d-1} &\mbox{if }\theta \ne \infty, \\[10pt]
\displaystyle (z_d)^{d-1}  &\mbox{if }\theta = \infty.} \\
\]
These kernels occur naturally when computing separations from ideal nuclei:

\begin{lem}[\textsc{Separation in coordinates}]\label{lem:evaldist}

The separation from $z \in \HH_d$ to $(\theta, r) \in \widetilde{ \partial \mathbb{H}_d }$ satisfies
\[
\mathbf{d}\bigl(z, (\theta, r)\bigr)
=
\rcases
{\displaystyle \frac{r}{K(z, \theta)} &\mbox{in the ball model,}\\[10pt]
\displaystyle \frac{r}{\Kh(z, \theta)} &\mbox{in the upper half-space model.}\\
}
\]
\end{lem}

Thus, if $ (\theta_1, r_1),(\theta_2, r_2) $ are two points of the corona, then a given point $z \in \uball_d$ has a smaller separation to $(\theta_1, r_1)$ than to $(\theta_2, r_2)$ iff 
 \begin{equation} \label{eq:closernuclei} \frac{r_1}{K(z, \theta_1)} \leq \frac{r_2}{ K(z, \theta_2)}
 \quad\mbox{iff}\quad
 r_1|z-\theta_1|^{2(d-1)} \le  r_2|z-\theta_2|^{2(d-1)},  \end{equation} 
whereas for $z \in \uhs_d$, the condition is
 \begin{equation} \label{eq:closernucleiUHS} \frac{r_1}{\Kh(z, \theta_1)} \leq \frac{r_2}{ \Kh(z, \theta_2)}
 \quad\mbox{iff}\quad
 r_1\Bigl(\frac{|z-\theta_1|^2}{1+|\theta_1|^2}\Bigr)^{d-1} \le  r_2\Bigl(\frac{|z-\theta_2|^2}{1+|\theta_2|^2}\Bigr)^{d-1},  \end{equation} 
where the fraction is interpreted as $1$ if $\theta_i = \infty$.

\rproof
Since $\ue^{\arcsinh t} -\ue^{-\arcsinh t}=2t$, we have  $\ue^{\arcsinh t} \sim 2t$ as $t \to \infty$. Therefore, if $\uball_d \ni x_n \to \theta \in \sph$ as $n \to\infty$, we have for all $y \in \uball_d$ that 
\rlabel e.asympdist
{\exp\bigl\{\ud_{\HH_d}(x_n, y)\bigr\}
\sim
\frac{4|\theta - y|^2}{\bigl(1 -|\theta|^2\bigr)\bigl(1-|y|^2\bigr)}.}
Recall from \cref{s.horo} that the horofunction $d_\theta(y) \coloneqq \lim_{n\rightarrow \infty} \bigl(\ud_{\HH_d}(x_n, y) - \ud_{\HH_d}(x_n, \origin)\bigr)$. Applying \eqref{e.asympdist} to $y = z$ and to $y = \origin$ yields
\rlabel e.evallim
{\ue^{\ud_\theta(z)}
=
\frac{|z-\theta|^2}{1-|z|^2}.}
Using the definitions of $\mathbf{d}\bigl(z, (\theta, r)\bigr)$ and of $K$ completes the proof.

The proof is almost the same in $\uhs_d$: For $\uhs_d \ni x_n \to \theta \in \sph$ as $n \to\infty$, we have for all $y \in \uhs_d$ that 
\rlabel e.asympdistUHS
{\exp\bigl\{\ud_{\HH_d}(x_n, y)\bigr\}
\sim
\rcases
{\displaystyle \frac{|\theta - y|^2}{x_{n, d}\, y_d} &\mbox{if }\theta \ne \infty, \\[10pt]
\displaystyle \frac{x_{n, d}}{y_d} &\mbox{if }\theta = \infty.} \\
}
Here, when $\theta = \infty$, we may take the first $d-1$ coordinates of $x_n$ to be fixed (not changing with $n$).
Applying this to $y = z$ and to $y = \origin=(0_{d-1},1)$ yields
\rlabel e.evallimUHS
{\ue^{\ud_\theta(z)}
=
\rcases
{\displaystyle \frac{|z-\theta|^2}{z_d\bigl(1+|\theta|^2\bigr)} &\mbox{if }\theta \ne \infty, \\[10pt]
\displaystyle \frac1{z_d} &\mbox{if }\theta = \infty.} \\
}
Using the definitions of $\mathbf{d}\bigl(z, (\theta, r)\bigr)$ and of $\Kh$ completes the proof.
\Qed

It may appear incongruous that $\Kh$ rather than $K$ appears in the upper half-space model. The reason is that the radius $r$ is based on the origin, $\origin$, which is not as natural a reference point for $\uhs$ as it is for $\uball$. A more natural reference point for $\uhs$ is the ideal point, $\infty$. One can, in fact, couple the PPPs of intensity $\lambda > 0$ of nuclei in $\uhs$ by dilating from $\infty$: see \cref{l.UHSdilate} below to see how a more natural expression ensues.

\medskip

The following corollary will be useful for the proof of~\cref{thm.superposition}.
 
\begin{cor} \label{prop:bissec} Consider two ideal nuclei, $( \theta_{1}, r_{1})$ and $( \theta, r)$. In the upper half-space model $\mathbb{U}_{d}$, if $\theta_{1} = \infty$, then the $(d-1)$-hyperplane of points at equal separation from the two ideal nuclei is the Euclidean hemisphere centered at $\theta$ with radius 
$$ \sqrt{1 + |\theta|^{2}} \,  \Bigl(\frac{r_{1}}{r}\Bigr)^{\frac{1}{2(d-1)}}.$$
\end{cor}

\rproof
This is immediate from \eqref{eq:closernucleiUHS}.
\Qed

\subsubsection{Almost-sure limit via dilations}

The convergence in \cref{cor.thetaD} is a convergence in law, but actually there are explicit ways to realize the coupling towards the end of the proof for different $\lambda $  via dilation as follows:

\begin{remark}[\textsc{Dilations from $\origin$ in $\uball_d$}] \label{remark_dilation} Let $(\boldsymbol{\Theta}, \mathbf{R}) = \bigl(( {\Theta}_{i}, {R}_{i}) \ST i \geq 1\bigr)$ be a Poisson point process on $  \mathbb{S}_{d-1} \times \mathbb{R}_{+}$ with intensity $\Unif \otimes \Leb$. Let $v_d(r)$ be the hyperbolic volume of the ball of Euclidean radius $r$ centered at the origin in the ball model of $  \mathbb{H}_{d}$. For $\lambda >0$, define the point process $ \tilde{ \mathbf{X}}^{(\lambda)}$ as 
$$ \mathrm{Angle}(\tilde{X}_{i}^{(\lambda)}) \coloneqq {\Theta}_{i} \quad \mbox{ and }\quad \mathrm{d}_{ \mathrm{Euc}}( \mathbf{0}, \tilde{X}_{i}^{(\lambda)}) \coloneqq v_d^{-1}\bigl({R}_{i}/\lambda^{d-1}\bigr).$$
Then it is straightforward to check that for each $\lambda>0$, the point process $ \tilde{\mathbf{X}}^{ (\lambda)}$ is Poisson with intensity $\lambda^{d-1}$ and for which we have the almost sure convergence $ \tilde{\mathbf{X}}^{ (\lambda)} \to (\boldsymbol{\Theta}, \mathbf{D})$, where $  \mathbf{R} =  \frac{c_{d}}{d-1} \exp\{ (d-1)\mathbf{D}\}$. In particular, by the proof of Theorem \ref{cor.thetaD}, its Voronoi tessellation converges almost surely towards $ \mathcal{V}_{d}$; see \Cref{f.successiveVor}.
\end{remark}

\begin{figure}[!h]
 \begin{center}
 \includegraphics[height=5cm]{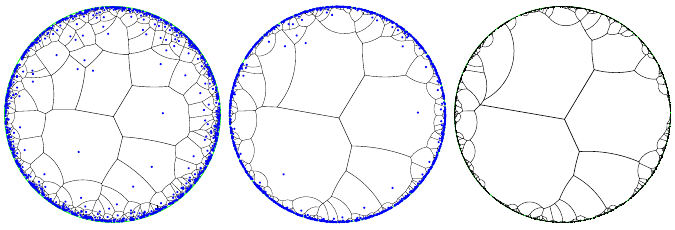}
 \caption{ Left to right: Poisson--Voronoi tessellations of the hyperbolic plane (in the unit disk model) with decreasing intensity coupled via dilations from $\origin$ using 1000 nuclei. Their limit (on the right with 10,000 ideal nuclei) is $ \mathcal{V}_2$, the  \emph{ideal Poisson--Voronoi tessellation} of the hyperbolic plane.}
 \label{f.successiveVor}
 \end{center}
 \end{figure}
 
 \begin{lem}[\textsc{Dilations in $\uhs_d$}]\label{l.UHSdilate}
Let $\XX$ be a Poisson point process with intensity 1 in $\HH_d$. In the upper half-space model, the intensity measure, $\ud \Vol_{\HH_d}$, is $\ud x_1 \,\ud x_2 \cdots \,\ud x_{d-1}  \, x_d^{-d} \,\ud x_d$. Define $\XX^{(\lambda)}$ by mapping $(x_1, \dots, x_d) \mapsto (x_1, \dots, x_{d-1}, \lambda x_d)$, which we write as $x \mapsto x^{(\lambda)}$. Then $\XX^{(\lambda)}$ is a PPP of intensity $\lambda^{d-1}$. As $\lambda \downarrow 0$, the Voronoi tessellation associated to $\XX^{(\lambda)}$ a.s.\ converges to the diagram given by \eqref{e.coordCells}.
\end{lem}

Note that in contrast to the ball model, the process $N$ used for the upper half-space model in \cref {t.weighted} does not use the radii as defined by using an origin in $\HH_d$.

\rproof
The Poisson mapping theorem gives that $\XX^{(\lambda)}$ is a PPP of intensity $\lambda^{d-1}$.
By \eqref{e.asympdistUHS}, for each $x \in \XX$ and all $z \in \uhs_d$, we have 
\[
\exp\bigl\{\ud_{\HH_d}(x^{(\lambda)}, z)\bigr\}
\sim
\displaystyle \frac{|\theta - z|^2}{\lambda x_{d}\, z_d} ,
\]
where $\theta \coloneqq (x_1, \dots, x_{d-1})$, showing that the proto-delays converge a.s. Write $p(x) \coloneqq \bigl(\theta, \frac{1}{(d-1)} x_d^{1-d}\bigr)$. Since $\XX^{(\lambda)}$ converges to $\cells_d$ by \cref{cor.thetaD}, it remains to show that $\bigl\{p(x) \ST x \in \XX\bigr\}$ is a Poisson process of intensity $\Leb \otimes \Leb$. This follows from the Poisson mapping theorem.
\Qed

We can now conclude the proof of \cref{t.weighted}:

 \begin{proofof}{\cref{t.weighted}}
Knowing that the limit $\cells_d$ exists and is described by \cref{cor.thetaD}, the proof follows by appeal to \eqref{eq:closernuclei} in the case of the ball model.
The description in the case of the upper half-space model follows from \cref{l.UHSdilate}.
\end{proofof}

 \subsection{First properties of $ \mathcal{V}_{d}$}

Let us establish the first few properties of $ \mathcal{V}_d$ and Theorem \ref{thm:decomposition}. The faces of the ideal Voronoi cells are totally geodesic. This follows from the equivalent fact for the Poisson--Voronoi tessellations with positive intensity. 

\subsubsection{Topological properties}

Our ideal Voronoi diagrams share the same a.s.\ local properties as standard Poisson--Voronoi tessellations in dimension $d$:

\begin{prop}\label{p.locfinite} Almost surely, the diagram $ \mathcal{V}_{d}$ is a locally finite, face-to-face, and normal tessellation.
\end{prop}

\begin{proof} The argument is almost the same as for standard Poisson--Voronoi tessellations: 
Given a fixed ball in $\mathbb{B}_d$, the inequality \eqref{eq:closernuclei} combined with the fact that the ideal radii tend to infinity a.s.\ shows that there are only finitely many ideal nuclei that could have smaller separation from any point in that ball than the separation from the ideal nucleus with smallest radius. Therefore, only finitely many ideal Voronoi cells intersect that ball. Since the cells are convex, so is each intersection with the ball, as well as all finite intersections among the cells. As a consequence, the diagram has only finitely many faces in the ball, i.e., is locally finite. It then follows from \cref{cor.thetaD} that $\cells_d$ is a.s.\ a tessellation.

We next prove that $\cells_d$ is face-to-face.
Indeed, suppose that $X$ and $Y$ are two distinct ideal nuclei whose cells $C(X)$ and $C(Y)$ have nonempty intersection.
Write $H$ for the geodesic hyperplane of points at equal separation from $X$ and $Y$.
Then 
$C(X) \cap C(Y) \subseteq C(X) \cap H$. Since $C(X) \cap H \ne \varnothing$, we have that $H$ is a supporting hyperplane of $C(X)$, whence $C(X) \cap H$ is a face of $C(X)$ by definition of face. It remains to prove that $C(X) \cap C(Y) = C(X) \cap H$. If not, then choose $w \in C(X) \cap H \setminus C(Y)$. Since $w \notin C(Y)$, there is some ideal nucleus $Z$ with $\bud(w, Z) < \bud(w, Y)$. However, $\bud(w, X) \le \bud(w, Z)$ because $w \in C(X)$, and these two inequalities contradict $w \in  H$, i.e., $\bud(w, X) = \bud(w, Y)$.

To show normality, 
recall that equality in \eqref{eq:closernuclei} is the condition that $z$ has equal separation from two points in the corona. Given $x_i \in \corona$, let $S(x_1, \ldots, x_m)$ be the set of points $z \in \uball_d$ whose separation from $x_i$ is the same for all $i \in [1, m]$.  When $x_1 \ne x_2$,
the set $S(x_1, x_2)$ is a totally geodesic hyperplane in $\HH_d$, contained in a Euclidean sphere of dimension $d-1$.\footnote{For $d = 2$, the fact that this locus in the plane is a circle is due to Apollonius.} 
It is straightforward to see that for $m > 2$, the set $S(x_1, \ldots, x_m)$ is either empty or contained in a sphere of dimension one less than a sphere containing $S(x_1, \ldots, x_{m-1})$ for $\mu_d^m$-a.e.\ $(x_1, \ldots, x_m) \in \corona^m$ with $\mu_d$ defined as in \eqref{eq:defmud}. 
Therefore, $\varnothing \ne S(x_1, \ldots, x_m) = S(x_1, \ldots, x_{m-1})$ only on a set of $\mu_d^m$-measure 0.

Consider now the ideal nuclei with radius at most $\rho$. Given that there are $m$ such ideal nuclei, their distribution in random order is proportional to the restriction of $\mu_d^m$ to $m$-tuples of points in the corona with radius at most $\rho$. Thus,
a.s., for any set $\{Y_1, \ldots, Y_m\}$ of $m$ ideal nuclei of radius at most $\rho$, the totally geodesic plane $S(Y_1, \ldots, Y_m)$ has dimension $d+1-m$ or is empty.
Because this holds for all $\rho$, the same holds without this restriction on the radii.

Finally, suppose that $ \mathfrak{s}$ is a $k$-face of $\cells_d$ with $0 \le k \le d-1$. Because $\cells_d$ is face-to-face a.s., there exist distinct ideal nuclei $Y_1, Y_2, \ldots, Y_m$ the intersection of whose cells is $ \mathfrak{s}$. By what we proved in the last paragraph, $m \le d+1$. Thus, $ \mathfrak{s}$ affinely spans $S(Y_1, \ldots, Y_m)$ and $d+1-m = k$.
\end{proof}

In fact, we have quantitative control over the number of faces that appear in a given ball:

\begin{prop}[\textsc{Tail bounds for crowding}]\label{p.tailbound}
Let $B_u = B_u(\origin)$ be the ball centered at $\origin$ with hyperbolic radius $u$. Write $\alpha \coloneqq 1/(2\ue^{2(d-1)u} - 1)$.
The probability that the number of $k$-faces of\/ $\cells_d$ that intersect $B_u$ is at least $n$ is at most $2(1-\alpha)^{n^{1/(d+1-k)}-1} < 2\ue^{-\alpha (n^{1/(d+1-k)}-1)}$ for $0 \le k \le d$ and $n \ge 0$.
\end{prop}


\rproof
If $z$ is a point of $B_u$ that belongs to the cell of an ideal nucleus $(\theta, r) \ne (\Theta_1, R_1)$, then the separation of $z$ from $(\theta, r)$ is at most that from $(\Theta_1, R_1)$, whence
\[
R_1^{1/2(d-1)} (1+a)
\ge
R_1^{1/2(d-1)} |z - {\Theta}_1|
\ge
r^{1/2(d-1)} |z - \theta|
\ge
r^{1/2(d-1)} (1 - a),
\]
where $a \coloneqq \tanh (u/2)$ is the Euclidean radius of $B_u$. That is, $r \le R_1 \bigl[(1+a)/(1-a)\bigr]^{2(d-1)} = R_1 \ue^v$, where $v \coloneqq {2(d-1)u}$. Hence, the number $N$ of ideal nuclei with radius in $(R_1, R_1 \ue^v]$ is at least the number of cells---other than the cell of $\origin$---that intersect $B_u$. Now, given $R_1$, the conditional distribution of $N$ is Poisson with parameter $R_1 (\ue^v - 1)$. Thus,
\[
\EE\bigl[\ue^{t N}\bigr]
=
\EE\bigl[\ue^{ R_1 (\ue^v-1)(\ue^t-1)}\bigr]
=
\frac{1}{1 - (\ue^v-1)(\ue^t-1)}
\]
when $(\ue^v-1)(\ue^t-1) < 1$. Choosing $t$ so that this product is $1/2$ yields the bound
\[
\PP[N \ge n]
=
\PP[\ue^{tN} \ge \ue^{tn}]
\le
2 (1 - \alpha)^n
<
2 \ue^{-\alpha n}
\]
by Markov's inequality.

By normality, each $k$-face is the intersection of $d+1-k$ cells. If such a face intersects $B_u$, then so does each of the cells whose intersection forms that face. If the number of cells intersecting $B_u$ is $m$, then there are at most $m^s$ intersections of $s$ of them. Therefore, if at least $n$ $k$-faces intersect $B_u$, then at least $n^{1/(d+1-k)}$ cells intersect $B_u${, and so $N \ge n^{1/(d+1-k)}- 1$}. This gives the result desired.
\Qed

We now prove that each cell has a unique end. In fact, we give a quantitative bound on how far the cell of the origin can extend away from the ideal nucleus {$(\Theta_1, R_1)$} in the opposite direction within a spherical cap of angle $\alpha \in (0, 2\pi)$ subtended at $\origin$:
\begin{prop}[\textsc{Tail bound for one end}]\label{p.tailEnd}
Let $u > 0$, $\alpha \in (0, 2\pi)$, $\gamma \coloneqq \bigl((\ue^u - 1) \cos (\alpha/4)\bigr)^{2(d-1)}$ and $\beta \coloneqq (\gamma -1)^{{+}}$. Let $A$ be the portion of the unit sphere intersected by the cone $C$ of opening angle $\alpha \in (0, 2\pi)$ subtended at $\origin$ in the antipodal direction to $\Theta_1$. Write $\alpha'$ for the $(d-1)$-volume of $A$. The chance that the zero cell contains some $z$ at hyperbolic distance $u$ from $\origin$ and within the cone $C$ is at most $1/(1+\alpha'\beta)$. Therefore, all cells have only one end a.s. 
\end{prop}

\begin{figure}[!hbtp]
\centering
\begin{overpic}[width=.5\textwidth]{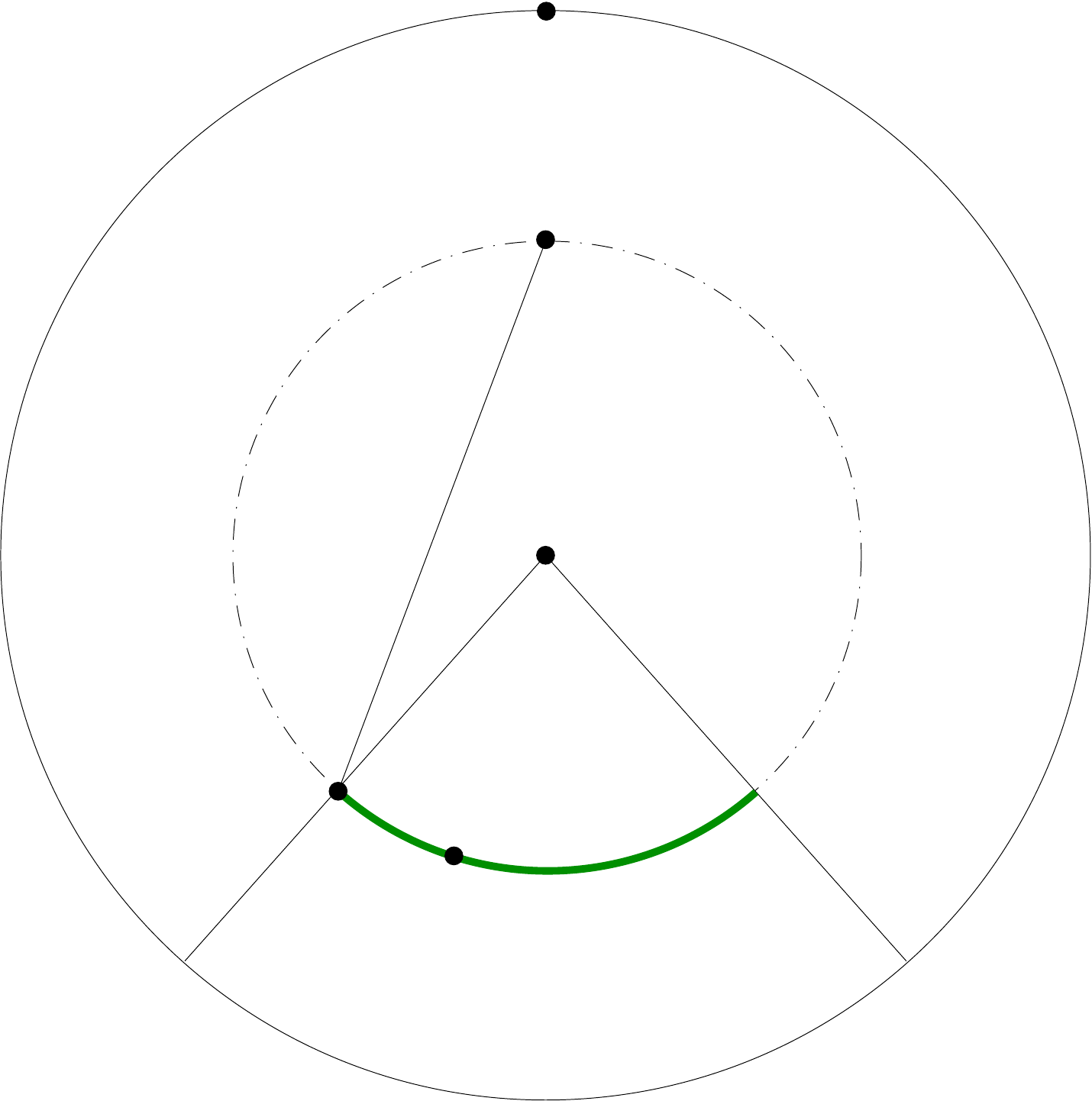}
   \put (48,52) {$\origin$}
   \put (48,102) {$\Theta_1$}
   \put (42,24) {$z$}
   \put (50,16) {$S$}
   \put (48,81) {$y$}
   \put (26,27) {$x$}
   \put (48,40) {$\alpha$}
   \put (63,36) {$a$}
   \put (48,-4) {$A$}
\end{overpic}\\[10pt]
    \caption{We have $|z - \Theta_1| \ge |x - \Theta_1| > |x - y| = 2a \sin \bigl((\pi - \alpha/2)/2\bigr) = 2a \cos(\alpha/4)$.}
    \label{f.tailbound}
\end{figure}

\rproof 
Let $S$ be the intersection of the sphere about $\origin$ of hyperbolic radius $u$ with the cone $C$. For $z \in S$, we have ${R_1|z - \Theta_1|^{2(d-1)} \ge R_1} \bigl(2 a \cos (\alpha/4)\bigr)^{2(d-1)}$, where $a \coloneqq \tanh (u/2)$ is the Euclidean radius of the sphere: see \Cref{f.tailbound}. Note that the closest point $z \in S$ to $\theta \in A$ has Euclidean distance $1-a$ from $\theta$. Therefore, if there is some $z \in S \cap \ocell_d$, then for all $i$ with ${\Theta}_i \in A$, we have by \eqref{eq:closernuclei} and the preceding inequality that
\[
{R_1} \bigl(2 a \cos (\alpha/4)\bigr)^{2(d-1)}
\le
{R_i} (1-a)^{2(d-1)}.
\]
The radii of the PPP on the corona restricted to ideal nuclei with angles in $A$ has intensity $\alpha'$, so the chance, given $R_1$, that the above inequality holds for all $i$ with $\Theta_i \in A$ is $\exp\bigl\{- \alpha' {R_1} \beta\bigr\}$ because $2a/(1-a) = \ue^u - 1$. The unconditional probability that this inequality holds is therefore the bound claimed.

Since this bound tends to 0 as $u \to\infty$, it follows that $\ocell_d$
a.s.\ contains no limit point in $A$. Since $\alpha$ can be chosen as close to $2\pi$ as we wish, we conclude that $\ocell_d$ has only the end ${\Theta}_1$. There are only countably many cells, whence the same holds for all cells.
\Qed

\subsubsection{M\"obius action on the corona via the Poisson kernel}\label{s.mobius}

In \eqref {e.horoaction}, we extended isometries to the extended boundary in a way that preserves separations. 
Converting coordinates to the corona and using \eqref{e.evallim}, we find that an isometry $ \phi \colon \mathbb{B}_d \to  \mathbb{B}_d$ of hyperbolic space acts on the corona preserving separation via
\rlabel e.action
{\phi(\theta, r) \coloneqq \Bigl(\phi(\theta), r/K\bigl(\phi^{-1}(\origin), \theta\bigr)\Bigr).
}

 \begin{lem} \label{lem:mobius} The action of the M\"obius group $ \mob_d$ on the corona $ \widetilde{ \partial \mathbb{B}}_d$ defined by \eqref{e.action} is a transitive group action that leaves the measure $\mu_d $ of \eqref{eq:defmud} invariant. In particular, the law of\/ $ \mathcal{V}_d$ is invariant under isometries.
 \end{lem}
 \begin{proof} This can be checked directly on the corona using the equivariance of harmonic measure (i.e., $\phi_*\bigl(K(z, \cdot) \Unif\bigr) = K\bigl(\phi(z), \cdot\bigr) \Unif$), but let us prove it using  PPPs: 
 Let $ \mathbf{X}^{(\lambda)}$ be a PPP on $ \mathbb{B}_d$ with intensity $\lambda$. Let $\mathbf{R}^{(\lambda)}$ be the corresponding proto-delays, exponentiated as in the coordinates of the corona. It follows from \Cref{cor.thetaD} that $(\mathbf{X}^{(\lambda)}, \mathbf{R}^{(\lambda)})$ converges as $\lambda \to 0$ towards a PPP on the corona with intensity $\mu_d$.  On the other hand, the law of $(\mathbf{X}^{(\lambda)}, \mathbf{R}^{(\lambda)})$ is invariant under isometries. It follows that $ \mob_d$ acts in a way that leaves the measure $\mu_d$ of \eqref{eq:defmud} invariant. It is easy to check that this action is indeed transitive. \end{proof}
 
The measure $\mu_d$ is therefore the Haar measure on $\corona$ for the action of $ \mob_d$. Also, the subgroup of $\mob_d$ that stabilizes $(\theta, r)$ is the subgroup that fixes the horosphere at $\theta$ that passes through the origin (no matter the value of $r$).

\begin{proofof}{\cref{thm:decomposition}}
 By \cref{p.tailEnd}, the cells are unbounded with one end a.s. Since the ends are distinct, it follows that all lower-dimensional faces are bounded a.s. The remaining properties were proved in
\cref{p.locfinite} and \cref{lem:mobius}.
\end{proofof}

The \dfn{separation field}, whose value at $z$ is the minimum separation between $z$ and all ideal nuclei, determines the ideal Voronoi and Delaunay tessellations. See \Cref{f.field} for a sample. 
A measurable map that intertwines two actions of a given group $\Gamma$ is called $\Gamma$-equivariant. When the actions preserve probability measures, the second action is then called a $\Gamma$-equivariant factor of the first. When the map is invertible with measurable inverse, then the two actions are called $\Gamma$-equivariantly isomorphic.

\begin{figure}[htp] 
    \centering
\includegraphics[width=.7\textwidth]{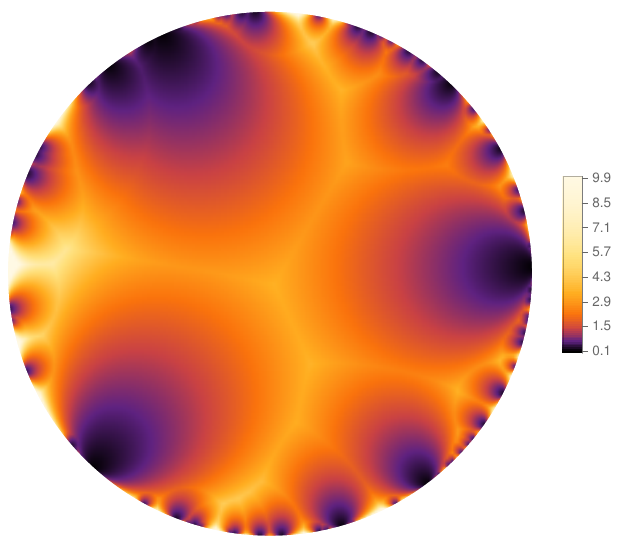}
\caption{Part of the separation field in $\HH_2$. The boundaries of the cells of the ideal Voronoi tessellation are formed by the singular points of the field (where it is not differentiable).} 
		\label{f.field}
\end{figure}

Our next result uses \cref{p.tailEnd}, which proved the first bullet point of \cref{thm:decomposition}.

\begin{theorem}[\textsc{Isomorphisms and factors}]  \label{t.factor}
The Poisson point process of ideal nuclei on the corona, the separation field, and the IPVT are all $\mob_d$-equivariantly isomorphic. 
\end{theorem}

\rproof
It suffices to define measurable, equivariant maps $\fNS$, $\fSV$, and $\fVN$ from the ideal nuclei to the separation field, from the separation field to the Voronoi tessellation, and from the Voronoi tessellation to the ideal nuclei, respectively, such that $\fVN \circ \fSV \circ \fNS$ is the identity a.s. In fact, we need not prove that $\fVN$ is equivariant, because the composition being the identity forces it to be so.

For a simple counting measure $N$ on $\corona$ with radii tending to infinity, let $\fNS(N)$ be the separation field of $N$, i.e., $\fNS(N)(z) \coloneqq \min \{ r/K(z, \theta) \sut N(\theta, r) = 1\}$. Because $\mob_d$ preserves separations \eqref{e.preserve-sep}, $\fNS\bigl(\phi_*(N)\bigr)(z) = \fNS(N)\bigl(\phi^{-1}(z)\bigr)$ for all $\phi\in \mob_d$ and $z \in \HH_d$, i.e., $\fNS\bigl(\phi_*(N)\bigr) = \phi\bigl(\fNS(N)\bigr)$, which is equivariance. This shows that the separation field is a factor of the ideal nuclei via $\fNS$.

The next step is trivial: \cref{d.IVT} defines the map $\fSV$, which \eqref{e.preserve-sep} shows is equivariant.

Finally, consider $\cells_d$. Each cell has one end a.s.\ by \cref{p.tailEnd}, and that end is the angle of its ideal nucleus, so defining the set of angles of $\fVN(\cells_d)$ is easy. Furthermore, we know the angle $\theta_1$ of the ideal nucleus of the cell of the origin. 
By \cref{lem:evaldist}, we have that for $z$ on the boundary of two neighboring cells with ideal nuclei $(\theta, r)$ and $(\theta', r')$,
\[
\frac{r}{r'}
=
\frac{K(z, \theta)}{K(z, \theta')}.
\]
It follows that $\cells_d$ determines all such quotients and thus the set $\{r_i/r_1 \ST i \ge 1\}$ (we may index $r_i$ in increasing order).
We determine $r_1$ by $1/r_1 = \lim_{i \to\infty} \frac{r_i/r_1}i$ a.s. This defines $\fVN$ and shows that the composition of all three maps is the identity a.s. 
\Qed

Because every factor of an ergodic process is itself ergodic, it follows that the IPVT is ergodic.

\subsection{Delaunay tessellations}\label{s.delaunay}

Intimately tied with Voronoi tessellations are their dual, Delaunay tessellations. Given a locally finite set $X$ of nuclei in $\HH_d$ whose convex hull is all of $\HH_d$, let $V$ be the corresponding Voronoi tessellation. For each vertex $v$ of a cell in $V$, the dual Delaunay cell is defined as the convex hull of those nuclei whose Voronoi cell contains $v$. Note that the nuclei just mentioned lie on the boundary of a ball whose interior is disjoint from $X$. The collection of all Delaunay cells defines the \dfn{Delaunay tessellation}. This is a face-to-face tessellation: a slight modification of the words in the proof of \cite[Theorem 10.2.6]{SchneiderWeil} for the Euclidean case works in our case. 
When $V$ is a normal tessellation, all Delaunay cells are simplices of dimension $d$ and so all their faces are simplices as well.

If $V$ is an ideal Voronoi tessellation, we similarly define dual ideal Delaunay cells. These are ideal simplices with $d+1$ ideal angles for vertices when $V$ is normal. We now show that when $V = \cells_d$, the dual ideal Delaunay cells a.s.\ form a face-to-face simplicial tessellation, which we call the \dfn{ideal Poisson--Delaunay tessellation} (IPDT) and denote by $\dcells_d$. We also show that $\dcells_d$ is the limit of the Delaunay tessellations $\dcells_d^{(\lambda)}$ formed by PPPs of intensity $\lambda > 0$ in $\HH_d$: see \Cref{f.successiveDel}.

\begin{figure}[!h]
 \begin{center}
 \includegraphics[height=5cm]{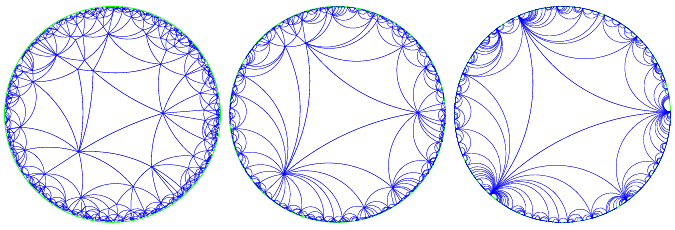}
 \caption{Left to right: Poisson--Delaunay tessellations of the hyperbolic plane (in the unit disk model) with decreasing intensity coupled via dilations using 1000 nuclei. Their limit (on the right) is $ \mathcal{D}_2$, the  \emph{ideal Poisson--Delaunay tessellation} of the hyperbolic plane.}\label{f.successiveDel}
 \end{center}
 \end{figure}

\begin{prop}[\textsc{Dilations for Delaunay}]\label{p.dilateDel}
Adopt the notation of \cref{remark_dilation}. Let $\dcells_d^{(\lambda)}$ be the Delaunay tessellation corresponding to $\tilde\XX^{(\lambda)}$ and $\dcells_d$ be the Delaunay cells associated to $\cells_d$. Then as $\lambda \to 0$, a.s.\ $\dcells_d^{(\lambda)}$ converges to $\dcells_d$ in the sense that each cell of $\dcells_d$ is a limit in the Fell topology of cells of $\dcells_d^{(\lambda)}$.
\end{prop}

The same holds for the dilations in the upper half-space model, \hbox{\cref{l.UHSdilate}}.

\rproof
By \cref{remark_dilation}, a.s.\ each ideal Voronoi vertex, $v$, is a limit of Voronoi vertices $v^{(\lambda)}$ of $\cells_d^{(\lambda)}$. Now $v$ a.s.\ has exactly $d+1$ ideal nuclei at smallest separation, $s$, by \cref{p.locfinite}, and the other ideal nuclei have separations from $v$ that do not cluster at $s$. Therefore, when $\lambda$ is sufficiently small,  $v^{(\lambda)}$ also have closest nuclei that approach these $d+1$ ideal nuclei, which means that the Delaunay simplices corresponding to $v^{(\lambda)}$ also converge to the ideal Delaunay simplex corresponding to $v$. Hence, $\dcells_d^{(\lambda)}$ a.s.\ converges to $\dcells_d$.
\Qed

We will use the results and observations of \cref{s.zerocell} for this next result, but the reasoning will not be circular.

\begin{prop}[\textsc{Basic properties of the IPDT}]
The ideal Poisson--Delaunay tessellation, $\dcells_d$, is a.s.\ a face-to-face tessellation. 
\end{prop}

\rproof
Since each ideal Delaunay cell is an ideal simplex, it has nonempty interior.
In \cref{s.zerocell}, we consider the zero cell in the upper half-space model with its ideal nucleus' angle at $\infty$. In \cref{s.asymp0cell}, we explain how the zero cell is closely related to the Laguerre diagrams studied in \cite{gusakova2022delaunay} for certain choices of parameters there. In particular, a.s.\ each set of $d$ ideal nuclei in $\R^{d-1}$ that form an ideal Delaunay simplex when joined with the ideal nucleus at $\infty$ gives a Euclidean simplex in $\R^{d-1}$ that is one of the so-called $\beta'$-Delaunay simplices, and conversely. These $\beta'$-Delaunay simplices tessellate $\R^{d-1}$ \cite[p.~1263]{gusakova2022delaunay}.  It follows that the set of ideal Delaunay simplices containing $\infty$ is locally finite a.s.\ and  each ideal Delaunay simplex containing $\infty$ has the property that each of its $(d-1)$-faces containing $\infty$ is a face of another such Delaunay simplex. The same then holds a.s.\ for every ideal nucleus, whence $\dcells_d$ is locally finite and covers $\HH_d$ a.s.
\Qed

There is an equivariant, measurable bijection between the $k$-faces of $\cells_d$ and the $(d-k)$-faces of $\dcells_d$ (and similarly for the positive-intensity tessellations), known as \dfn{duality}: the $k$-face of $\cells_d$ given by the intersection of the cells corresponding to ideal nuclei $Y_1, \ldots, Y_{d+1-k}$ is mapped to the $(d-k)$-face of $\dcells_d$ given by the convex hull of $Y_1, \ldots, Y_{d+1-k}$.

\begin{remark}
It follows from \cref{t.factor} that the ideal Delaunay tessellation is a $\mob_d$-factor of the point process of ideal nuclei.
\end{remark}

\Cref{f.opera} gives a side view of some of the 2-faces belonging to the ideal Delaunay tetrahedra that have one of their vertices at infinity in the upper halfspace model of $\HH_3$. The 2-faces containing $\infty$ are not shown.

\begin{figure}[htp] 
\centering
\includegraphics[width=\textwidth]{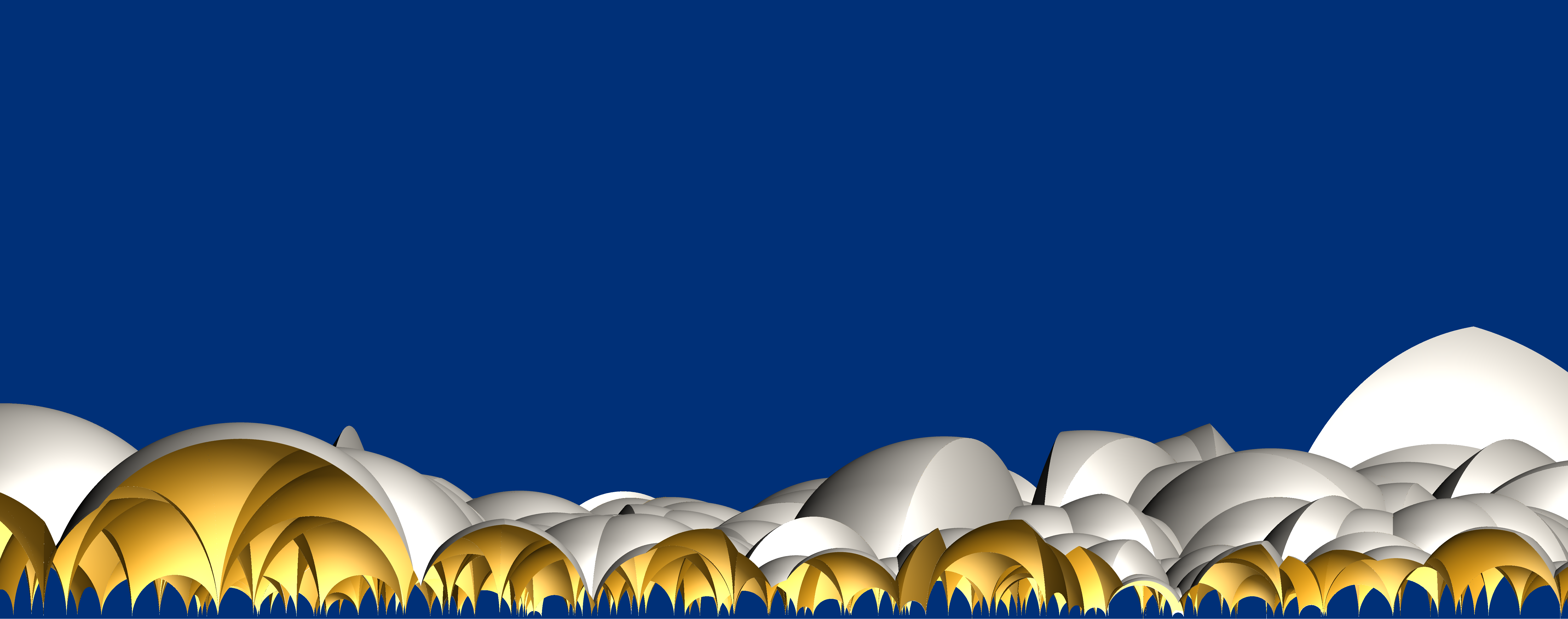}
\caption{A side view of some of the 2-faces belonging to the ideal Delaunay tetrahedra that have one of their vertices at infinity in the upper halfspace model of $\HH_3$. Shown is the portion over a $6 \times 6$ square in the ideal boundary.}
		\label{f.opera}
\end{figure}

\section{Face intensities} \label{s:intensities}
Since $ \mathcal{V}_{d}$ is invariant under isometries of $ \mathbb{H}_{d}$, we can define the intensities of various geometric quantities such as $k$-faces. Using existing results in stochastic geometry, many of those quantities can be explicitly computed for the ideal Poisson--Voronoi tessellations. Let us first recall the classical construction of the Palm distribution for general stationary point processes.

\subsection{Palm measures and typical points}\label{s.palm}
\subsubsection{Euclidean space}
Recall that a point process on $\mathbb{R}^d$ can be seen as an (atomic) measure on the Borel sets $\borel(\mathbb{R}^d)$ of $\R^d$, and that the set $\mathrm{Mes}(\mathbb{R}^d)$ of measures on $\mathbb{R}^d$ can be endowed with the smallest $\sigma$-algebra $\Sigma$ that renders evaluation functionals on Borel sets measurable. Given a random point process $\mathbf{X}$ (hence a probability measure on $\mathrm{Mes}(\mathbb{R}^d$)), one can define a (possibly infinite) measure $\mu$ on the space  $\mathbb{R}^d \times \mathrm{Mes}(\mathbb{R}^d)$ endowed with the product sigma-field $\borel(\R^d) \times \Sigma$ by the formula $\mu(B) \coloneqq \EE\int_{\R^d} \II{B}(x, \XX - x) \,\ud\XX(x)$ for $B \in \borel(\R^d) \times \Sigma$. The intuitive meaning of $\mu$ is that it gives the joint ``distribution'' of a point of the process together with how the other points look relative to that point. When the law of $\mathbf{X}$ is invariant under translation, $A \mapsto \EE\bigl[\XX(A)\bigr]$ is a translation-invariant map on Borel subsets of $\R^d$, whence is a multiple of Lebesgue measure; that constant multiplier is called the \dfn{intensity} $\mathrm I_{\XX}$ of $\XX$. Similarly in this case, $\mu$ is invariant under translations in the first coordinate, that is, $\mu(B_1 \times B_2) = \mu\bigl((B_1 - x) \times B_2\bigr)$ for Borel $B_1 \subseteq \R^d$, measurable $B_2 \in \Sigma$, and $x \in \R^d$. Hence, if $\Intn_{\XX}$ is positive and finite, then $\mu = \mathrm I_{\XX} \cdot \Leb \otimes \PP^\origin$ for a unique probability measure $\PP^\origin$ on measures on $\R^d$ \cite[Theorem 3.3.1]{SchneiderWeil}. For example, $\PP^\origin(B_2) = \Intn_{\XX}^{-1} \cdot \EE \int_{x \in B_1} \II{B_2}(\XX - x) \,\ud\XX(x)$ for every $B_1$ with $\Leb(B_1) = 1$. The measure $\PP^\origin$ is called the \dfn{Palm distribution} of $\XX$. 
When $\XX$ is a point process, $\PP^\origin$ has one of its points at $\origin$ a.s., $\origin$ then being referred to as a \dfn{typical point} of $\XX$. 
When $\XX$ is a PPP, Slivnyak's theorem \cite[Theorem 3.3.5]{SchneiderWeil} shows that the $\PP^\origin$-law of $\XX$ equals the $\PP$-law of $\XX + \delta_\origin$.
Let $\EE^{\origin}$ denote integration with respect to $\PP^{\origin}$.
The definition of $\mu$ tells us that to integrate a nonnegative, measurable function $f$ with respect to $\mu$, we merely replace the indicator $\II{B}$ by $f$, which is the content of the \dfn{refined Campbell theorem}: $\Intn_{\XX} \cdot \EE^{\origin} \int_{\R^d} f(x, \XX) \,\ud x = \EE\int_{\R^d} f(x, \XX - x) \,\ud\XX(x)$ \cite[Theorem 3.3.3]{SchneiderWeil}.

\subsubsection{Hyperbolic space}\label{s.palm-hyp}
Similar definitions hold for a random measure on a group whose law is invariant under translations. In the case of a random measure $\XX$ on a homogeneous space like $\HH_d$, a slightly more complicated definition of Palm measure is needed~\cite{Last}. Assume that the law of $\XX$ is invariant under $\mob_d$. Then $A \mapsto \EE\bigl[\XX(A)\bigr]$ is an isometry-invariant map on Borel subsets of $\HH_d$, whence is a multiple of $\Vol_{\HH_d}$; that multiplier is called the \dfn{intensity} $\mathrm I_{\XX}$ of $\XX$. Let $\kappa$ be the probability Haar measure on the isotropy (stabilizer) subgroup $\mob_d^{\origin}$ of $\origin$. 
Choose maps $x \mapsto \phi_x$ from $\HH_d \to \mob_d$ such that $\phi_x(\origin) = x$. Write $\mob_d^x \coloneqq \{\phi \in \mob_d \sut \phi(\origin) = x\}$ and $\kappa_x \coloneqq \kappa \circ \phi_x^{-1}$; the measure $\kappa_x$ does not depend on the choice of $\phi_x$. Define $\mu(B) \coloneqq \EE\int_{\HH_d} \int_{\mob_d^{x}} \II{B}(x, \XX \circ \phi) \,\ud\kappa_x(\phi) \,\ud\XX(x)$ for $B$ in the product $\sigma$-field generated by Borel sets in $\HH_d$ and the evaluation functionals on Borel sets. If $\Intn_{\XX}$ is positive and finite, then $\mu = \mathrm I_{\XX} \cdot \Vol_{\HH_d} \otimes \PP^\origin$ for a unique probability measure $\PP^\origin$ on measures on $\HH_d$. 
The measure $\PP^\origin$ is called the \dfn{Palm distribution} of $\XX$. 
{When $\XX$ is a point process, $\PP^\origin$} has one of its points at $\origin$ a.s., $\origin$ then being referred to as a \dfn{typical point} of $\XX$. 
The refined Campbell theorem says that $\Intn_{\XX} \cdot \EE^{\origin} \int_{\HH_d}  f(x, \XX ) \,\ud\Vol_{\HH_d}(x) = \EE\int_{\HH_d} \int_{\mob_d^{x}} f(x, \XX \circ \phi) \,\ud\kappa_x(\phi) \,\ud\XX(x)$ for nonnegative, measurable functions, $f$. 
When $\XX$ is a PPP, we again have a Slivnyak-type theorem: the $\PP^\origin$-law of $\XX$ equals the $\PP$-law of $\XX + \delta_\origin$. This is proved as follows.
Let $\XX$ be a PPP with intensity $\lambda$. By Mecke's theorem \cite[Theorem 3.2.5]{SchneiderWeil}, we have that for all nonnegative, measurable $g$,
\rlabel e.mecke
{\EE \int_{\HH_d} g(x, \XX) \,\ud\XX(x)
=
\lambda \int_{\HH_d} \EE\bigl[g(x, \XX + \delta_x)\bigr] \,\ud\Vol_{\HH_d}(x).
}
Choose $B_1 \subset \HH_d$ with $\Vol_{\HH_d}(B_1) = 1$ and a Borel set $B_2$ of measures on $\HH_d$. Consider
\[
g(x, \XX) \coloneqq \int_{\mob_d^{x}} \II{B_1}(x) \II{B_2}(\XX \circ \phi) \,\ud\kappa_x(\phi).
\]
By the refined Campbell theorem,
\rlabel e.rC
{\EE \int_{\HH_d} g(x, \XX) \,\ud\XX(x)
=
\lambda \, \EE^\origin \int_{\HH_d} \II{B_1}(x) \II{B_2}(\XX) \,\ud\Vol_{\HH_d}(x)
=
\lambda\, \PP^\origin[\XX \in B_2].
}
On the other hand, because $(\XX + \delta_x) \circ \phi = \XX \circ \phi + \delta_\origin$ for all $\phi\in \mob_d^x$,
\begin{align}\label{e.mbinv}
\lambda \int_{\HH_d} \EE\bigl[g(x, \XX + \delta_x)\bigr] \,\ud\Vol_{\HH_d}(x) 
&=
\lambda \int_{\HH_d} \int_{\mob_d^x} \II{B_1}(x)\, \EE\II{B_2}(\XX \circ \phi+\delta_\origin) \,\ud\kappa_x(\phi)\,\ud\Vol_{\HH_d}(x) \notag
\\ &=
\lambda \int_{\HH_d} \int_{\mob_d^x} \II{B_1}(x)\, \EE\II{B_2}(\XX +\delta_\origin) \,\ud\kappa_x(\phi)\,\ud\Vol_{\HH_d}(x), \notag
\\ &=
\lambda\,\PP[\XX + \delta_\origin \in B_2],
\end{align}
where, in the second step, we used the $\mob_d$-invariance of the $\PP$-law of $\XX$. Comparing \eqref{e.mecke}, \eqref{e.rC}, and \eqref{e.mbinv} proves our claim.

\subsubsection{Marked random measures}\label{s.marks}
We will be especially interested in the Palm distribution for marked
random measures, with marks being polytopes. Let $E$ be either $\R^d$ or $\HH_d$.  Although we can use any measurable space for the space of marks \cite[Remark 3.9]{Last}, we will use 
the metrizable space $M$, which is the set of closed subsets of $E$ with the Fell topology. 
A \dfn{marked random measure} $\XX$ on $E$ with marks in $M$ is a random measure on $E \times M$ that projects to a random measure on $E$.
We say that the law of $\XX$ is invariant if it is invariant by the group of translations or isometries in the first coordinate, that is, if $\PP[\XX \in A \times B]$ does not change under the group action on $A$ for each fixed $B$, where $A \subseteq E$ and $B \subseteq M$ are Borel. In this case, the preceding paragraphs with virtually no modification give the Palm distribution of $\XX$, which{, when the projection of $\XX$ to $E$ is a point process, is concentrated on}  points of the form $(\origin, m)$, and also give the refined Campbell theorem. In this point-process case, we also call the (random) mark $m$ of $\origin$ a \dfn{typical mark}.

\subsubsection{Tessellations}\label{s.palm-tess}
Let us apply the preceding paragraphs to our setting. Recall that we denote by $ \mathcal{V}_d^{( \lambda)}$ the Poisson--Voronoi tessellation on $ \mathbb{H}_d$ with positive intensity $\lambda >0$ and by $ \mathcal{V}_d$ our limiting ideal Poisson--Voronoi tessellation. For fixed $0 \leq k \leq d-1$, to each $k$-dimensional face $ \mathfrak{s}$ of $ \mathcal{V}_{d}$ or of $\cells_d^{(\lambda)}$ (we speak
of vertices when $k=0$ and edges when $k=1$), we associate a point $\mathrm{Cen}( \mathfrak{s}) \in \mathfrak{s}$, called the  \dfn{center}, in an isometry-equivariant, measurable way (i.e., $ \mathrm{Cen} \bigl( \phi (\mathfrak{s})\bigr) = \phi\bigl(  \mathrm{Cen}( \mathfrak{s})\bigr)$ for every isometry, $\phi$). It will be convenient to assume that $\Cen(\mathfrak s)$ lies in the relative interior of $\mathfrak s$, which implies that $\Cen$ is injective on the set of $k$-faces. The centroid \cite[p.~222]{ratcliffe} is such a center. We will also have use for the \dfn{measure centroid}, defined like the centroid using the hyperboloid model \cite[p.~59]{ratcliffe}, but integrating with respect to $\nu_{ \mathfrak{s}}$ rather than with respect to $\sum_{x \in \mathfrak{s} \cap \cells_{d, 0}} \delta_x$, where $\nu_{ \mathfrak{s}}$ is the $k$-dimensional hyperbolic measure on $ \mathfrak{s}$.
Thus, the centers of the $k$-faces of a Poisson--Voronoi tessellation allow us to define the invariant marked random measure corresponding to $\cells_{d, k}$ that was described in \Cref{s.marks} as
\begin{equation} \label{markedX}
\XX_{d,k} \colon A \mapsto \sum_{ \begin{subarray}{c} \mathfrak{s} \in k\textrm{-faces} \end{subarray}} \int_{\mob_d^{\Cen(\mathfrak s)}} \II{A}\bigl(\mathrm{Cen}( \mathfrak{s}), \phi^{-1}(\mathfrak s)\bigr) \,\ud\kappa_{\Cen(\mathfrak s)}\end{equation}
 for Borel $A \subseteq E \times M$, where $\kappa_x$ is as above \cite[Section 5.4]{Herold:thesis}. 
A typical mark is now called a \dfn{typical $k$-face}. This applies to $k = d$ as well when $\lambda > 0$, since the cells have finite volume, making it possible to define a center; in this case, we refer to a \dfn{typical cell}. When $k = 0$, we speak of a typical vertex; in this case, the origin may be thought of as conditioned to be a Voronoi vertex.

Similar considerations apply for the faces of the (possibly ideal) Delaunay tessellations $\dcells_d^{(\lambda)}$ or $\dcells_d$ and for the Delaunay simplices themselves. Here, it may be necessary to use \dfn{(generalized) centers} of the $k$-simplices, where we no longer require that the generalized center of a $k$-simplex lie in that simplex or we use extra randomness. In our case, we may use the randomness of the PPP of (ideal) nuclei. In fact, for the generalized center of a Delaunay $k$-simplex, we will use the center of the corresponding $(d-k)$-face of the Voronoi tessellation. Here, $1 \le k \le d$, although in the case of positive-intensity nuclei, rather than ideal nuclei, we may allow $k = 0$. The only centers we will use for the ideal Delaunay tessellation are the generalized centers of the ideal Delaunay simplices, which are, then, the corresponding Voronoi vertices.

\subsection{Intensities and ergodic properties}  Let $\cells_{d, k}$ denote the collection of $k$-faces of cells of $\cells_d$; in particular, $\cells_{d, d} = \cells_d$.  Consider the random measures
\begin{equation} \label{def_random_measures}
\mu_{d,k} \coloneqq \sum_{ \begin{subarray}{c} \mathfrak{s} \in {\cells_{d, k}} \end{subarray}} \delta_{ \mathrm{Cen}( \mathfrak{s})} \quad \mbox{and} \quad \tilde{\mu}_{d,k}\coloneqq \sum_{ \begin{subarray}{c} \mathfrak{s} \in {\cells_{d, k}} \end{subarray}}\nu_{ \mathfrak{s}}.
\end{equation}
Of course, $\mu_{d, 0} = \tilde \mu_{d, 0}$.
Since $ \mathcal{V}_{d}$ is isometry invariant in law (Lemma \ref{lem:mobius}), so are $\mu_{d,k}$ and $ \tilde{\mu}_{d,k}$, whence the mean measures (their expectations) are  multiples of the hyperbolic measure on $ \mathbb{H}_{d}$. We thus set 
  \begin{eqnarray} \label{eq:defintensities}  \mathbb{E}[\mu_{d,k}] =  \mathrm{I}_{d,k} \cdot \mathrm{Vol}_{ \mathbb{H}_{d}} \quad \mbox{and} \quad \mathbb{E}[\tilde{\mu}_{d,k}] =  \tilde{\mathrm{I}}_{d,k} \cdot \mathrm{Vol}_{ \mathbb{H}_{d}} \; . \end{eqnarray}
We call the constant $ \mathrm{I}_{d,k} \in [0, \infty]$ (resp., $ \tilde{ \mathrm{I}}_{d,k}$), the \dfn{counting  $k$-face intensity}
(resp.,~\dfn{measure  $k$-face intensity}) of $ \mathcal{V}_{d}$. The latter intensity is called by \cite[p.~141]{SchneiderWeil} the \dfn{$k$-volume
density} or \dfn{specific $k$-volume}. 

{
The counting intensities do not depend on the choice of center function. To see this, let $\Cen$ and $\Cen'$ be two center functions. Denote by $\mu_{d, k}$ and $\mu'_{d, k}$ the corresponding random measures as in \eqref{def_random_measures}.
Define the measure
\[
\widehat\mu_{d,k} \coloneqq \sum_{ \mathfrak{s} \in \cells_{d, k} } \delta_{ \mathrm{Cen}( \mathfrak{s})} \otimes \delta_{\Cen'( \mathfrak{s})}
\]
on $\HH_d \times \HH_d$. For Borel $A \subseteq \HH_d$, we have 
\[
\widehat\mu_{d,k}(A \times \HH_d) = \mu_{d,k}(A) \quad \text{and}  \quad
\widehat\mu_{d,k}(\HH_d \times A) = \mu'_{d,k}(A).
\]
The measure $\widehat\mu_{d,k}$ is equivariant under the diagonal action of $\mob_d$ and the law of $\cells_{d, k}$ is invariant under the action of $\mob_d$, whence the
mean measure $\EE[\widehat\mu_{d,k}]$ is invariant under the diagonal action of $\mob_d$.
By unimodularity of $\mob_d$, the mass-transport principle \cite[Theorem 8.47]{LP:book} yields
\[
\EE\bigl[\widehat\mu_{d,k}(A \times \HH_d)\bigr] =
\EE\bigl[\widehat\mu_{d,k}(\HH_d \times A)\bigr].
\]
This gives our claim.
}

These intensities can naturally be defined for the Poisson--Voronoi tessellation $\cells_d^{(\lambda)}$ with positive intensity $\lambda >0$, which we denote by
$\mathrm{I}_{d,k}^{( \lambda)}$ and $ \tilde{\mathrm{I}}^{( \lambda)}_{d,k}$ for $0 \leq k \leq d$. 
{By choosing the center of a cell to be its nucleus, we obtain $\Intn_{d,d}^{(\lambda)} = \lambda$. The definition gives $\tilde\Intn_{d,d}^{(\lambda)} = 1$. Similar intensities can be defined for the (ideal) Poisson--Delaunay tessellation.}

As one would expect, the intensities for $\lambda > 0$ converge to the corresponding intensities of the ideal tessellation as $\lambda \to 0$:
\begin{lem}[{\textsc{Convergence of $k$-face intensities}}] \label{lem:convintensities}For $0 \leq k \leq d-1$, we have 
$$ \mathrm{I}_{d,k}^{(\lambda)} \xrightarrow[\lambda \to 0]{} \mathrm{I}_{d,k} \in (0, \infty) \quad \mbox{ and } \quad \tilde{\mathrm{I}}_{d,k}^{(\lambda)} \xrightarrow[\lambda \to 0]{} \tilde{\mathrm{I}}_{d,k} \in (0, \infty).$$
\end{lem}
\proof Let $B_u = B_u(\origin)$ be the ball of hyperbolic radius $u \geq 0$ around the origin in $ \mathbb{H}_{d}$, so that 
  \begin{equation}\label{eq:espintensity} \mathrm{I}_{d,k} =  \frac{\mathbb{E}[ \mu_{d,k}(B_{u})]}{ \mathrm{Vol}_{ \mathbb{H}_{d}}(B_{u})} \end{equation} for $u > 0$, with a similar formula for positive intensities $\lambda >0$. As the center of a face, choose its measure centroid, which is continuous with respect to the Fell topology when restricted to compact, nonempty polytopes. Since $ \mathcal{V}_{d}^{(\lambda)} \to \mathcal{V}_{d}$ by~\cref{thm.convhyp}, we have, with obvious notation, that $\mu_{d,k}^{(\lambda)}(B_{u}) \to \mu_{d,k}(B_u)$ in distribution as $\lambda \to 0$. By \eqref{eq:espintensity}, it thus suffices to prove that $( \mu_{d,k}^{(\lambda)}(B_{u}) \sut 0 < \lambda < 1)$ is uniformly integrable to deduce convergence of the expectations from the convergence in law. Define $N^{(\lambda)}_{d,k}(B_u)$ to be the number of $k$-faces of $\cells_d ^{(\lambda)}$ that intersect $B_u$.  Adapting \cref{p.tailbound} for positive intensity $\lambda>0$, we obtain again a (uniform in $\lambda \in (0, 1)$) stretched exponential bound on the tail probabilities of $N^{(\lambda)}_{d,k}(B_u)$, which yields the desired uniform integrability because $\mu^{(\lambda)}_{d,k}(B_u) \le N^{(\lambda)}_{d,k}(B_u)$.  

Since $\tilde \mu^{(\lambda)}_{d,k}(B_u) \le \Vol_{\HH_k}(B^{(k)}_u)\cdot  N^{(\lambda)}_{d,k}(B_u)$, where $B^{(k)}_u$ is a ball of radius $u$ in $\HH_k$, a similar proof holds for the convergence of $\tilde{\mathrm{I}}_{d,k}^{(\lambda)}$.  
\endproof

These face intensities are directly related to aspects of typical faces and cells as shown in the following two propositions.

\begin{prop}[\textsc{Intensities and typical volumes}]\label{p.typ-vol}
For $0 \le k \le d-1$, the quotient $\tilde{\Intn}_{d, k}/\Intn_{d, k}$ equals the expected $k$-volume of the typical $k$-face of\/ $\cells_d$. 
{For $\lambda > 0$ and $0 \le k \le d$, the quotient $\tilde{\Intn}_{d, k}^{(\lambda)}/\Intn_{d, k}^{(\lambda)}$ equals the expected $k$-volume of the typical $k$-face of\/ $\cells_d^{(\lambda)}$.}
\end{prop}

\rproof
Define the measure
\[
\widehat\mu_{d,k} \coloneqq \sum_{ \begin{subarray}{c} \mathfrak{s} \in {\cells_{d, k}}\end{subarray}} \delta_{ \mathrm{Cen}( \mathfrak{s})} \otimes \nu_{ \mathfrak{s}}
\]
on $\HH_d \times \HH_d$. Write $\| \nu_{ \mathfrak{s}}\|$ for the total $k$-volume of $\mathfrak s$. For Borel $A \subseteq \HH_d$, we have 
\[
\widehat\mu_{d,k}(\HH_d \times A) = \tilde\mu_{d,k}(A) \quad \text{and}  \quad
\widehat\mu_{d,k}(A \times \HH_d) = \sum_{ \begin{subarray}{c} \mathfrak{s} \in {\cells_{d, k}},\\ {\Cen(\mathfrak s) \in A} \end{subarray}}\|\nu_{ \mathfrak{s}}\|.
\]
The mean measure $\EE[\widehat\mu_{d,k}]$ is invariant under the diagonal action of $\mob_d${, whence}
the mass-transport principle yields
\rlabel e.mtp1
{\EE\bigl[\widehat\mu_{d,k}(\HH_d \times A)\bigr] = 
\EE\bigl[\widehat\mu_{d,k}(A \times \HH_d)\bigr].
}
The left-hand side equals $\tilde{\Intn}_{d, k} \cdot \Vol_{\HH_d}(A)$, while the right-hand side equals $\EE\Bigl[\sum_{ \begin{subarray}{c} \mathfrak{s} \in {\cells_{d, k}},\\ \Cen(\mathfrak s) \in A \end{subarray}}\|\nu_{ \mathfrak{s}}\|\Bigr]$. Let $\XX_{d, k}$ denote the marked random measure corresponding to $\cells_{d, k}$ defined in \cref{markedX}. Using the function 
\[
f\bigl(\Cen(\mathfrak s), \XX_{d, k}\bigr)
\coloneqq
\rcases{
\II{A}\bigl(\Cen(\mathfrak s)\bigr) \|\nu_{\mathfrak t}\| &\text{if } \mathfrak t \text{ is the unique $k$-face with } \Cen(\mathfrak t) = \origin\\
0 &\text{otherwise}
}
\]
in the refined Campbell theorem, we obtain that $\EE\Bigl[\sum_{ \begin{subarray}{c} \mathfrak{s} \in {\cells_{d, k}},\\ \Cen(\mathfrak s) \in A \end{subarray}}\|\nu_{ \mathfrak{s}}\|\Bigr]$ equals the expected $k$-volume of the typical $k$-face times $\Intn_{d,k} \cdot \Vol_{\HH_d}(A)$. Hence \eqref{e.mtp1} gives our claim.

{A similar proof gives the result for positive $\lambda$.}
\Qed

\begin{prop}[\textsc{Typical face vectors and volumes}]\label{p.typ-face}
For $\lambda > 0$ and $0 \le k \le d-1$, the mean number (resp., $k$-volume) of $k$-faces of\/ the typical cell of $\cells_d^{(\lambda)}$ is $(d+1-k) \Intn_{d, k}^{(\lambda)}/\lambda$ (resp., $(d+1-k) \tilde\Intn_{d, k}^{(\lambda)}/\lambda$).
\end{prop}

\rproof
Let $f_k(\mathfrak{t})$ be the number of $k$-faces of a cell, $\mathfrak{t}$. Define the measure
\[
\widehat\mu_{d,k}^{(\lambda)} \coloneqq \sum_{ \mathfrak{s} \in \cells_{d, k}^{(\lambda)},\ \mathfrak{t} \in \cells_d^{(\lambda)} } \II{\mathfrak{s} \subset \mathfrak{t}} \cdot \delta_{ \mathrm{Cen}( \mathfrak{s})} \otimes \delta_{ \mathrm{Cen}( \mathfrak{t})}
\]
on $\HH_d \times \HH_d$. For Borel $A \subseteq \HH_d$, we have 
\[
\widehat\mu_{d,k}^{(\lambda)}(A \times \HH_d) = (d+1-k) \mu_{d, k}^{(\lambda)}{(A)} \text{ a.s.}
\quad \text{and}  \quad
\widehat\mu_{d,k}^{(\lambda)}(\HH_d \times A) = \sum_{ \substack{ \mathfrak{t} \in  \cells_d^{(\lambda)}{,} \\ {\Cen(\mathfrak{t}) \in A}} } f_k(\mathfrak{t}).
\vadjust{\kern -10pt}%
\]
The mean measure $\EE[\widehat\mu_{d,k}^{(\lambda)}]$ is invariant under the diagonal action of $\mob_d$, whence the mass-transport principle yields
\begin{equation} \label{mass_transport}
\EE\bigl[\widehat\mu_{d,k}^{(\lambda)}(A \times \HH_d)\bigr]
= 
\EE\bigl[\widehat\mu_{d,k}^{(\lambda)}(\HH_d \times A)\bigr]
.
\end{equation}
The left-hand side equals $(d+1-k)\Intn_{d, k}^{(\lambda)} \cdot \Vol_{\HH_d}(A)$, while the right-hand side equals, by the refined Campbell theorem, 
the expected number of $k$-faces of the typical cell times $\Intn_{d,d}^{(\lambda)} \cdot \Vol_{\HH_d}(A)$. Hence \eqref{mass_transport} establishes our claim.
The proof is similar for the mean $k$-volume.
\Qed

We now give the distribution of the separation of the typical vertex from its closest ideal nucleus. Surprisingly, it equals the distribution of $R_{d+1}$.

\begin{prop}[\textsc{Typical separation}]\label{p.typ-sep}
The separation of a typical vertex of\/ $\cells_d$ to its corresponding $d+1$ ideal nuclei follows a $\text{\rm Gamma}(d+1)$ distribution.
\end{prop}
\rproof 
For a Voronoi vertex $v$ of the ideal tessellation $\mathcal{V}_d$, let $s(v)$ denote the least separation from any ideal nucleus, as defined in \eqref{eq.defsep}.
Let $\EEtyp$ denote expectation with respect to the Palm distribution of $\cells_d$ for vertices. 
We have that for every pair of compactly supported, continuous functions $f_1, f_2 \colon \R_+ \to \R$ with $\int f_1\bigl(\ud_{\HH_d}(v, \origin)\bigr) \, \ud\Vol_{\HH_d}(v) = 1$,
\rlabel e.typ-vertex
{\EEtyp\bigl[f_2\bigl(s(\origin)\bigr)\bigr]
= 
\Intn_{d,0}^{-1}\cdot \mathbb{E} \biggl[ \sum_{v \in \cells_{d, 0}} f_1\bigl(\ud_{\HH_d}(v, \origin)\bigr) f_2\bigl(s(v)\bigr) \biggr].
}
Consider the Poisson--Voronoi tessellation associated to a Poisson point process $\mathbf{X}^{(\lambda)}$ with intensity $\lambda^{d-1}$. For a $(d+1)$-tuple $\vec{\eta}=(\eta_1 , \dots , \eta_{d+1}) \in \mathbb{H}_d ^{d+1}$, we denote by $B(\vec{\eta})$ an open ball whose boundary contains the points $\eta_1 , \dots , \eta_{d+1}$, by $O(\vec{\eta})$ the ball's center, and by $R({\vec{\eta}})$ its radius{---unless no such ball exists. Let $\mathcal B \subset \HH_d^{d+1}$ denote the set where such a circumscribed ball does exist and is unique. 
Consider the coupling of $\mathbf{X}^{(\lambda)}$ as in \cref{remark_dilation}. As remarked there, we have convergence of Voronoi vertices towards the ideal Voronoi vertices. Furthermore, as shown in the proof of \cref{p.dilateDel}, the closest nuclei to the Voronoi vertices converge to the corresponding ideal nuclei, and so the distances of Voronoi vertices to their closest nuclei converge, when properly normalized as proto-delays, to the linear separations of the ideal Voronoi vertices to their corresponding ideal nuclei. It remains to calculate those distances and take the limit of their normalized distribution. We will use the normalization corresponding to the exponential separations of \eqref{eq.defsep} rather than the proto-delays}.

{We now continue with evaluating the right-hand side of \eqref{e.typ-vertex}.} For any compactly supported, continuous function $f{\colon \HH_d \times \R_+ \to \R}$,
\[
\mathbb{E}  \biggl[ \sum_{v \in \cells_{d, 0}} f\bigl(v,s(v)\bigr) \biggr] = \lim_{\lambda \downarrow 0} \mathbb{E}   \biggl[ \sum_{\vec{\eta} \in (\mathbf{X}^{(\lambda)})^{d+1} {\cap \mathcal{B}} }  f \Bigl(O(\vec{\eta}), \frac{c_{d}}{d-1}\lambda^{d-1}\ue^{(d-1)R(\vec{\eta})} \Bigr)  \mathbf{1}_{B(\vec{\eta}) \cap \mathbf{X}^{(\lambda)} = \emptyset}  \biggr].
\]
By the Slivnyak--Mecke formula, the expectation on the right-hand side is given by
\[
 \frac{1}{(d+1)!} \int \limits_{{\mathcal{B}}} f\Bigl( O(\vec{x}),  \frac{c_{d}}{d-1}\lambda^{d-1}\ue^{(d-1)R(\vec{x})} \Bigr) \exp\Bigl\lbrace -\lambda^{d-1}  \mathrm{Vol}_{ \mathbb{H}_{d}}\bigl(B(\vec{x})\bigr)  \Bigr \rbrace \lambda^{d^2-1} \,\ud \vec{x} \; .
\]

Using the spherical Blaschke--Petkantschin change of variable $x_i \coloneqq z + {u} \cdot {\theta}_i$ for $(z,{u}, {\theta}_1, \dots, {\theta}_{d+1}) \in \mathbb{H}_d \times \mathbb{R}_{+} \times (\mathbb{S}_{d-1})^{d+1} $ stated in~\cite[Proposition 1.1]{chapron2018}\footnote{The formula there is missing a factor of $\II{\mathcal{B}}(x_1, \dots, x_{d+1})$ on the left-hand side, where $\mathcal B$ is as defined here. Necessarily, $\{x_1, \ldots, x_{d+1}\}$ lie on a common sphere, namely, the sphere centered at $z$ with radius ${u}$.}, the previous integral becomes
\begin{equation}\label{eq.convexp} 
C(d) \int \limits_{\mathbb{H}_d \times \mathbb{R}_{+} } f\bigl(z, \frac{c_{d}}{d-1}\lambda^{d-1}\ue^{(d-1){u}}\bigr) \exp \Bigl\{ -\lambda^{d-1} \Omega_d \int _0 ^{{u}} (\sinh t)^{d-1} dt \Bigr\} (\sinh u)^{d^2-1} \,\ud u  \,\ud z,
\end{equation}
where $C(d)$ is a finite constant. By the change of variable $s\coloneqq\frac{c_{d}}{d-1}\lambda^{d-1}\ue^{(d-1){u}}$ and passing to the limit $\lambda \downarrow 0$ {with the aid of the Lebesgue dominated convergence theorem, we see that} the quantity \eqref{eq.convexp} converges to $C'(d) \int \limits_{\mathbb{H}_d \times \mathbb{R}_+} f(z,s) s^d \ue^{-s} \,\ud s \,\ud z$ {for another constant $C'(d)$}.  The proof is concluded by choosing $f(z,s)=f_{1}\bigl(\ud_{\HH_d}(v, \origin)\bigr) f_{2}(s)$ for $f_{1}$ and $f_{2}$ as above.
\Qed

  We close this subsection by showing that the face intensities are not only expectations, but also a.s.\ limits on large balls, as well as showing two ergodic limits concerning the separation. For this, we use the following ergodic theorem of \cite{NevoI, NevoStein} specialized from the context of a general probability-measure-preserving action to our setting: 

\begin{thm}[Specialization of \cite{NevoI, NevoStein}]\label{t.nevo}
Let $f$ be a function on the space of ideal Voronoi tessellations such that $f \in L^p$ with respect to the natural probability measure $\PP$ on IPVTs for some $p > 1$. Let $F_u$ be the set of isometries of\/ $\HH_d$ that map the origin into the ball $B_u = B_u(\origin)$ of hyperbolic radius $u$ and volume $b_u$, and $\mu$ be a Haar measure on $\mob_d$ normalized so that $\mu(F_1) = b_1$ (and hence $\mu(F_u) = b_u$ for all radii, $u$). Then $\lim_{u \to\infty} b_u^{-1} \int_{F_u} f(g \cells_d) \,\ud\mu(g) = \EE[f]$ a.s. \qed
\end{thm}

We give three examples.

\begin{prop}[\textsc{Ergodic limits}] We adopt the notation of \cref{t.nevo}.
\begin{enumerate}[label=\textup{(\roman*)}]
\item For $0 \le k \le d-1$, the counting face intensities satisfy
\[
\mathrm{I}_{d,k} = \lim_{u \to\infty} \frac{\mu_{d,k}(B_u)}{\Vol_{\HH_d}(B_u)} \mbox{ a.s.}
\]
\item Let $s(z)$ be the separation of $z$ from the ideal nucleus of its Voronoi cell. Then a.s., for all $t \in (0, 1)$, we have
\[
\lim_{u \to\infty} \frac{1}{\Vol_{\HH_d}(B_u)} \int_{B_u} \ue^{t s(z)} \,\ud \Vol_{\HH_d}(z) = \frac{1}{1 - t}.
\]
In other words, the limit empirical moment generating function of the separation is the moment generating function of an $\Exp(1)$ random variable.
\item Let $D_u$ be the empirical distribution of the separations of the Voronoi vertices in $B_u$ from their associated ideal nuclei. Then $D_u$ tends weakly as $u \to\infty$ to a ${\rm Gamma}(d+1)$ distribution a.s.
\end{enumerate}
\end{prop}

\rproof 
\textbf{(i)}
Recall \eqref{def_random_measures} and let $f_\epsilon({\cells_d}) \coloneqq \mu_{d,k}\bigl({B_\epsilon}\bigr)/b_\epsilon$ be the number of $k$-face centers within distance $\epsilon$ of $\origin$ divided by
$b_\epsilon$.  Note that $\EE[f_\epsilon({\cells_d})] = \mathrm{I}_{d,k}$. Write
\[
A(u, \epsilon) \coloneqq b_u^{-1} \int_{F_u} f_\epsilon(g \cells_d) \,\ud\mu(g).
\]
According to \cref{t.nevo}, $\lim_{u \to\infty} A(u, \epsilon) = \mathrm I_{d, k}$ a.s.\ for each $\epsilon > 0$.
We have the inequalities
\[
   b_{u-\epsilon} A(u-\epsilon, \epsilon) \le {\mu_{d,k}(B_u)} \le b_{u+\epsilon} A(u + \epsilon, \epsilon).
\]
Divide by $b_u$, take $u \to\infty$, and then take $\epsilon$ to 0 along a sequence. 

\textbf{(ii)}
All $s(z)$ have the same distribution, namely, that of $R_1$, which is $\Exp(1)$. Since 
\[
\int_{B_u} \ue^{t s(z)} \,\ud \Vol_{\HH_d}(z) = \int_{F_u} \ue^{t s(g \origin)} \,\ud \mu(g),
\]
the result follows for each $t$ separately from \cref{t.nevo}. Applying this result to rational $t$ and noting that $\ue^{t s(z)}$ is monotone in $t$ allows a comparison for irrational $t$ that gives the result.

\textbf{(iii)}
Let $f_{\epsilon, t}({\cells_d})$ be the number of vertices within distance $\epsilon$ of $\origin$ whose separations from their nuclei are at most $t$, divided by ${\mathrm{I}_{d,0}} \cdot b_\epsilon$.  Then $\EE[f_{\epsilon, t}(\cells_d)] = \frac{1}{d!}\int_{0}^{t} x^{d} \ue^{-x} \,\ud x $ for all $\epsilon, t > 0$ by \cref{p.typ-sep} and the refined Campbell theorem.
{Write
\[
A(u, \epsilon, t) \coloneqq b_u^{-1} \int_{F_u} f_{\epsilon, t}(g {\cells_d}) \,\ud\mu(g).
\]
According to \cref{t.nevo}, $\lim_{u \to\infty} A(u, \epsilon, t) = \frac{1}{d!}\int_{0}^{t} x^{d} \ue^{-x} \,\ud x $ a.s.\ for each $\epsilon, t > 0$.
We have the inequalities
\[
   \Intn_{d, 0} b_{u-\epsilon} A(u-\epsilon, \epsilon, t) \le \mu_{d,0}(B_u) D_u(t) \le \Intn_{d, 0} b_{u+\epsilon} A(u + \epsilon, \epsilon, t).
\]
Divide by $\Intn_{d,0} b_u$, take $u \to\infty$, and then take $\epsilon$ to 0 along a sequence. The result of part (i) for $k = 0$ then gives $\lim_{u \to\infty} D_u(t) = \frac{1}{d!}\int_{0}^{t} x^{d} \ue^{-x} \,\ud x $ a.s. Use this for all rational $t$ together with monotonicity in $t$ to get the result for all $t$.}
\Qed

One can prove similar ergodic theorems for various statistics within horospherical shells. This follows from the Euclidean ergodic theorem (e.g., \cite[Theorem 9.3.1]{SchneiderWeil}) by using the upper half-space model, $\uhs_d$. We also have that the action by translation along a geodesic or along a horocycle is ergodic. 

\subsection{Computations via limits as $\lambda \to 0$}
With Lemma \ref{lem:convintensities} at hand, we can now use the explicit results \cite{isokawa2000H2,isokawa2000H3,GodKabTha,CCEpreprint} to compute the $k$-face intensities. To give the exact expressions, define
\[
\cc(a)
\coloneqq
\frac{\Gamma(a)}{\sqrt\pi\,\Gamma(a - 1/2)}
\]
and
\rlabel e.defJJ
{\JJ_{d, k}
\coloneqq
2\binom{d}{k} \cc\Bigl(\frac{d^2}2\Bigr) \int_0^\infty \frac1{\cosh^{d^2-1}(u)} \Re \biggl(\half+\ui\cdot \cc\Bigl(\frac{d+1}2\Bigr) \int_0^u \cosh^{d-1}(v) \,\ud v\biggr)^{\! k} \,\ud u.
}
In particular, $\JJ_{d, 0} = 1$.
(It actually turns out that $\JJ_{d, k}$ is the expected sum of the internal angles at all $(d-k-1)$-faces of a simplex in $\R^{d-1}$ generated by $d$ i.i.d.\ points with density proportional to $x \mapsto \bigl(1+|x|^2\bigr)^{1/2-d}$; see \cite[Section 1.4]{KabAngles} for more on the values and computation of $\JJ_{d, k}$.)
For $d = 2, 3, 4, 5$, we have the following values of $(\JJ_{d, k} \ST 1 \le k \le d-1)$: $(1)$, $(3/2, 1/2)$, $(2,170/143,27/143)$, and 
\[
\Bigl(\frac{5}{2},\frac{5}{3} + \frac{62173301}{13970880 \pi^2},\frac{62173301}{9313920 \pi^2},-\frac{1}{6} + \frac{62173301}{27941760 \pi^2}\Bigr).
\]
Also, define 
\rlabel e.defIDV
{\IDV_d
\coloneqq
\frac{2 \sqrt{\pi }\, \Gamma \big(\frac{d+1}{2}\big)^{d}
\Gamma \big(\frac{d^2}{2}\big)}{\bigl((d-1) \Gamma (\frac{d}{2})\bigr)^{d+ 1} \Gamma \big(\frac{d^2-1}{2} \big)}.
}

\begin{thm}[\textsc{Counting face intensities}]\label{t.FaceInt} 
$\IDV_d$ is the mean volume of the typical ideal Delaunay simplex.
For $0 \leq k \leq d-1$ and $d \geq 2$, we have 
\[ \mathrm{I}_{d,k} = \frac{d+1}{d+1-k} \JJ_{d, k} \cdot \frac{1}{\IDV_d} \, .\]
When $k = 0$, this limit is asymptotic to $\sqrt2 \ue^{13/12-d/2} d^{d-1/2}$ as $d \to\infty$.
\end{thm}

For $2 \le d \le 5$, the vertex intensity (which is the reciprocal of $\IDV_d$) evaluates to
\[
\frac{1}{\pi },\
\frac{16 \pi }{35},\ \frac{1287}{16 \pi ^2},\ \frac{3981312 \pi ^2}{676039}.
\]
The case of $d = 2$ is easy directly, since every vertex of $ \mathcal{V}_{d}$ is the generalized center of its corresponding Delaunay triangle, which is an ideal triangle, and because every ideal triangle has area $\pi$.

\rproof
\cite[Theorem 3.11]{GodKabTha} gives the mean face vector of the typical cell of $\cells_d^{(\lambda)}$. We then use \cref{p.typ-face} to get $\Intn_{d,k}^{(\lambda)}${ and} take the limit as $\lambda \to 0$. Here is an outline of the calculation of the limit. The formula for the mean face vectors, \cite[(3.18)]{GodKabTha}, is a sum over an index, $s$. The term with $s = 0$ tends to infinity as $\lambda \to 0$ much faster than the others, so that is the only one whose asymptotics concerns us. In the notation of~\cite{GodKabTha}, this term is written $2 \mathbb I_{\alpha, d}^*(d) \tilde {\mathbb J}_{d, d-k}(d - 1/2)$. In our notation (see \cite[(2.8)]{GodKabTha}):
\begin{itemize}
\item $\mathbb I_{\alpha, d}^*(d) = \alpha^d I_{\alpha, d}^*(d)/d!$,
\item $\tilde {\mathbb J}_{d, d-k}(d - 1/2) = \JJ_{d, k}$,
\end{itemize}
with $\alpha \coloneqq 2^d \lambda/\tilde c_{d, d}$, $\tilde c_{a, b} \coloneqq \Gamma(b)/\bigl(\pi^{a/2} \Gamma(b-a/2)\bigr)$, and 
\[
I_{\alpha, d}^*(d)
=
\tilde c_{1, (d^2+1)/2} \int_0^\infty \sinh^{d^2-1} \varphi \cdot \exp \biggl\{ - \alpha \tilde c_{1, (d+1)/2} \int_0^\varphi \sinh^{d-1} \theta \,\ud \theta \biggr\} \ud \varphi.
\]

Taking the limit $\lambda \to 0$, we can approximate in the integrand $\sinh \theta$ by $\ue^\theta/2$ and likewise $\sinh \varphi$ by $\ue^\varphi/2$. Approximate the resulting inner integral by $\ue^{(d-1)\varphi}/\bigl((d-1)2^{d-1}\bigr)$ and change variables to $\psi \coloneqq \ue^{(d-1) \varphi}$. Then observe that the resulting integral is asymptotic to $1/\bigl(2^{d(d-1)}\alpha \tilde c_{1, (d+1)/2}\bigr)$ times the $d$th moment of an exponential random variable with parameter $(d-1)2^{d-1}/(\alpha \tilde c_{1, (d+1)/2})$. 
\Qed

We remark that for $k = 0$, we could also have used the vertex intensity for $\lambda > 0$ given in \cite[Theorem 1.2(ii)]{CCEpreprint} (which does not appear in the published version, \cite{CCE}). Furthermore, instead of using the tail bounds for $k = 0$, we could have used the fact that the volumes of simplices in $\HH_d$ are bounded (see, e.g., \cite[Lemma 5.2]{Bow}), which yields convergence of the mean volumes of typical Delaunay simplices and thereby convergence of vertex intensities.

\begin{cor}[\textsc{Typical face vectors}]\label{c.f-vector}
Let $\tcell_d^{(\lambda)}$ be the typical cell in the Poisson--Voronoi tessellation of\/ $\HH_d$ with intensity $\lambda > 0$.
For $k, \ell \in [0, d-1]$, the limit as $\lambda \to 0$ of the mean number of $k$-faces of $\tcell_d^{(\lambda)}$ divided by the mean number of $\ell$-faces of $\tcell_d^{(\lambda)}$ is $\JJ_{d, k}/\JJ_{d,\ell}$.
\end{cor}

\rproof 
Use \cref{p.typ-face} and \cref{lem:convintensities} to get that the quotient has limit $\frac{(d+1-k)\Intn_{d,k}}{(d+1-\ell)\Intn_{d,\ell}}$. Then substitute the values from \cref{t.FaceInt}.
\Qed

This can also be proved from the formula for the mean $f$-vector of $\tcell_d^{(\lambda)}$ given in \cite[Theorem 3.11]{GodKabTha}, using the fact that the term with $s = 0$ in (3.18) there approaches $\infty$ much faster than the other terms do (see (2.7) there for the definition of the terms).

We can combine \cref{p.typ-face} with results of Isokawa \cite{isokawa2000H2,isokawa2000H3} to compute measure intensities in dimensions $2$ and $3$.

\begin{prop}[{\textsc{Measure face intensities through Isokawa}}] We have 
$$ \tilde{\mathrm{I}}_{2,1} = \frac{2}{\pi}, \quad   \tilde{\mathrm{I}}_{3,2} = \frac{4}{3} ,   \mbox{ and}\quad  \tilde{\mathrm{I}}_{3,1} = \frac{8\pi}{15}  , $$
where $ \tilde{\mathrm{I}}_{d,k}$ is defined in \eqref{eq:defintensities}.
\end{prop}

\rproof
We first note that $2 \cdot \tilde{\mathrm{I}}_{2,1}$ has already been computed in \cite[page 5]{BudzinskiCurienPetri} using~\cite{isokawa2000H2} as the
$\lambda \rightarrow 0$ limit of the mean surface area of the typical cell of $\cells_2^{(\lambda)}$ divided by the mean volume of the typical cell of $\cells_2^{(\lambda)}$.  By \cref{p.typ-face}, this limit amounts to
$2\cdot \tilde{\mathrm{I}}_{2,1}$ since the mean typical cell volume is $1/\lambda$. To compute
$\tilde{\mathrm{I}}_{3,2}$ and $\tilde{\mathrm{I}}_{3,1}$, we proceed similarly and use the mean surface area and mean perimeter length of the typical cell of $\cells_3^{(\lambda)}$ in~\cite[Theorem 1.1]{isokawa2000H3} divided by the mean volume of the typical cell of $\cells_3^{(\lambda)}$ and by the combinatorial factors 2 and 3 respectively, as every 2-dimensional face is common to 2 and every 1-dimensional face is common to 3 cells in $\mathcal V_3^{{(\lambda)}}$. More specifically, by \cref{lem:convintensities},
\begin{equation} \label{eq:Itildes}
\tilde{\mathrm{I}} _{3,2} = \lim_{\lambda \rightarrow 0}  \frac{2^6 \pi^2}{3 \cdot 2} \lambda^2 \cdot J_4 (\lambda) \quad\mbox{and}\quad \tilde{\mathrm{I}} _{3,1} = \lim_{\lambda \rightarrow 0} \frac{2^6 \pi ^4}{5 \cdot 3} \lambda^3 \cdot J_6 (\lambda), 
\end{equation}
where
$$
J_{m}(\lambda) \coloneqq \int_0^\infty  \exp{ [-2\pi \lambda (\sinh r \cosh r -r)]} \sinh ^m r  \, \ud r \quad\text{for } m \in \mathbb{N} \, .
$$

For small $\lambda$, the maximum of the integrand occurs at $ r_{\rm max} = \frac{1}{2} \log (\frac{m}{\pi \lambda}) + o (1)$.  Letting $x \coloneqq r-r_{\rm max}$, the dominated convergence theorem results in
$$
\lim_{\lambda \rightarrow 0} \lambda ^{m/2} \cdot J_m({\lambda}) = \left(\frac{m}{\pi}\right)^{m/2} 2^{-m} \int_{-\infty}^\infty \exp\{ -\frac{m}2 \ue^{2x} + m x\} \,\ud x = \frac{1}{2 (2 \pi)^{m/2}} \Gamma\Big(\frac{m}{2}\Big) \, .
$$
Substituting the right-hand side evaluated at $m=4$ and $m=6$, respectively, into \eqref{eq:Itildes} concludes the proof.
\Qed

\cref{tab:isokawa} gathers these measure face intensities with the corresponding counting face intensities.
An extension of these results is given in \cref{p.isokawa}.

\begin{center}
\begin{tabular}{  m{5em} | m{2cm}| m{2cm} | m{2cm} } 
  
 $\tilde{\mathrm{I}}_{(d,k)}  \Big/  \mathrm{I}_{(d,k)}$ & \centering $k=0$ & \centering $k=1$ & \, \, \,   $k=2$   \\ 
  \hline
  \vspace{0.2 cm} \centering $d=2$ \vspace{0.2 cm} & \centering $\frac{1}{\pi} \Big/ \frac{1}{\pi}$\vspace{0.03 cm} & \centering \centering $\frac{2}{\pi} \Big/ \frac{3}{2\pi}$\vspace{0.03 cm} &  \cellcolor{black!40}  \\ 
  \hline
  \vspace{0.2 cm} \centering $d=3$  \vspace{0.2 cm} & \centering $\frac{16\pi}{35} \Big/ \frac{16\pi}{35}$\vspace{0.03 cm} & \centering $\frac{8\pi}{15} \Big/ \frac{32\pi}{35}$\vspace{0.03 cm}& \centering $\frac{4}{3} \Big/ \frac{16\pi}{35}$\vspace{0.03 cm}
\end{tabular}
\end{center}

\begin{center}
\captionof{table}{Values of $\tilde{\mathrm{I}}_{(d,k)}$ and $  \mathrm{I}_{(d,k)}$ for $d=2$ and $d=3$, $0 \leq k \leq d-1$. Note that the mean length of a typical edge for $d=2$, i.e., $\tilde{\mathrm{I}}_{(2,1)}  \big/  \mathrm{I}_{(2,1)}$ by \cref{p.typ-vol}, is equal to $4/3$, while for $d=3$, it is equal to $7/12$. }
\label{tab:isokawa}
\end{center}

\section{{T}he tile of the origin (zero cell)}\label{s.zerocell}
In this section we will describe the law of the cell $ \mathcal{C}_d$ containing the origin  in the upper half-space model $ \mathbb{U}_{d}$ of hyperbolic space, once the ideal nucleus closest to $\origin$ (the only end of its tile by Theorem \ref{thm:decomposition}) is sent to $\infty$  (Theorem \ref{thm.superposition}). We then use it to compute the hole probability, as well as several other quantitative characteristics of $ \mathcal{C}_d$.
\subsection{A deposition model for the tile of the origin}\label{s.deposition}
Recall the construction of $ \mathcal{V}_{d} = \mathrm{Vor}\bigl( ({\Theta}_{i}, R_{i})_{i\geq 1}\bigr)$ given in \cref{t.weighted}. We will focus here on the cell $ \mathcal{C}_{d}$ containing the origin $ \origin \in \mathbb{H}_{d}$. Recall that by our change of variable, $$R_i = \frac{c_{d}}{d-1} \mathrm{exp}\bigl\{(d-1)D_i\bigr\}$$ are the radii of the nuclei on the corona, where $c_d$ is as in \eqref{e.defc_d}. Clearly, the origin is closer to the nucleus $({\Theta}_{1}, R_{1})$ than to any other ideal nuclei. We will consider the tile $ \mathcal{C}_{d}$ in the upper half-space model and where the ideal nucleus angle ${\Theta}_{1}$ has been sent to $\infty$ in $ \mathbb{U}_{d}$.

\begin{proofof}{\cref{thm.superposition}}
By \cref{prop:bissec}, inside the upper half-space model~$\mathbb{U}_{d}$ and where ${\Theta}_{1}$ is sent to $\infty$, for all $i \geq 2$, the bisector between $({\Theta}_1, R_1)$ and $({\Theta}_i, R_i)$ is a Euclidean hemisphere: denoting the stereographic projection of ${\Theta}_{i}$ as $\mathrm{Ste}({\Theta}_{i}) \in \mathbb{R}^{d-1}$, we may write its center and radius as
\begin{equation*} 
(\mathrm{Ste}({\Theta}_{i}), \sqrt{1+ |\mathrm{Ste}({\Theta}_{i})|^{2}} (R_1/R_i)^{ \frac{1}{2(d-1)}}),
\end{equation*}
or, more usefully, 
$$\Bigl(\mathrm{Ste}({\Theta}_{i}),\sqrt{1+ |\mathrm{Ste}({\Theta}_{i})|^{2}} \; \Bigl(\frac{R_1}{R_1+\left(R_i-R_1\right)}\Bigr)^{\frac{1}{2(d-1)}}\Bigr) \, .$$ 
By Lemma~\ref{lem.stereo}, conditionally on $R_1$, the point process $\left(\mathrm{Ste}({\Theta}_{i}),R_i-R_1\right)_{ i \geq 2 }$ is a PPP with intensity 
$$
\frac{1}{c_d}\frac{1}{\bigl(1+\sum_{i=1}^{d-1} x_i^2 \bigr)^{d-1}} \,\mathrm{d}x_1 \cdots  \mathrm{d}x_{d-1} \otimes   \mathrm{d}t \,\mathbf{1}_{t >0} 
$$
in $ \mathbb{R}^{d-1} \times \mathbb{R}_{+}$.

Now conditionally on $ R_1=s$, we apply the Poisson mapping theorem to compute the intensity of the point process of the centers and radii
by the change of variable $\rho \coloneqq \sqrt{1+|x|^{2}} \bigl(\frac{s}{s+t}\bigr)^{\frac{1}{2(d-1)}}$, which gives, for a test function $u$,
\begin{equation*}
\begin{split}
   \frac{1}{c_d} \int_{\mathbb{R}^{d-1}\times \mathbb{R}_{+}} u\Bigl(x,\sqrt{1+|x|^{2}} &\Bigl(\frac{s}{s+t}\Bigr)^{\frac{1}{2(d-1)}}\Bigr) \frac{1}{\left(1+|x|^{2} \right)^{d-1}}\,   \mathrm{d}x_1 \cdots  \mathrm{d}x_{d-1}\, \mathrm{d}t\\
& =  \frac{d-1}{c_{d}} \int_{\mathbb{R}^{d-1}\times \mathbb{R}_{+}} u\left(x,\rho \right) \frac{2s}{\rho^{2d-1}} \mathbf{1}_{1+|x|^{2}\geq \rho^{2}} \,\mathrm{d}x_1 \cdots  \mathrm{d}x_{d-1} \,  \mathrm{d}\rho.
\end{split}
\end{equation*}
Recalling that $R_{1}$ is an ${\rm Exp}(1)$ random variable and noting that $\frac{d-1}{c_{d}}$ times an ${\rm Exp}(1)$ random variable is an ${\rm Exp}(\frac{c_{d}}{d-1})$ random variable concludes the proof.
\end{proofof}

\subsection{The hole probability}
In this section, we use the preceding construction of the zero cell of $\mathcal{V}_{d}$ to compute  the probability for $\mathcal{C}_d$ to contain a ball centered at the origin: this provides the law of the distance of $\origin$ to $\partial \mathcal{C}_{d}$.

\begin{prop}[\textsc{Hole probability}]\label{prop.holeprob}
The hole probability, i.e., the probability that the ball $B_u(\origin)$ centered at $\origin$ with hyperbolic radius $u$ is contained in $\mathcal{C}_d$, is given by, respectively,
\begin{enumerate}[label=\textup{(\roman*)}]
\item conditional on $R_{1}=s$, 
$$
\mathbb{P}\bigl[B_u(\origin) \subset \mathcal{C}_d \bigm | R_{1}=s \bigr] = {\rm exp}\left(-s I_{d}(u) \right) ;
$$ 

\item averaging on the value of $R_{1}$,
$$
\mathbb{P}[B_u(\origin) \subset \mathcal{C}_d]  = \frac{1}{1+I_{d}(u)}  \, .
$$ 
\end{enumerate}
Here, $I_d(u)$ as defined below in \eqref{eq.idr2} satisfies
\[
1+I_d(u)
=
\frac{2^{d-1} \Gamma(d/2)}{\sqrt\pi\, \Gamma\bigl(\frac{d-1}2\bigr)}
(\cosh u)^{1-d}  \int_{0}^{1}  \frac{ t^{d-2}\left(1-t^{2}\right)^{\frac{d-3}{2}}}{\left(1- t \tanh u \right)^{2d-2}}\, \mathrm{d}t .
\]
In particular, 
$\mathbb{P}\bigl[B_u(\origin) \subset \mathcal{C}_d \bigm | R_{1}=s \bigr] \le \ue^{-s(d-1)u}$ and $\mathbb{P}[B_u(\origin) \subset \mathcal{C}_d] \le \ue^{-(d-1)u}$; that is, the distance from $\origin$ to the boundary of\/ $\ocell_d$ is stochastically dominated by an $\Exp\bigl(s(d-1)\bigr)$ random variable when conditioned on $R_1 = s$ and is stochastically dominated by an $\Exp(d-1)$ random variable unconditionally. Also,
$\mathbb{P}[B_u(\origin) \subset \mathcal{C}_d] \ge \ue^{-2(d-1)u}$, that is, the distance from $\origin$ to the boundary of\/ $\ocell_d$ stochastically dominates an $\Exp\bigl(2(d-1)\bigr)$ random variable.
\end{prop}

\begin{proof} First, we condition on $R_{1}=s$, so that the intensity measure of the Poisson point process $(x,\rho)$ is given by $2s \frac{d-1}{c_{d}} \mathrm{d}x \, \frac{ \mathrm{d}\rho}{\rho^{2d-1}}\mathbf{1}_{ \rho \leq \sqrt{1+|x|^{2}}}$. Second, we parametrize $B_u(\origin)$ in the upper half-space model $\mathbb{U}_{d}$: here, $B_u(\origin)$ is represented by a Euclidean ball $B(C_\origin,R_\origin)$ of Euclidean center $C_\origin=\left(0,\ldots,0,\cosh{u}\right)$ and Euclidean radius $R_\origin=\sinh{u}$.  Hence the event $B_u(\origin) \subset \mathcal{C}_{d}$ corresponds to the event that this point process has no point in the region $\bigl\lbrace  (x,\rho) \in \Pi \ST \left(\rho+R_\origin \right)^{2} \geq |x|^{2}+|C_\origin|^{2}   \bigr\rbrace$. Therefore, by \cref{thm.superposition}, the conditional hole probability is 
$$
\mathbb{P}[B_u(\origin) \subset \mathcal{C}_d \mid R_1 = s] = \text{exp}\left\{-s \, I_{d}(u) \right\} , 
$$
where 
\begin{equation}\label{eq.idr2}
\begin{split}
 I_{d}(u) &\coloneqq \frac{2(d-1)}{c_{d}} \int_{\mathbb{R}^{d-1}}\mathrm{d}x \int_{\mathbb{R}_{+}} \frac{\mathrm{d}\rho}{\rho^{2d-1}} \mathbf{1}_{(\rho + R_\origin)^{2}\geq |x|^{2}+C_\origin^{2}}\times \mathbf{1}_{\rho^2\leq 1+|x|^{2}}\\
 &=\frac{1}{c_{d}} \int_{\mathbb{R}^{d-1}}\mathrm{d}x \frac{1}{\bigl(\sqrt{|C_\origin|^2+|x|^{2}}-R_\origin\bigr)^{2d-2}}-\frac{1}{c_{d}} \int_{\mathbb{R}^{d-1}}\mathrm{d}x \frac{1}{\left(1+|x|^{2}\right)^{d-1}}  \; .\\
 \end{split}
\end{equation}
The equality arises from the fact that the region $\bigl\lbrace  (x,\rho) \in \Pi \ST \left(\rho+R_\origin \right)^{2} \geq |x|^{2}+|C_\origin|^{2}   \bigr\rbrace$ contains the set $\bigl\{(x, \rho) \ST \rho^2 \ge 1 + |x|^2\bigr\}$, which is where the hemisphere corresponding to $(x, \rho)$ contains $\origin$.
The second integral is equal to 1 by \cref{lem.stereo}. For the first integral, change to polar coordinates, then to $\sinh \eta \coloneqq |x|/|C_\origin|$, and finally to $t \coloneqq 1/\cosh \eta$. This gives (i), which immediately implies (ii).

The last unconditional inequalities are equivalent to $\ue^{(d-1)u} \le 1 + I_d(u) \le \ue^{2(d-1)u}$ for all $u \ge 0$. (For the conditional inequality, we also use that $\ue^{(d-1)u} - 1 \ge (d-1)u$.)
These follow from elementary calculations:
Because $1 = \ue^{-2u} + 2(\sinh u) \sqrt{\ue^{-2u}} \le \ue^{-2u} + 2(\sinh u) \sqrt{\ue^{-2u} + |x|^2}$, it follows that $|C_{\origin}|^2 + |x|^2 \le \bigl(\sqrt{\ue^{-2u} +|x|^2} + R_\origin\bigr)^2$. Therefore,
\eqaln
{1 + I_{d}(u)
&=
\frac{1}{c_{d}} \int_{\mathbb{R}^{d-1}} \frac{\mathrm{d}x}{\bigl(\sqrt{|C_\origin|^2+|x|^{2}}-R_\origin\bigr)^{2d-2}}
\ge
\frac{1}{c_{d}} \int_{\mathbb{R}^{d-1}} \frac{\mathrm{d}x}{\left(\ue^{-2u}+|x|^{2}\right)^{d-1}} 
\\ &=
\frac{1}{c_{d}} \int_{\mathbb{R}^{d-1}} \frac{\ue^{(d-1)u}\,\mathrm{d}y}{\left(1+|y|^{2}\right)^{d-1}} = \ue^{(d-1)u}  \, ,
}
where we used the change of variables $y \coloneqq \ue^{u} x$.
Also, $(1-\ue^{-2u})(1+|x|^2) \ge (1-\ue^{-2u})\sqrt{1+|x|^2} = 2 \ue^{-u} (\sinh u) \sqrt{1+|x|^2}$, whence $|C_{\origin}|^2 + |x|^2 \ge \bigl(\ue^{-u} \sqrt{1+|x|^2} + R_{\origin}\bigr)^2$. Hence,
\[
1 + I_{d}(u)
=
\frac{1}{c_{d}} \int_{\mathbb{R}^{d-1}} \frac{\mathrm{d}x}{\bigl(\sqrt{|C_\origin|^2+|x|^{2}}-R_\origin\bigr)^{2d-2}}
\le
\frac{1}{c_{d}} \int_{\mathbb{R}^{d-1}} \frac{\ue^{2(d-1)u}\,\mathrm{d}x}{\left(1+|x|^{2}\right)^{d-1}} = \ue^{2(d-1)u}  \; .
\qedhere
\]\end{proof}

\begin{remark} For $d=2$, we have $I_{2}(u)=\frac{1}{\pi}\bigl(4 \left(\arctan{\mathrm{e}^{u}}\right)\cosh^2{u}+2\sinh{u}\bigr)-1$, which gives, for the integrated hole probability,
\rlabel e.edgehole
  {\mathbb{P}[B_u(\origin) \subset \mathcal{C}_2] = \frac{\pi}{4 \left(\arctan{\mathrm{e}^{u}}\right)\cosh^2{u}+2\sinh{u}} \, ,}
a result first obtained in \cite{bhupatiraju} (Theorem 3.3) by computing the hole probability in Poisson--Voronoi tessellations with positive intensity $\lambda$ on $ \mathbb{H}_d$ and then taking the limit as $\lambda \to 0$.  
It is easily seen, using the change of variables $v \coloneqq 1 - t \tanh u$,
that for odd $d\ge3$, the (integrated) hole probability reduces to a rational function of $\mathrm{e}^{u}$.  For example, when $d=3$ we get
$$
  \mathbb{P}[B_u(\origin) \subset \mathcal{C}_3] =  \frac{3 \mathrm{e}^{-2u}}{2+\mathrm{e}^{2u}}\; .
$$ 
In dimension $2$, differentiating the above expression \eqref{e.edgehole} yields that the density for the distance from $\origin$ to $\partial \cells_2$ is 
\[
u \mapsto \frac{\pi  \cosh u \left(2 (\arctan \ue^u) \sinh u+1\right)} {\bigl(\sinh u+2 (\arctan \ue^u) \cosh ^2u\bigr)^2}\,,
\]
whose value at $0$ is $4/\pi$. The tail probability is asymptotic to $2\ue^{-2u}$ as $u \to\infty$. The mean distance is $0.66137^+$, and the median distance is $0.50264^-$. The tail probabilities are plotted in \Cref{f.edgehole}. These statistics are reflected in the portrait of the cell given by 60 (pseudoindependent, pseudorandom) samples in \Cref{f.2Dportrait}. 
{Here are some additional nice values: for $d = 3$, the mean distance is $3(2-\log 3)/8 = 0.34^-$ and median is $\frac12\log(\sqrt7 - 1) = 0.25^-$, while for $d = 5$, the mean is $5\bigl(279 \log7 - 184 \sqrt6 \arctan(\sqrt6/13) - 378\bigr)/2352 = 0.17^+$ and median is $0.12^+$.}

\begin{figure}[htp] 
\centering
\includegraphics[width=.6\textwidth]{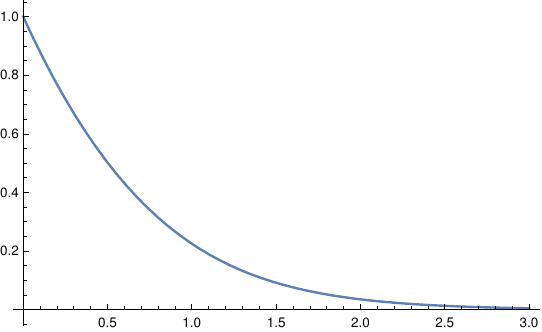}
\caption{The tail probability to be farther than hyperbolic distance $u$ from the ideal Voronoi edges in dimension 2.}
		\label{f.edgehole}
\end{figure}

\begin{figure}[htp] 
\centering
\includegraphics[height=0.95\textheight]{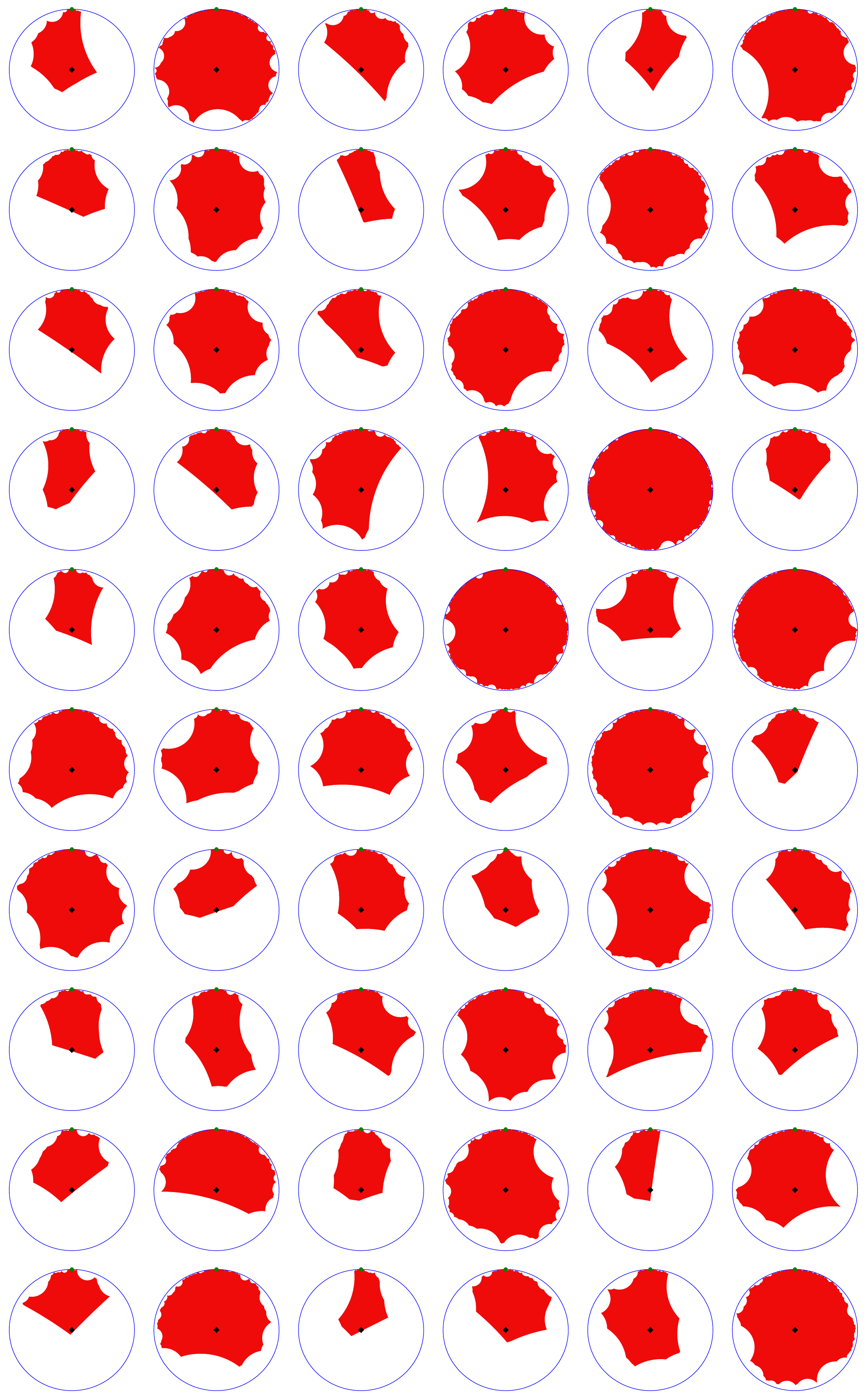}
\caption{A portrait of the cell of the origin in two dimensions given by 60 samples. Each corresponding ideal nucleus is at the top.}
		\label{f.2Dportrait}
\end{figure}
\end{remark}

\begin{prop}[\textsc{Measure face intensities $ \tilde{ \mathrm{I}}_{d,d-1}$}]
\label{p.isokawa}
For all $d \ge 2$, we have
\[
\tilde \Intn_{d,d-1} = \frac{\Gamma (\frac{d}{2})\Gamma(d)}{\Gamma(\frac{d-1}2)\Gamma (d-\frac{1}{2})}.
\]
\end{prop}

{The first few values are $2/\pi$, $4/3$, $32/(5 \pi)$, and $96/35$.
For large $d$, it is asymptotically $ d/\sqrt{2}  -\frac{9}{8\sqrt{2}}+O(\frac{1}{d})$.}

\rproof
Fix $d \geq 2$, and for  $\lambda, \varepsilon>0$,  consider the  probability that the origin is within distance less than $\varepsilon$ from the boundary $\partial \cellsl_d \coloneqq \bigcup \cellsl_{d, d-1}$ of the Poisson--Voronoi tessellation $\cells_d^{(\lambda)}$ with intensity $\lambda$ on $ \mathbb{H}_d$:
$$ h(\varepsilon,d,\lambda) \coloneqq  \mathbb{P}\bigl[ B_\varepsilon( \origin) \cap \partial{\cells_d^{(\lambda)}} \ne \varnothing \bigr].$$  By isometry invariance, $ h(\varepsilon,d,\lambda)$ can be interpreted as the mean volume (per unit volume of $ \mathbb{H}_d$) of the region within distance $\varepsilon$ from  $\partial{\cells_d^{(\lambda)}}$. We then  have the limit
  \begin{equation} \label{eq:limitisokawa} \frac{h(\varepsilon,d,\lambda)}{2\varepsilon} \xrightarrow[ \varepsilon\to 0]{}  \tilde{ \mathrm{I}}_{d,d-1}^{{( \lambda)}} \xrightarrow[ \lambda \to 0]{ \mathrm{Lem.\ \ref{lem:convintensities}}} \tilde{ \mathrm{I}}_{d,d-1}    \end{equation} 
and the one in which we interchange the order of limits, which can be evaluated explicitly thanks to our results,  
 \begin{eqnarray*}
 \lim_{\varepsilon \to 0} \lim_{\lambda \to 0} \frac{h(\varepsilon,d,\lambda)}{2\varepsilon} &\underset{  \mathrm{Thm.}\ \ref{thm:decomposition}}{=} & \lim_{\varepsilon\to 0} \frac{ \mathbb{P}\bigl[ B_\varepsilon(\origin) \cap \partial \mathcal{V}_d \ne \varnothing\bigr]}{2\varepsilon}\\
  & {=} &  - \frac{1}{2}\cdot  \left. \frac{ \mathrm{d}}{ \mathrm{d} \varepsilon} \mathbb{P}\bigl[B_\varepsilon(\origin) \subset \mathcal{C}_d\bigr] \right|_{\varepsilon=0}\\ &\underset{ \mathrm{Prop.}\ \ref{prop.holeprob}}{=}&  \frac{\Gamma (\frac{d}{2})\Gamma(d)}{\Gamma(\frac{d-1}2)\Gamma (d-\frac{1}{2})}.
  \end{eqnarray*}
  We claim that exchanging the order of limits gives the same result, whence the above display combined with \eqref{eq:limitisokawa} computes $ \tilde{\mathrm{I}}_{d,d-1}$.  

To establish our claim, suppose that we have a sequence $\mathbf C^{(j)} = \bigl((C_i^{(j)} \ST i \geq 1)\bigr)_{j \ge 1}$ of normal Voronoi tessellations that converges to a normal, ideal Voronoi tessellation $\mathbf C = (C_i \ST i \geq 1)$ as $ j \to\infty$. Fix $1 \le k \le d-1$. Denote the $\varepsilon$-neighborhood of the $k$-skeleton of a tessellation $\mathbf C$ by $\mathbf C(\varepsilon)$. Suppose that $x$ is an interior point of a $k$-face of a cell, $C_i$. Then $x$ is a limit of interior points of $k$-faces of $C_i^{(j)}$ but not of any other $k$-faces, because faces are convex and the tessellations are normal. This implies that $\varepsilon^{-1} \Vol\bigl(\mathbf C^{(j)}(\varepsilon) \cap B_1(\origin)\bigr)$ converges to $\varepsilon^{-1} \Vol\bigl(\mathbf C(\varepsilon) \cap B_1(\origin)\bigr)$ as $j \to\infty$ uniformly in $\varepsilon \in (0, 1)$, whence $\lim_{j \to\infty} \lim_{\varepsilon \to 0} \varepsilon^{-1} \Vol\bigl(\mathbf C^{(j)}(\varepsilon) \cap B_1(\origin)\bigr) =  \lim_{\varepsilon \to 0} \lim_{j \to\infty} \varepsilon^{-1} \Vol\bigl(\mathbf C^{(j)}(\varepsilon) \cap B_1(\origin)\bigr)$. In particular, this holds a.s.\ for ${\cells_d^{(\lambda)}}$ and $\cells_d$ if we take a coupling that converges a.s.\ as $\lambda \to 0$. As in the proof of Lemma \ref{lem:convintensities}, uniform integrability shows that
\eqaln{
\lim_{\lambda \to 0} \lim_{\varepsilon \to 0} \EE\bigl[\varepsilon^{-1} \Vol\bigl({\cells_d^{(\lambda)}}(\varepsilon) \cap B_1(\origin)\bigr)\bigr] 
&=
\EE\bigl[\lim_{\lambda \to 0} \lim_{\varepsilon \to 0} \varepsilon^{-1} \Vol\bigl({\cells_d^{(\lambda)}}(\varepsilon) \cap B_1(\origin)\bigr)\bigr] 
\\ &=
\EE\bigl[\lim_{\varepsilon \to 0} \lim_{\lambda \to 0} \varepsilon^{-1} \Vol\bigl({\cells_d^{(\lambda)}}(\varepsilon) \cap B_1(\origin)\bigr)\bigr]
\\ &=
\lim_{\varepsilon \to 0} \lim_{\lambda \to 0} \EE\bigl[\varepsilon^{-1} \Vol\bigl({\cells_d^{(\lambda)}}(\varepsilon) \cap B_1(\origin)\bigr)\bigr].
}
Choose $k = d-1$.  Since $\EE\bigl[\Vol\bigl({\cells_d^{(\lambda)}}(\varepsilon) \cap B_1(\origin)\bigr)\bigr] = h(\varepsilon, d, \lambda) \Vol\bigl(B_1(\origin)\bigr)$, dividing by $\Vol\bigl(B_1(\origin)\bigr)$ gives our claim.
\Qed

It would be interesting to compute the hole probabilities for lower-dimensional faces, in particular for vertices,  but the calculations seem considerably more intricate. As in the preceding proposition, this would provide one way to compute $\tilde {\mathrm I}_{d, k}$ for other $k$.

\subsection{Asymptotic properties of the zero cell: the typical cell}

In this section, we give the basic properties of the underlying stationary model defining the law of $ \mathcal{C}_d$ far away from the origin. It is obtained from the original deposition model of Theorem~\ref{thm.superposition} by removing the indicator function in the intensity \eqref{eq.intm}, which ensured that no ball would contain the origin.

More precisely, following the proof of Theorem~\ref{thm.superposition} in \Cref{s.deposition}, conditioning on the value of the smallest radius $R_1 =s$, we denote by $ \scell_{d} =\scell_d(s)$ the random closed subset obtained by removing the balls whose centers and radii $(x,\rho)$ are distributed according to a Poisson point process with intensity 
\begin{equation}\label{eq.intmes}
 2 \frac{d-1}{c_{d}}{s}\, \mathrm{d}x \otimes \frac{ \mathrm{d}\rho}{\rho^{2d-1}}.
 \end{equation}
In their natural coupling, the sets $\scell_d{(R_1)}$ and $\ocell_d$ are created by the same PPP of balls except that $\scell_d(R_1)$ uses $\Pois(R_1)$ more balls given $R_1$,
as we saw in evaluating the last integral of \eqref{eq.idr2}.
Hence $\scell_d$ uses $\Geom(1/2) - 1$ more balls than $\ocell_d$, so that $\scell_d = \ocell_d$ with probability 1/2. (The fact that an $\Exp(1)$ mixture of Poisson random variables is $\Geom(1/2) - 1$ can be seen by classifying the events of a Poisson process of rate 2 into two equally likely types, then counting those of type 1 that appear before the first event of type 2.)
In particular, the asymptotic properties of $\ocell_d$ and of $\scell_d$ are the same a.s. Clearly, the law of $\scell_d$ is $\R^{d-1}$-invariant and ergodic. Generally, we are interested only in geometric properties of $\scell_d$ that are $\mob_d$-invariant. For example, given $t > 0$, the homothety $(x, \rho) \mapsto (tx, t\rho)$ of $\uhs_d = \R^{d-1} \times \R_+$ is an isometry that preserves $\infty$. It corresponds to changing $s$ in \eqref{eq.intmes} to $t^{d-1}s$. Thus, $\scell_d(s)$ for all $s > 0$ can be coupled to be isometric to each other. Hence, in studying $\scell_d(s)$, the value of $s$ will mainly be important for the natural coupling of $\ocell_d$ with $\scell_d$.

{
By virtue of \cref{prop:bissec}, another way to describe $\scell_d(s)$ is that it is the ideal Voronoi cell of the ideal nucleus $(\infty, s)$ in $\Vor\bigl(\XX \cup \{(\infty, s)\}\bigr)$, where $\XX$ is a PPP with intensity $\mu_d$ on $\corona$. Because $\mob_d$ acts transitively on $\corona$ preserving $\mu_d$ (\cref{lem:mobius}), we have that $\scell_d(s)$ has the same law as the cell $C(Y)$ of $Y$ in $\Vor\bigl(\XX \cup \{Y\}\bigr)$ for any fixed $Y \in \corona$, provided that we restrict to the $\mob_d$-invariant $\sigma$-field.

Consider now the typical cell $\tcell_d^{(\lambda)}$ in $\Vor(\XXl)$ for $\lambda > 0$. As before, we use the nucleus of a cell for its center, so that the typical cell is the cell of a typical nucleus in $\XXl$. By Slivnyak's theorem, as explained in \cref{s.palm-hyp}, it follows that $\tcell_d^{(\lambda)}$ is the cell of $\origin$ in $\Vor\bigl(\XXl \cup \{\origin\}\bigr)$. Again, this has the same law on the $\mob_d$-invariant $\sigma$-field as the cell $C(z)$ of $z$ in $\Vor\bigl(\XXl \cup \{z\}\bigr)$ for any $z \in \HH_d$. Consider temporarily the ball model of $\HH_d$ and fix any $z \ne \origin$. Let us use the dilations $\tilde\XX^{(\lambda)}$ of \cref{remark_dilation} together with the corresponding dilations $z^{(\lambda)}$, and let $\lambda \to 0$. The cell $C(z^{(\lambda)})$ in $\mathrm{Vor}(\tilde\XX^{(\lambda)} \cup \{z^{(\lambda)}\})$ will tend to the cell of $Y \coloneqq \lim_{\lambda \to 0} z^{(\lambda)}$ in $\mathrm{Vor}\bigl((\Theta, \RR) \cup \{Y\}\bigr)$ in the Fell topology on $\HH_d$, whose law is thus that of $\scell_d$. Therefore, it is reasonable to call $\scell_d$ the \dfn{typical (ideal) cell}.}

{There is another reason for that nomenclature.
Namely, we have constructed our IPVT based on a PPP on the corona, $\corona$. We regard this as a marked PPP on $\corona$ with the mark of an ideal nucleus being its cell in $\HH_d$. In this way, we may define a typical cell of $\cells_d$ as the cell of a typical ideal nucleus. However, the action of $\mob_d$ on $\corona$ does not yield compact isotropy groups. Instead, we may use an alternative notion of Palm distribution based on disintegrations \cite[Chapter 6]{Kall:RM}. For a PPP $\XX$ on a standard Borel space $E$, 
the Palm distribution at $x \in E$ is $\XX + \delta_x$ according to Mecke's theorem \cite[Lemma 6.15]{Kall:RM}. In our case, if $\XX$ is a PPP with intensity $\mu_d$ on $\corona$ and $Y \in \corona$, then the typical cell corresponding to $Y$ is the cell of $Y$ in $\Vor\bigl(\XX \cup \{Y\}\bigr)$. Transitivity allows us to call any fixed $Y$ typical.}

\subsubsection{Height and angle of the boundary}
The typical cell $ \scell_{d}$ can be seen as the epigraph (see Figure \ref{fig.cell_A} for the graph) of a stationary random function $ \widetilde{\mathcal{H}} \colon \mathbb{R}^{{d-1}} \to \mathbb{R}_{+}$ whose law is computed below. Specifically, for $x_{0} \in \mathbb{R}^{d-1}$, let us denote by $\widetilde{\mathcal{H}}(x_{0})$ the height of $ \partial \scell_{d}$ at $x_{0}$ and by $\Psi(x_{0})$ the angle the hypersurface $ \partial \scell_{d}$ makes with the  vertical direction at that point  (in other words, the complement of the angle between the vertical direction and the direction orthogonal to the hyper{surface}). It should be clear that this angle is defined for almost all $x_{0}  \in \mathbb{R}^{{d-1}}$, and since the law of $\bigl(\widetilde{\mathcal{H}}(x_{0}),\Psi(x_{0})\bigr)$ is independent of $x_{0}$, we denote it for short by $(\mathcal{H}, \Theta)$.

\begin{prop}[\textsc{Height and angle of the boundary}]\label{prop.hscpl} 
For $d\geq2$, the law of $(\mathcal{H}, \Theta)$ satisfies
\begin{enumerate}[label=\textup{(\roman*)}]
\item conditional on $R_{1}=s$, 
$$
\Bigl(\frac{1}{ \mathcal{H} ^{d-1}}, \sin^2(\Theta) \Bigr) \sim  {\rm{Exp}} \left( s \right)\otimes\, {\rm Beta}\Bigl(\frac{d+1}{2},\frac{d-1}{2}\Bigr);
$$ 
\item averaging on the value of $R_{1}$,
$$
\bigl( \mathcal{H}^{d-1}, \sin^2(\Theta) \bigr) \sim  \frac{1-U}{U} \otimes\, {\rm Beta}\Bigl(\frac{d+1}{2},\frac{d-1}{2}\Bigr),
$$ 
where $U$ is a uniform random variable over $[0,1]$. 
\end{enumerate}
\end{prop}
\begin{proof} \textbf{(i)} \enspace Without loss of generality, we consider the variables $(\mathcal{H}, \Theta) = \bigl(\widetilde{\mathcal{H}}(\bfz),\Psi(\bfz)\bigr)$.
Condition on $R_{1}=s$. In a preliminary step, we study the first marginal of this pair, namely, the law of $\mathcal{H}$. Recall that the intensity measure of the Poisson point process $\Pi_d$ is given by $(x,\rho) \mapsto 2s \frac{d-1}{c_{d}} \mathrm{d}x \otimes \frac{ \mathrm{d}\rho}{\rho^{2d-1}}$.  The event $\mathcal{H}\leq h$ corresponds to the event where this point process has no point inside the region $\bigl\lbrace (x,\rho) \ST \rho^{2} > |x|^{2} + h^{2}\bigr\rbrace$. Hence, changing variables to $y \coloneqq x/h$ and $\sigma \coloneqq \rho/h$ gives
\begin{equation}\label{eq.phq}
\begin{split}
\mathbb{P}[\mathcal{H} \leq h ]
&=
\text{exp}\Bigl\{-2s \frac{d-1}{c_{d}}\, \int_{\mathbb{R}^{d-1}\times \mathbb{R_{+}}} \mathrm{d}x \frac{ \mathrm{d}\rho}{\rho^{2d-1}} \mathbf{1}_{\rho^{2}\geq h^{2} + |x|^{2}}\Bigr\} \\
&=
\text{exp}\Bigl\{-2s \frac{d-1}{c_{d}h^{d-1}}\, \int_{\mathbb{R}^{d-1}\times \mathbb{R_{+}}} \mathrm{d}y \frac{ \mathrm{d}\sigma}{\sigma^{2d-1}} \mathbf{1}_{\sigma^{2}\geq 1 + |y|^{2}}\Bigr\} \\
&= \exp\Bigl\{- \frac{s}{h^{d-1}} \Bigr\}  .
\end{split}
\end{equation}
Let $g{\colon \R_+ \times [0, \pi/2) \to \R_+}$ be a nonnegative Borel function. For $|x| \le \rho$, let $A(x, \rho)$ be the event that {$\rho^2 - |x|^2 = \max\{\rho'^2-|x'|^2 \sut (x', \rho') \in \Pi_d\}$}. Then
$$
g(\mathcal{H},\Theta) = \sum_{\substack{ (x,\rho) \in \Pi_d, \\ |x|\leq \rho}} g\Bigl(\sqrt{\rho^2-|x|^2},\arccos{\frac{|x|}{\rho}}\Bigr) \mathbf{1}_{A(x, \rho)}.
$$
Also, $\Pi_d\setminus \bigl\{(x, \rho)\bigr\}$ has the same law as $\Pi_d$ itself. Thus, by Mecke's formula,
$$
\mathbb{E}\bigl[g(\mathcal{H},\Theta) \bigr]  = 2s \frac{d-1}{c_{d}} \int_{\mathbb{R}^{d-1}\times \mathbb{R_{+}}}  \frac{\mathrm{d}x \,\mathrm{d}\rho}{\rho^{2d-1}}\, g\Bigl(\sqrt{\rho^{2}-|x|^{2}}\,,\arccos{\frac{|x|}{\rho}}\Bigr)  \mathbf{1}_{|x|\leq\rho}\, \mathbb{P}\bigl[\mathcal{H} \leq \sqrt{\rho^{2}-|x|^{2}}\, \bigr]
.
$$
Using polar coordinates and \eqref{eq.phq} yields
\begin{equation}\label{eq.condexpg}
\begin{split}
\mathbb{E}\bigl[g(\mathcal{H},\Theta) \bigr]
&= 2s \frac{d-1}{c_{d}}  \ \Omega_{d-1}  \int_{\mathbb{R}_{+}\times \mathbb{R}_{+}}\mathrm{d}r \, r^{d-2}  \frac{ \mathrm{d}\rho}{\rho^{2d-1}}  g\Bigl(\sqrt{\rho^{2}-r^{2}}\,,\arccos{\frac{r}{\rho}}\Bigr)  \mathbf{1}_{r\leq\rho}\, \mathbb{P}[\mathcal{H} \leq \sqrt{\rho^2-r^2}\, ]\\
&=    \int_{\mathbb{R}_{+}\times [0,\pi/2)}  \Bigl((d-1)\frac{s}{h^{d}} \ \exp\Bigl\{-\frac{s}{h^{d-1}}\Bigr\} \Bigr) \Bigl( 2^{d}\frac{ \Gamma(\frac{d}{2}) }{\sqrt{\pi}\,\Gamma(\frac{d-1}{2})}  \sin^{d}{\theta}\, \cos^{d-2}{\theta} \,  \Bigr)  
g(h,\theta)\,\mathrm{d}h\,\mathrm{d}\theta,\\
\end{split}
\end{equation}
which implies (i). 

\textbf{(ii)} \enspace
Recalling that $R_{1} \sim \rm{Exp}(1)$, we find that $\PP[\mathcal H^{d-1} < t] = t/(t+1)$ for $t > 0$, i.e., $\mathcal H^{d-1} \sim (1-U)/U$, as desired.
\end{proof}

\subsubsection{Intensities of the Laguerre tessellation $\GKT_{{d-1}}$}\label{s.asymp0cell}
For $0 \leq k \leq d-2$, we can define the $k$-dimensional facets of the hypersurface $ \partial \scell_{d}$. After orthogonal projection, this defines a random tessellation $ \GKT_{{d-1}}$ of $ \mathbb{R}^{d-1}$ whose law is stationary; see \Cref{f.foam-views,f.foam-top}. 

\begin{figure}[!h]
 \begin{center}
 \includegraphics[width=.32\linewidth]{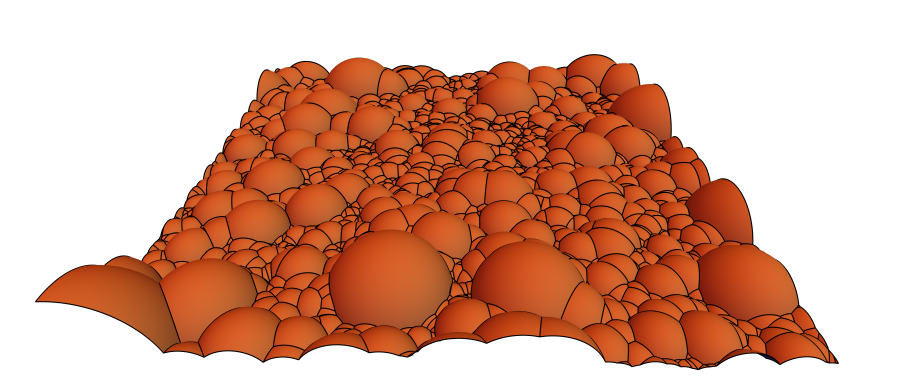}
  \includegraphics[width=.32\linewidth]{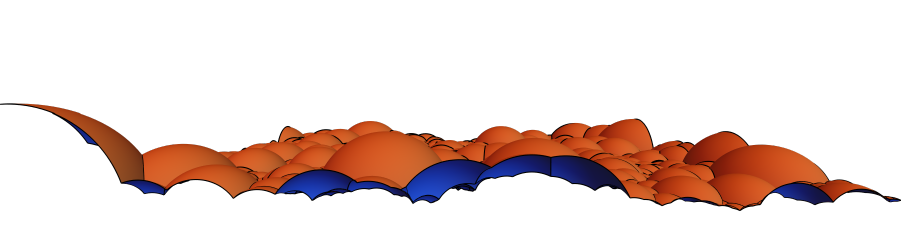}
   \includegraphics[width=.32\linewidth]{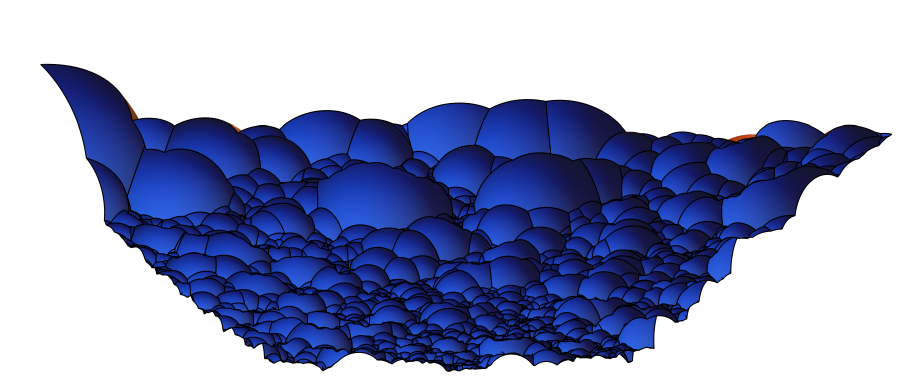}
 \caption{A piece of the hypersurface $\partial \scell_{{3}}$ seen from various angles. After orthogonal projection, it yields a tessellation $ \GKT_{{2}}$ of $  \mathbb{R}^{{2}}$.}
\label{f.foam-views}
 \end{center}
 \end{figure}
\begin{figure}[!h]
 \begin{center}
 \includegraphics[width=.5\linewidth]{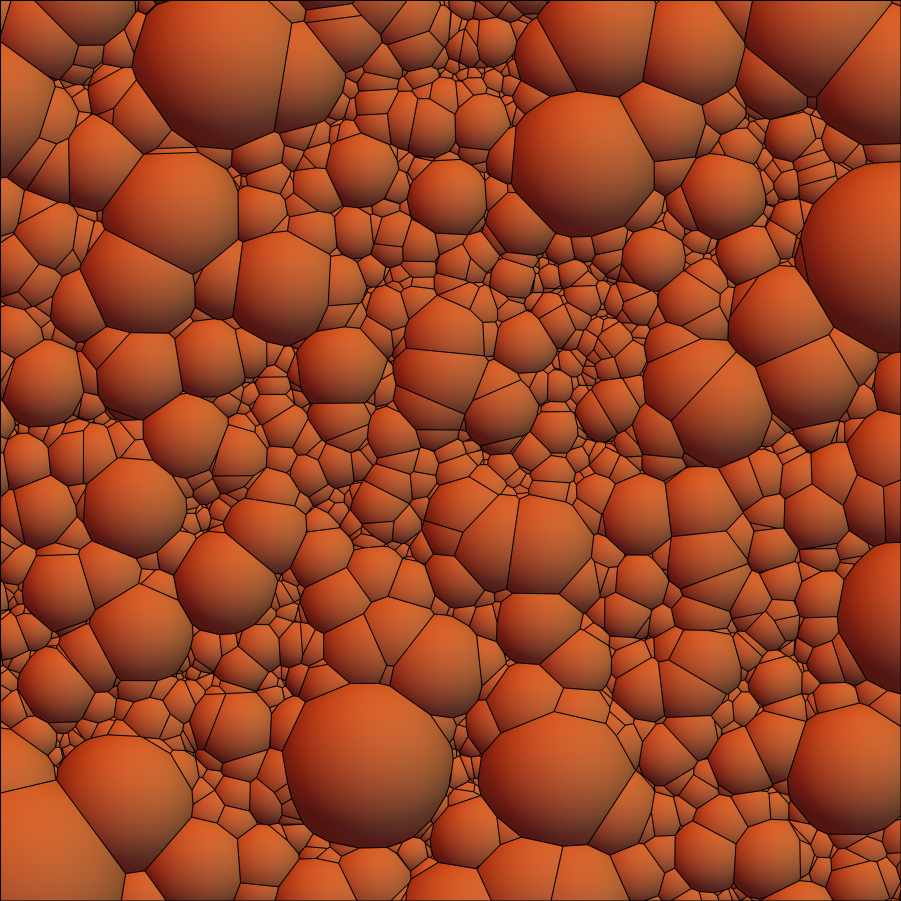}
 \caption{A piece of the hypersurface $\partial \scell_{{3}}$ seen from above, which also shows the tessellation $ \GKT_{{2}}$ of $  \mathbb{R}^{2}$.}
\label{f.foam-top}
 \end{center}
 \end{figure}

 This tessellation $\GKT_{{d-1}}$ is actually a special case of point processes studied in \cite{gusakova2022delaunay}: the relevant process for us is what they call the $\beta'$-model with certain choices for their parameters $\beta$ and $\gamma$. They use the PPP on $\R^{d-1} \times \R_-$ of nuclei and heights with intensity given in \cite[(3.2)]{gusakova2022delaunay}, where their $-h$  corresponds to our $\rho^2$: see the last line of \cite[p.~1258]{gusakova2022delaunay}. The point is that the Laguerre diagram in $\R^{d-1}$ corresponding to their power function is precisely the tessellation $\GKT_{{d-1}}$ for an appropriate choice of their parameters, $\beta$ and $\gamma$.  We use $\beta \coloneqq d$. Their other parameter  $\gamma$ is a scaling factor. To be able to make the correspondence, we need to match our intensity in \eqref{eq.intmes} with their intensity.  We will condition that $R_1 = s$. Writing \eqref{eq.intmes} as 
\vadjust{\kern2pt}%
$s\frac{d-1}{c_d} \,\ud x \,\ud(\rho^2)/(\rho^2)^d$, 
we see that for the two models to match, we need to choose $\gamma \coloneqq s \sqrt\pi\,\Gamma(d/2)/\Gamma\bigl(\frac{d-1}2\bigr)$.

In \cite[Theorem 2]{gusakova2022delaunay}, they compute moments of volumes of weighted typical cells of the Delaunay tessellation corresponding to their Laguerre diagram. Of most interest to us is the case where their $\beta \coloneqq d$ as we said above, their $\nu \coloneqq 0$, and their $s \coloneqq 1$. 
This yields after some simplification the following:

\begin{lem}\label{lem.DelVol} 
The mean (Euclidean) volume of the typical Delaunay cell of $\GKT_{{d-1}}$ is $s\cdot \DelVol_d$, where
\[
\DelVol_d
\coloneqq
\frac{\Gamma(\frac{d+1}2)}{2\cdot (d-2)!} \Bigl(\frac{\Gamma(\frac{d-1}2)}{\Gamma(d/2)}\Bigr)^{\!d} \,\frac{\Gamma(d^2/2)}{\Gamma(\frac{d^2+1}2)}.
\tag*{\qed}
\]
\end{lem}

For example, $\DelVol_d$ for $2 \le d \le 5$ takes the following values: $\pi/3$, $35/(32 \pi)$, $(16 \pi^2)/2145$, and $676039/(5971968 \pi^2)$.

\cite[Proposition 3]{gusakova2022delaunay} gives the face intensities of the Voronoi tessellation in terms of those of the Delaunay tessellation by duality. The latter are given on the same page in Theorem 6.\footnote{However, there are two typos: Here and before, the factor of 1 should be $\ui$ and the subscript of $c'$  in the last line should have $+1$  in place of $-1$. These mistakes stem from a mistake at the bottom of p.~1282 when quoting ``[14] (see Theorem 1.7 and the discussion thereafter)'' and then a slight miscalculation.}
This yields the following:

\begin{lem}[{\textsc{Counting face intensities}}]\label{lem.EuclFaceInt} 
For $0 \le k \le d-1$, the (Euclidean) counting intensity of the $k$-faces of $\GKT_{{d-1}}$ and thus of $\partial\scell_d{(s)}$ is
\[
\frac{\JJ_{d, k}}{s\cdot \DelVol_d},
\]
where $\JJ_{d, k}$ is defined in \eqref{e.defJJ}. \qed
\end{lem}

The values of $(\JJ_{d, k}/\DelVol_d \ST 0 \le k \le d-1)$ for $2 \le d \le 5$ are $(3/\pi, 3/\pi)$, $(32 \pi/35, 48 \pi/35, 16 \pi/35)$, $\bigl(2145/(16 \pi^2), 2145/(8 \pi^2), 1275/(8 \pi^2), 405/(16 \pi^2)\bigr)$, and 
\[
\Bigl(\frac{5971968 \pi^2}{676039}, \frac{14929920 \pi^2}{676039}, \frac{641088}{245245} +  \frac{9953280 \pi^2}{676039}, \frac{4461632}{245245}, \frac{820544}{245245} - \frac{995328 \pi^2}{676039}\Bigr).
\]

This allows us to deduce the mean $f$-vector of the typical cell in $\GKT_{{d-1}}$, which equals the mean $f$-vector of the typical boundary cell of $\scell_d$:

\begin{thm}[{\textsc{Mean face vectors}}]\label{t.f-vector}\!\!\!\footnote{This is a special case of \cite[Theorem 7]{gusakova2022delaunay} as stated for $\GKT_{{d-1}}$, but there are several typos there: The denominator in both right-hand sides should instead be multiplying the middle quantities, and both the middle and right-hand sides should be divided by $\mathbb{J}^{(\raisebox{-4pt}{$'$})}_{d,1}(\beta-1/2)$. The displayed equation in the proof is missing a factor on the right-hand side of $\gamma_{d-1}(\mathcal{V}_\beta)$, which equals $\mathbb{J}_{d,1}(\beta-1/2)/\mathbb{E}[\mathrm{Vol}(Z_{\beta,0})]$.} 
For $0 \le k \le d-2$, the mean number of $k$-faces of the typical cell of the boundary of $\scell_d$ is
$
(d-k)\JJ_{d, k}/\JJ_{d,d-1}.
$
\end{thm}

\rproof 
Each $k$-face belongs to $d-k$ cells a.s., being the intersection of that many cells. Thus, the result follows from \cite[Theorem 10.1.2]{SchneiderWeil} and \cref{lem.EuclFaceInt}.
\Qed

The values for the mean $f$-vectors for small $d$ are as follows: $(6, 6)$ for $d = 3$; $(572/27, 286/9,\allowbreak 340/27)$ for $d = 4$; and 
\[
\frac{(139708800 \pi^2,
{279417600 \pi^2},
{6 (62173301 + 23284800 \pi^2)},
{373039806 \pi^2})}
{62173301 - 4656960 \pi^2}
\]
for $d = 5$. When $d = 3$, the 1-skeleton of $\scell_3$ is a 3-regular graph a.s., whence one can deduce already from Euler's formula that the mean number of sides per face is 6.

\begin{remark}\label{r.f-vector}
{Comparing \cref{t.f-vector} with \cref{c.f-vector} with $\ell \coloneqq d-1$ there shows agreement when we note that each $k$-face of $\partial \scell_d$ belongs to $d-k$ cells of $\partial \scell_d$, but we cannot explain this coincidence.}
\end{remark}

\subsubsection{Intensities of $\partial \scell_d$}

Beware that \cref{lem.EuclFaceInt} deals only with the intensities of the projection of $ \partial \scell_{d}$. But we can combine it with \cref{prop.hscpl} to deduce some information on $ \partial \scell_{d} \in \mathbb{R}^{{d-1}} \times  \mathbb{R}_{+}$ in our next result:

\begin{cor}[\textsc{Mean volume of boundary cells}]\label{cor.mvol}
{T}he volume of the typical $(d-1)$-face of $\partial \scell_d{(s)}$ is
\[
\frac{d-1}{2\JJ_{d,d-1}} \Bigl(\frac{\Gamma(\frac{d-1}2)}{\Gamma(d/2)}\Bigr)^{\!d-1} \,\frac{\Gamma(d^2/2)}{\Gamma(\frac{d^2+1}2)}\, \Gamma(d-1/2).
\]
\end{cor}

\rproof
By \cref{lem.EuclFaceInt} and the ergodic theorem (e.g., \cite[Theorem 9.3.1]{SchneiderWeil}), the number of $(d-1)$-faces of $\GKT_{{d-1}}$ in a $(d-1)$-dimensional Euclidean ball of radius $\rho$ about $\bfz$ is a.s.\ asymptotic to $b_\rho\JJ_{d, d-1}/(s \cdot \DelVol_d)$, where $b_\rho$ is the Euclidean $(d-1)$-volume of that ball. By \cref{prop.hscpl}, the hyperbolic $(d-1)$-volume of the union of those faces is a.s.\ asymptotic to 
\eqaln
{b_\rho \EE[\mathcal{H}^{1-d}/\sin \Theta] 
&= b_\rho \EE[\mathcal{H}^{1-d}] \EE[1/\sin \Theta] 
= b_\rho \frac{B\bigl(d/2,(d-1)/2\bigr)}{s \cdot B\bigl((d+1)/2, (d-1)/2\bigr)}
\\ &= b_\rho \frac{\Gamma(d/2)\Gamma(d)}{s \cdot \Gamma(d-1/2) \Gamma\bigl((d+1)/2\bigr)}.}
Dividing the latter by the former gives the result.
\Qed

The first few values of these volumes for $d = 2, 3, 4, 5$ are $4/3$, $35/(12 \pi)$, $1024 \pi/6075$, and $52055003/( 746079612 - 55883520 \pi^2)$.

\begin{remark}[\textsc{{Measure $(d-1)$-face intensities}}]\label{r.facet-inten}
We may regard ${\scell}_d$ as the union of the regions over its boundary $(d-1)$-faces (that is, the cones of those faces with apex at $\infty$). We may then compare the typical face volume given by \cref{cor.mvol} with the volume of its cone, which is, by similar reasoning, $\EE[\int_{\mathcal H}^\infty dy/y^d] = \EE[\mathcal{H}^{1-d}/(d-1)] = 1/s(d-1)$ divided by $\JJ_{d,d-1}/(s\cdot\DelVol_d)$. Dividing the typical face volume by this yields 
\[
\frac{(d-1)\Gamma(d/2)\Gamma(d)}{\Gamma(d-1/2) \Gamma\bigl((d+1)/2\bigr)}.
\]
This agrees with the measure face intensity we obtained in \cref{p.isokawa} after accounting for the fact that we are comparing here to the volume of the cone on only one side of the $(d-1)$-face, although we do not have an explanation for this coincidence. Of course, we did not need \cref{lem.EuclFaceInt}, nor even the distribution of $\mathcal H$; the result is simply $\EE[\mathcal{H}^{1-d}/\sin \Theta]/\EE[\int_{\mathcal H}^\infty dy/y^d] = (d-1)\EE[1/\sin \Theta]$.
\end{remark}

\begin{remark}[\textsc{Hyperbolic vs.\ Euclidean volumes of Delaunay simplices}]\label{r.Del-vols}
If we compare $\DelVol_d$ with $\IDV_d$ defined in \eqref{e.defIDV}, we discover that $\IDV_d = \frac{d+1}{d-1}\DelVol_d$ for all $d \ge 2$. A heuristic explanation for this coincidence follows. We desire to explain why the hyperbolic intensity of the vertices of $\cells_d$ (i.e., $1/\IDV_d$) is equal to $s(d-1)/(d+1)$ times the Euclidean intensity of the vertices of $\GKT_{d-1}$ conditional on $R_1 = s$ (i.e., $1/(s\cdot \DelVol_d$)). First, note that the latter is the Euclidean intensity of the vertices of $\scell_d$, disregarding their heights. Next, if we think of the hyperbolic intensity of the vertices of $\scell_d$ per unit volume of the zero cell $\ocell_d$, then we can convert Euclidean $(d-1)$-volume to hyperbolic $d$-volume by multiplying by $\EE[\int_{\mathcal H}^\infty dy/y^d] = 1/s(d-1)$, thus multiplying the Euclidean intensity by $s(d-1)$. However, each vertex belongs to $d+1$ cells of $\cells_d$, so to get the hyperbolic intensity of vertices per unit volume of $\HH_d$, we should divide this by $d+1$.

Alternatively, we can say that the ideal Delaunay simplices associated to the vertices of $\scell_d$ are the simplices whose ideal vertices are $\infty$ and the vertices of the Delaunay simplices of $\GKT_{d-1}$ (which are some of the nuclei of the Laguerre diagram). The portions of their hyperbolic volumes that lie in the cell of $\infty$ is $1/s(d-1)$ times their Euclidean volumes. Because they have $d+1$ vertices, their total hyperbolic volumes are $\frac{d+1}{s(d-1)}$ times their Euclidean volumes.
\end{remark}

We conclude this section by computing the intensity of the vertices in $ \partial \scell_{d}$. Note that the vertices do not form a Poisson point process; they are merely a point process.

\begin{prop}[\textsc{Vertex intensity}]\label{prop.vertint}
{T}he process of vertices on the hypersurface $\partial \scell_d{(s)}$ has the intensity
\[
 \frac1{s \cdot \DelVol_d} \Leb \otimes \zeta,
\]
where $\zeta$ is the probability distribution of a random variable $Z$ when $1/Z^{d-1}$ has law $\text{\rm Gamma}(d+1, s)$. Thus, for $z > 0$,
\[
\ud\zeta(z)
=
\frac{(d-1)s^d}{d!}\frac{\mathrm{e}^{-\frac{s}{z^{d-1}}}}{z^{d^{2}}} \, \mathrm{d} z .
\]
\end{prop}

\begin{proof}
Because the vertices of $\GKT_{{d-1}}$ have finite intensity, the vertices of $\partial \scell_d$ form a marked point process. By stationarity, the intensity measure of the vertices of $\partial\scell_d$ is a constant times $\Leb \otimes \zeta$ for some probability measure, $\zeta$ \cite[Theorem 3.5.1]{SchneiderWeil}. By \cref{lem.EuclFaceInt}, the constant is $1/(s\cdot \DelVol_d)$. It remains to find $\zeta$.

Let $f$ be a nonnegative Borel function defined on $\uhs_d = \R^{d-1} \times \R_+$. We will use $\harp{x}$ to denote points in $\R^{d-1}$. %
For $\sigma \coloneqq (\harp{x_1}, \dots, \harp{x_d}, \rho_1, \dots, \rho_d) \in (\mathbb{R}^{d-1})^{d}\times (\mathbb{R}_{+})^{d}$, write $v(\sigma) \coloneqq \bigcap_{i=1}^{d} \partial B(\harp{x_{i}},\rho_{i})$. 
Let $A$ be the set of $\sigma$ where $v(\sigma)$ 
is a single point.
By the Slivnyak--Mecke formula and \eqref{eq.phq}, there is a constant $\alpha_{d{, s}}$ whose value does not concern us such that 
\begin{equation}\label{eq.msf}
\mathbb{E}\Biggl[\sum_{\substack{v \text{ a vertex}\\[1pt] \text{of}\;  \scell_d}}f(v) \Biggr] = \alpha_{d{, s}} \int_A f\bigl(v(\sigma)\bigr)\mathrm{e}^{-\frac{s}{z^{d-1}}} \prod_{i=1}^d \mathrm{d}\harp{x_{i}} \frac{\ud\rho_{i}}{\rho_{i}^{2d-1}},
\end{equation}
where $z = z(\sigma)$ is the $d$th coordinate of $v(\sigma)$. 
Let $\Sigma \coloneqq \{ \sigma \in A \ST z(\sigma) = 1\}$.
Note that for all $t > 0$ and $\sigma \in  \Sigma$, we have $z(t \cdot \sigma) = t$.
There is a function $h \colon \Sigma \to \R_+$ such that
\[
{\II{A}} \prod_{i=1}^d \mathrm{d}\harp{x_{i}} \frac{\ud\rho_{i}}{\rho_{i}^{2d-1}}
=
h(\sigma) \,\ud \sigma \, \frac{\ud t}{t^{d^2}}
\]
for integration on $A = \Sigma \times \R_+$. 
{Writing $v(t \cdot \sigma) = \bigl(\harp x(\sigma), t\bigr)$, we obtain}

\[
\mathbb{E}\Biggl[\sum_{\substack{v \text{ a vertex}\\[1pt] \text{of}\;  \scell_d}}f(v) \Biggr]
=
\alpha_{d{, s}} \int_{\Sigma \times \R_+} f\bigl(v(t\cdot \sigma)\bigr) \ue^{-\frac{s}{t^{d-1}}} h(\sigma) \,\ud \sigma \,\frac{\ud t}{t^{d^2}}
=
{\alpha_{d, s} \int_{\uhs_d} f(\harp x, t) \ue^{-\frac{s}{t^{d-1}}} h'(\harp x) \,\ud\harp x \,\frac{\ud t}{t^{d^2}}}
\]
{for another function $h'\colon \R^{d-1} \to \R_+$ arising from}
 changing coordinates for $\sigma {\in \Sigma}$ in terms of $\harp{x}$ and other variables---which we integrate out. 
\Qed

{
\begin{remark}[\textsc{Separations from the ideal nucleus at $\infty$}]\label{r.sep-from-infty}
{T}he separations of vertices on the hypersurface $\partial \scell_d$ to the ideal nucleus at $\infty$ are $\mob_d$-invariant and have
law $\text{\rm Gamma}(d+1)${.} This is because the separation of $(\harp x, z)$ from $(\infty, {1})$ is ${1}/z^{d-1}$ by \cref{lem:evaldist}, and \cref{prop.vertint} showed that the law of ${1}/Z^{d-1}$ is $\mathrm{Gamma}(d+1)${.}
Thus, the distribution of separations is the same as that of the typical ideal Voronoi vertex, which we found in \cref{p.typ-sep}, but we have no explanation for the coincidence.
\end{remark}
}

\section{Ideal Poisson--Voronoi tessellations on regular trees} \label{sec.trees}
In this section, we apply our abstract results of \cref{sec:conv} to the $k$-regular tree ($k \ge 3$), denoted by $ \mathbb{T}_{k}$ with origin vertex $\origin$. We regard $ \mathbb{T}_{k}$ as a real tree by identifying each edge with a unit-length real segment, 
so that it carries a natural length measure $\mu$ induced by Lebesgue measure on its edges, and a geodesic distance $ \mathrm{d}_{ \mathbb{T}_{k}}$. As described in the introduction, one can then consider a Poisson process of points $ \mathbf{X}^{(\lambda)} = (X_{1}^{(\lambda)}, \ldots )$ which are ranked according to their increasing distances to $\origin$. In contrast to the case of hyperbolic spaces, here the asymptotic law of delays will depend upon the fractional part of $\log_{k-1}(\lambda)$ as $ \lambda \to 0$. More precisely, for $\lambda  \in (0,1) $ define 
$$ \ell_{\lambda} \coloneqq  - \lfloor \log_{k-1} (k \lambda) \rfloor, \quad \mbox{ and put   }\quad \xi_{\lambda} \coloneqq  (k-1)^{\ell_{\lambda}} \cdot k \cdot \lambda \in [1,k-1).$$
Let us introduce the proto-delay process
$$  \mathbb{D}_{i}^{(\lambda)} = \mathrm{d}_{ \mathbb{T}_{k}}( \origin, X_{i}^{(\lambda)}) - \ell_{\lambda}.$$
\begin{prop}[\textsc{Delays on trees}] \label{p.delay-conv-trees}
As $\lambda \to 0$ with \emph{$\xi_{\lambda}  = \xi \in [1,k-1)$ fixed}, the proto-delay process $( \mathbb{D}_{i}^{(\lambda)} \ST i \geq 1)$ converges in law towards a Poisson process on $ \mathbb{R}$ with intensity 
$$ \xi \cdot (k-1)^{m}  \quad \mbox{over the interval} \quad [m, m+1) \quad \mbox{for } m \in \mathbb{Z}.$$
\end{prop}
\begin{proof} For fixed $\lambda >0$, notice that the total intensity of the points falling in edges at distance $p\geq 0$ (i.e., whose closest point is at distance $p$) from the origin is equal to $ k \lambda (k-1)^{p}$. Write then $ p= \ell_{\lambda} + m$ to see that for $m$ fixed, as $\lambda \to 0$ with $\xi_{\lambda}$ fixed, this intensity converges to  $\xi (k-1)^m$, which concludes the proof. \end{proof}

The convergence to points on the boundary is trivial in this case. Recall that the Gromov boundary $ \partial \mathbb{T}_{k}$ of  the $k$-regular tree can be identified with the space of all infinite rays starting from the origin equipped with the natural local topology. It has a natural uniform measure. Given this, and the obvious fact that conditionally on their distances to the origin, the points of $ \mathbf{X}^{(\lambda)}$ are i.i.d.\ on the spheres prescribed by their distances, it follows that they converge towards i.i.d.\  uniform points on $ \partial \mathbb{T}_{k}$.  As in \cref{s.IPVT-stronger}, one can check that the ideal diagrams are a.s.\ nondegenerate (since the delays are a.s.\ distinct), and  we deduce the convergence of the Voronoi tessellations when $\lambda \to 0$ with $\xi_{\lambda}$ fixed. In other words, we have a one-parameter family of ideal tessellations $ \mathcal{I}_{\xi}$ on $ \mathbb{T}_{k}$ parametrized by $\xi \in [1,k-1)$ obtained as limit of Poisson--Voronoi tessellations on $ \mathbb{T}_{k}$. Although those ideal tessellations are indeed pairwise different (see \cref{p.different} below), here is a surprising fact:
\begin{theorem}[\cite{bhupatiraju}]
\label{t.treeuniq}
The restriction of $ \mathcal{I}_{\xi}$ to the vertices of\/ $ \mathbb{T}_k$ has the same law for all $\xi \in [1,k-1)$.
\end{theorem}

To be more precise, \cite{bhupatiraju} did not work with our model, but, rather, with the discrete Bernoulli--Voronoi tessellations solely on the vertices of $\TT_k$, where each vertex is independently a nucleus with probability $p \in (0, 1)$. In addition, to break ties in distance, each nucleus is given an independent, uniform $[0, 1]$ random label, so that a vertex belongs to the closest nucleus with the smallest label. To see the relationship between these two processes in the limit, let $E(\TT_k)$ be the union of real segments corresponding to the edges and $V(\TT_k)$ be the vertices. Define $f\colon E(\TT_k) \to V(\TT_k) \times [0, 1]$ by $f(x) \coloneqq (v, \ell)$, where $v$ is the endpoint closest to $\origin$ of the edge containing $x$ and $\ell$ is the distance from $x$ to $v$. Let $\YY^{(\lambda)}$ be obtained from $\bigl\{ f(x) \ST x \in \XX^{(\lambda)} \bigr\}$ by keeping only those $f(x)$ with smallest second coordinate when there is more than one pair having the same first coordinate. The first coordinates of $\YY^{(\lambda)}$ restricted to $V(\TT_k) \setminus \{\origin\}$ form a Bernoulli($p$) process with $p \coloneqq 1 - \ue^{-\lambda}$, while the second coordinates are i.i.d., continuous random variables. Fix a finite-radius ball $B$ of $\TT_k$ about $\origin$. Provided no vertices of $\YY^{(\lambda)}$ lie in $V(B)$, the Voronoi partition on $V(B)$ induced by $\XX^{(\lambda)}$ equals the Voronoi partition on $V(B)$ induced by $\YY^{(\lambda)}$. When $\lambda$ and thus $p$ are small, this proviso holds with high probability. Hence, if the limit exists of low-intensity Bernoulli--Voronoi tessellations on $\TT_k$, then so does the limit of the restriction to $V(\TT_k)$ of the Poisson--Voronoi tessellations. 

\cite{bhupatiraju} proved the existence of the limit of low-intensity Bernoulli--Voronoi tessellations on regular trees by explicitly calculating the probabilities of all elementary cylinder events, showing that they are polynomials in $p$ (with rational coefficients). For example, \cite[Lemma 2.5]{bhupatiraju} shows that the degree of the root equals $j \in [1, k]$ with limiting probability 
\[
\frac1{(k-2)(j-1) + 1}\cdot\frac1{\prod_{i=j}^{k-1} (1+\frac1{i(k-2)})}. 
\]
It is unclear whether there are any Cayley graphs other than trees where there is a unique low-intensity limit of Bernoulli--Voronoi tessellations, except when the limit is trivial; see \cite{bhupatiraju}.

The following proposition is the analogue of the corresponding result, \cref{p.tailEnd}, for $\HH_d$.

\begin{prop}[\textsc{One end on trees}]\label{p.OneEndTrees}
For each $\xi \in [1, k-1)$, a.s.\ no cell in the IPVT $\mathcal I_\xi$ on $\TT_k$ contains a biinfinite geodesic.
\end{prop}

\rproof
Let the delays be $(\DD_i \ST i \ge 1)$ in increasing order corresponding to the ideal boundary points ${\Theta}_i$.  By \cref{p.delay-conv-trees}, we have $\PP\bigl[\DD_1 < -n\bigr] = O\bigl((k-1)^{-n}\bigr)$ as $n \to\infty$ and $\PP\bigl[\DD_1 \ge n\bigr] \le \ue^{- \xi (k-1)^{n-1}}$. For any vertex $x$, let $\DD_i^x \coloneqq \ud_{{\Theta}_i}(x) + \DD_i$ be the separation of $({\Theta}_i, \DD_i)$ from $x$. Then $\min_i \DD_i^x$ has the same distribution as $\DD_1$ by automorphism invariance. Let $S'_{2n}$ be the set of vertices $x$ such that the geodesic from $x$ to ${\Theta}_1$ contains the geodesic from $\origin$ to ${\Theta}_1$ as well as precisely $2n$ additional edges. For $x \in S'_{2n}$, we have $\DD_1^x = \DD_1 + 2n$. Also, any such $x$ belongs to the cell of $\origin$ iff $\DD_1^x = \min_i \DD_i^x$. Let $A_n$ be the event that some $x \in S'_{2n}$ belongs to the cell of $\origin$. Let $S_{2n}$ be the sphere of radius $2n$ about $\origin$, which has cardinality $k(k-1)^{2n-1}$.  Then
\eqaln
{\PP(A_n) 
&\le
\PP\bigl[\DD_1 < -n\bigr] + 
\PP\bigl[\exists x \in S'_{2n}\enspace \DD_1 \ge -n,\; \DD_1^x < \DD_i^x \text{ for all } i > 1\bigr] 
\\ &\le
\PP\bigl[\DD_1 < -n\bigr] + 
\PP\bigl[\exists x \in S'_{2n}\enspace \min_i \DD_i^x \ge n \bigr] 
\\ &\le
\PP\bigl[\DD_1 < -n\bigr] + 
\PP\bigl[\exists x \in S_{2n}\enspace \min_i \DD_i^x \ge n \bigr] 
\\ &\leq 
O\bigl((k-1)^{-n}\bigr) + k(k-1)^{2n-1} \ue^{- \xi (k-1)^{n-1}}
\to 0 \quad \mbox{as } n \to\infty,
}%
where the last inequality follows from the second sentence of the proof.

It follows that $\PP\bigl(\bigcap_n A_n\bigr) = 0$, in other words, $\origin$ does not belong to a biinfinite geodesic in the cell of $\origin$ a.s. Since $\mathcal I_\xi$ is invariant under all automorphisms of $\TT_k$, the same holds for every vertex in place of $\origin$, whence all cells have only one end a.s.
\Qed

We use this result to prove our assertion that $ \mathcal{I}_{\xi}$ is different for different $\xi \in [1, k-1)$:

\begin{prop}\label{p.different}
There is a measurable function $f$ on tessellations of\/ $\TT_k$ such that for each $\xi \in [1, k-1)$, we have $f(\mathcal I_\xi) = \xi$ a.s. 
\end{prop}

\rproof
Let $V$ be a tessellation.  Define $f$ to be 0 if some cell of $V$ does not have a unique end. If each cell has a unique end, then let $\theta_1$ be the end of the cell of $\origin$. Set $g(\theta_1) \coloneqq 0$. There is a unique extension of $g$ to the set $E$ of ends of all cells such that when two cells with ends $\theta$ and $\theta'$ share a boundary point $x$, we have $g(\theta) + \ud_\theta(x) = g(\theta') + \ud_{\theta'}(x)$. When $V$ is the tessellation corresponding to ideal nuclei $({\Theta}_i, \DD_i)$, we have that $g({\Theta}_i) = \DD_i - \DD_1$ a.s. Define $G(t) \coloneqq |\{\theta \in E \ST g(\theta) \le t\}|$ for $t \ge 0$. Finally, let
\[
f(V)
\coloneqq
\limsup_{t \to\infty} \int_0^1 \frac{G(t+s)}{(k-1)^{t+s}} \,\ud s\Bigm/\int_0^1 \frac{s+1/(k-2)}{(k-1)^s} \,\ud s.
\]
In order to show that $f(\mathcal I_\xi) = \xi$ a.s., consider the Poisson process on $\R$ described in \cref{p.delay-conv-trees}, and denote by $N(t)$ the number of its points at most $t$ for $t \in \R$. Now $\lim_{t \to\infty} N(t)/\EE\bigl[N(t)\bigr] = 1$ a.s.\ with
\[
\EE\bigl[N(t) \bigr]
=
\xi \Bigl (\sum_{m < \flr{t}} (k-1)^m + (k-1)^{\flr{t}} \bigl(t - \flr{t}\bigr) \Bigr)
=
\xi \Bigl ((k-1)^{\flr{t}} \bigl(t - \flr{t} + \frac1{k-2}\bigr) \Bigr),
\]
whence  
\[
\frac{N(t)}{(k-1)^t}
\sim
\xi \frac{t - \flr{t} + 1/(k-2)}{(k-1)^{t-\flr t}} \mbox{ a.s.\ as } t \to\infty.
\]
At every sample where this holds, we obtain
\[
\lim_{t \to\infty} \int_0^1 \frac{N(t+s)}{(k-1)^{t+s}} \,\ud s = \xi \int_0^1 \frac{s+1/(k-2)}{(k-1)^s} \,\ud s.
\]
On the other hand, $G(t) = N(t+\DD_1)$ when $V = \mathcal I_\xi$, which proves that $f(\mathcal I_\xi) = \xi$ a.s.
\Qed

Similarly to the last part of the proof of \cref{t.factor}, one can extend this argument to prove that for each $\xi \in [1, k-1)$, there is an isomorphism between the Poisson process on $\R$ described in \cref{p.delay-conv-trees} and the IPVT on $\TT_k$ that is equivariant with respect to all automorphisms of $\TT_k$.
 
\section{Future directions}\label{s.open}

{
Many interesting questions remain, of which we present a few. We have not given substantial thought to all of them.


\begin{question}[\textsc{Other manifolds}]
The recent paper \cite{FMW} uses ideal Poisson--Voronoi diagrams defined directly on a generalization of the corona, similarly to \cref{d.IVT}, to establish some results in geometric group theory.\footnote{Their work was independent of ours. They were inspired by \cite{BudzinskiCurienPetri} and were unaware of \cite{bhupatiraju}.} They do not establish whether such diagrams are limits of Poisson--Voronoi diagrams. 
For which homogeneous Riemannian manifolds does the limit exist of Poisson--Voronoi diagrams as the intensity of the process of nuclei tends to 0? For example, does it exist on the Riemannian product $\HH_2 \times \R$?
\end{question}


\begin{question}[\textsc{Equivariant injectivity}]
\cref{t.factor} shows that certain equivariant maps are injective a.s.
Is the map that sends an IPVT to the set of its vertices a.s.\ injective?
Is the map that sends an IPVT to its IPDT a.s.\ injective? 
\end{question}

\begin{question}[\textsc{Typical neighboring vertices}]
What is the joint distribution of the neighboring vertices of a typical vertex?
\end{question}

\begin{question}[\textsc{Face intensities and hole probabilities}]
What are the values for the measure face intensities, $\tilde\Intn_{d, k}$ for $1 \le k \le d-2$? Related to this: what are the hole probabilities for $k$-faces with $0 \le k \le d-2$?
\end{question}


\begin{conj}[\textsc{Boundary hole probabilities}]
Regarding \cref{prop.holeprob}, we believe that $\mathbb{P}[B_u(\origin) \subset \mathcal{C}_d]\ue^{(d-1)u}$ is decreasing in $u$. In fact, for odd $d$, we believe that $\mathbb{P}[B_u(\origin) \subset \mathcal{C}_d]$ equals $\ue^{-(d-1)u}/p(\ue^{2u})$, where $p$ is a polynomial of degree $(d-1)/2$ with positive, rational coefficients. For even $d$, we believe that $\mathbb{P}[B_u(\origin) \subset \mathcal{C}_d]$ equals $\ue^{-(d-1)u}/f(u)$, where $f$ is an infinite power series with positive coefficients that are rational linear combinations of $1$ and $1/\pi$. 
The median distance from $\origin$ to $\partial \cells_d$ is likely less than the mean for every $d$, because the density of the distance is likely decreasing. 
\end{conj}


\begin{question}[\textsc{Typical values and the typical cell}]
We noted that $\R^{d-1}$-typical values for the $\mob_d$-typical ideal cell $\scell_d$ match $\mob_d$-typical values for $\cells_d$ in \cref{r.f-vector,r.facet-inten,r.Del-vols,r.sep-from-infty}. Is there a general theorem to this effect? 
\end{question}

\begin{question}[\textsc{Variance of the volume of the typical cell}]
Let $V_d(\lambda)$ be the variance of the volume of the typical cell of $\cellsl_d$. Does $\lambda V_d(\lambda)$ have a finite limit as $\lambda \to 0$? In fact, is the limit equal to the expected volume of $\ocell_d \setminus \scell_d(R_1)$ in their natural coupling? If so, what is this value? This is plausible because $\lambda V_d(\lambda)$ is the difference between the expected volumes of the size-biased typical cell (i.e., the zero cell) and the typical cell.
\end{question}

\begin{question}[\textsc{Indistinguishability}]
All cells of $\cells_d$ share the same asymptotics, which are those of the typical cell. Are they, in fact, indistinguishable? This would mean that for every measurable set $A$ of pairs $(C, V)$, where $C$ is a cell of a tessellation $V$ of $\HH_d$, if $A$ is invariant under the diagonal action of $\mob_d$, then a.s., for all cells $C \in \cells_d$, we have $(C, \cells_d) \in A$ or, for all cells $C \in \cells_d$, we have $(C, \cells_d) \notin A$. Since the IPVT is $\mob_d$-ergodic, this is equivalent to $\PP\bigl[(\ocell_d, \cells_d) \in A\bigr] \in \{0, 1\}$. Because of \cref{t.factor}, indistinguishability is equivalent to the nonexistence of a proper factor of the PPP $\XX$ on the corona, in other words, the nonexistence of an equivariant map $f$ on discrete subsets of $\corona$ such that $\varnothing \ne f(\XX) \subsetneq \XX$ a.s. The concept of indistinguishability is important in percolation theory \cite{LySch} and in measured group theory \cite{GaLy}.
\end{question}

Update: Sam Mellick \cite{mellick2024indist} has proved indistinguishability after seeing our question. 

\begin{question}[\textsc{Exceptional rays}]
Consider rays in the 1-skeleton of $\cells_2$, i.e., in the 3-regular tree embedded in $\HH_2$. Those rays that eventually stay in the boundary of some cell satisfy common strong laws of large numbers due to the ergodicity of the typical cell, $\scell_2$. What about other rays? For example, if $s(v)$ denotes the separation of a vertex $v$ from its closest ideal nucleus and if $(v_n \sut n \ge 1)$ is a ray of vertices that is eventually in the boundary of some cell, then $\lim_{n \to\infty} n^{-1} \sum_{k=1}^n s(v_k) = 3$ a.s.\ by \cref{p.typ-sep} and the ergodic theorem.  Those that do not satisfy this property are exceptional. By ergodicity of $\cells_2$, the existence of exceptional rays has probability 0 or 1 and the set of exceptional rays has an a.s.\ constant Hausdorff dimension (see, e.g., \cite[Section 1.8]{LP:book} for Hausdorff dimension in trees). What is that constant?
\end{question}

}
\medskip

\begin{center}
\rule{.5\textwidth}{.5pt}
\end{center}

\begin{quote}

May we not call them the Ghosts of departed Quantities?

\smallskip
--- \textit{The Analyst}, George Berkeley
\end{quote}

\bibliographystyle{alpha}

\bibliography{biblio}

\end{document}